\documentclass{amsart}
\usepackage{etex}
\usepackage{amssymb,amsmath,amsthm,mathtools,extpfeil}
\usepackage[all,cmtip]{xy}
\usepackage{tikz}
\usetikzlibrary{decorations.markings,matrix,arrows}
\usepackage{pstricks,rotating,graphicx}
\usepackage{fullpage,cite,hyperref}
\newtheorem{thm}{Theorem}[section]
\newtheorem{lem}[thm]{Lemma}
\newtheorem{thmA}{Theorem}

\newtheorem{prop}[thm]{Proposition}
\newtheorem{cor}[thm]{Corollary}
\newtheorem*{conj}{Conjecture}
\newtheorem*{thm*}{Theorem}
\theoremstyle{definition}
\newtheorem*{defn}{Definition}

\newtheorem{ex}[thm]{Example}
\newtheorem{exs}[thm]{Examples}

\numberwithin{equation}{section}

\DeclareMathOperator{\size}{size}
\DeclareMathOperator{\dist}{dist}
\DeclareMathOperator{\Homeo}{Homeo}
\DeclareMathOperator{\Ext}{Ext}
\DeclareMathOperator{\sign}{sign}
\DeclareMathOperator{\eu}{eu}
\DeclareMathOperator{\Tor}{Tor}
\DeclareMathOperator{\Jac}{Jac}
\DeclareMathOperator{\lcm}{lcm}
\DeclareMathOperator{\id}{id}
\DeclareMathOperator{\Hom}{Hom}
\DeclareMathOperator{\Lip}{Lip}

\DeclareMathOperator{\vol}{vol}
\DeclareMathOperator{\img}{im}

\DeclareMathOperator{\mass}{mass}
\DeclareMathOperator{\cel}{cell}

\DeclareMathOperator{\FV}{FV}

\DeclareMathOperator{\LD}{L\delta}
\DeclareMathOperator{\VD}{V\delta}

\DeclareMathOperator{\FVol}{FVol}
\DeclareMathOperator{\Fill}{Fill}
\newcommand{\ph}{\varphi}
\newcommand{\epsi}{\varepsilon}
\newcommand{\dmd}{\diamondsuit}
\newcommand{\lpair}{\langle\langle}
\newcommand{\rpair}{\rangle\rangle}

\begin{document}
\title{Volume distortion in homotopy groups}
\author{Fedor Manin}
\begin{abstract}
  Given a finite metric CW complex $X$ and an element $\alpha \in \pi_n(X)$,
  what are the properties of a geometrically optimal representative of $\alpha$?
  We study the optimal volume of $k\alpha$ as a function of $k$.
  Asymptotically, this function, whose inverse, for reasons of tradition, we
  call the volume distortion, turns out to be an invariant with respect to the
  rational homotopy of $X$.  We provide a number of examples and techniques for
  studying this invariant, with a special focus on spaces with few rational
  homotopy groups.  Our main theorem characterizes those $X$ in which all
  non-torsion homotopy classes are undistorted, that is, their distortion
  functions are linear.
\end{abstract}
\maketitle
The object of quantitative topology is to make more concrete various notions
coming from the existential results of algebraic topology.  Thus, while
classical rational homotopy theory gives an exhaustive family of algebraic
correlates to rational homotopy classes of simply-connected spaces, a
quantitative homotopy theory seeks to give geometric examples or descriptions
linked to the algebraic properties of these objects.

The term ``quantitative homotopy theory'' seems to have first been used by
Gromov in the conference paper \cite{GrQHT} and in Chapter 7 of the
near-simultaneous book \cite{GrMS}, although the ideas date back as far as
\cite{GrDil}.  Construed broadly, however, this program fits into a tradition of
extracting metric information from topological data which includes problems from
systolic geometry, geometric group theory, and other areas.  In particular,
geometric group theorists, as specialists in fundamental groups, have explored a
host of asymptotic invariants whose higher-dimensional analogues may also be of
interest.  These include the Dehn function, growth of groups, and distortion of
group elements and subgroups.

Higher-dimensional isoperimetric functions of groups, that is, of their
Eilenberg--MacLane spaces, have been studied in some detail, notably by Gromov
\cite{GrAsym}, Alonso--Wang--Pride \cite{AWP}, Brady--Bridson--Forester--Shankar
\cite{BBFS}, and Young \cite{Young}.  A common theme of this body of literature
is the plurality of possible definitions, many of which are equivalent in the
one-dimensional case.  The subject of growth of higher homotopy groups was
broached in chapters 2 and 7 of \cite{GrMS} with a number of examples and a
conjecture for simply-connected spaces.

Here we analyze a higher-dimensional analogue of distortion.  Heuristically, the
\emph{distortion} of a group element $\alpha \in G$ is given by
$$\delta_\alpha(k)=\max\{m \mid \size(\alpha^m) \leq k\}.$$
If $G$ is the fundamental group of a space, word length is a natural measure of
size.  On the other hand, if $G=\pi_n(X)$ for a space $X$, we can choose
inequivalent measures of size by taking advantage of different features of a
metric structure on $X$, leading once again to a plurality of definitions.  For
example, suppose that $X$ is a CW complex with a piecewise Riemannian metric.
Then we can choose to minimize the Lipschitz constant of a representative, its
volume, or more generally the $m$-dilation for some $1 \leq m \leq n$, that is,
how much the map $f:S^n \to X$ distorts $m$-dimensional tangent subspaces.
Moreover, the asymptotics of each such function are preserved by Lipschitz
homeomorphisms.  This means that, as long as $X$ is compact, each of these
definitions gives a topological invariant.

In all of these situations, it is natural to consider an element
\emph{undistorted} if the best asymptotics are attained by composing $\alpha$
with a degree $k$ map $S^n \to S^n$.  Thus an element is Lipschitz undistorted
if its Lipschitz distortion is $\sim Ck^n$, and volume undistorted if its volume
distortion is $\sim Ck$.

In \cite{GrQHT}, Gromov states a conjecture about the Lipschitz distortion of
homotopy groups.
\begin{conj}[Gromov]
  A class $\alpha \in \pi_nX$ of a simply-connected finite metric CW complex $X$
  is Lipschitz undistorted if and only if $\alpha$ has nonzero image under the
  rational Hurewicz map.  If $\alpha$ is distorted, then the growth of its
  Lipschitz norm is polynomial with rational exponent, and at most
  $Ck^{\frac{1}{n+1}}$; in the terminology above, the Lipschitz distortion of
  such an $\alpha$ is at least $\sim Ck^{n+1}$.
\end{conj}
Gromov points out that Sullivan's model of rational homotopy can be used to
demonstrate that Lipschitz distortion in this situation is at worst polynomial.
Lower bounds on distortion can be obtained in various cases using constructions
such as generalized Whitehead products.  However, a full proof of the conjecture
remains elusive.  In this paper, we instead focus on volume distortion, which
has not been previously studied.

In contrast with the Lipschitz case, volume distortion is trivial for simply 
connected spaces: every homotopy class is either undistorted or has a multiple
with zero volume, as shown in Theorem \ref{thm:sc}.  But for
non-simply-connected finite complexes, it's an interesting invariant: we
construct examples in which the volume distortion of an element is $k^r$ for
various rational $r$ and where it is $\exp(\sqrt[n]{k})$.  One feature that
makes volume distortion convenient for non-simply-connected spaces is its
invariance up to rational homotopy.  More precisely:
\begin{thmA} \label{thmQ}
  Suppose two compact connected CW-complexes $X$ and $Y$ are
  \emph{rationally equivalent}, that is, there is a space $Z$ and maps $X \to
  Z \leftarrow Y$ which induce isomorphisms on $\pi_1$ and $\pi_n \otimes
  \mathbb{Q}$ for every $n$.  If $\alpha \in \pi_n(X)$ and $\beta \in \pi_n(Y)$
  have the same image in $\pi_n(Z)$, then they have asymptotically equivalent
  volume distortion functions.
\end{thmA}
In order to show this, we develop some results in the rational homotopy theory
of non-simply-connected spaces.  As discussed in Example \ref{ex:Lip}, the
equivalent statement in the Lipschitz setting appears to be false.

Our main theorem characterizes those spaces which have no distortion in their
homotopy groups.  It turns out, however, that a natural definition of volume
distortion in this context includes not only individual elements but
finite-dimensional subspaces.  While the definition extends easily, this is a
significantly different concept: in Example \ref{exs:oddD}(1), we see that a
subspace might be distorted even if no element is.  This phenomenon can be
thought of as distortion in an irrational direction in $\pi_n(X) \otimes
\mathbb{R}$, induced by an irrational eigenvector of the action of $\pi_1(X)$ on
$\pi_n(X)$.

In order to state the theorem precisely, we need two invariants which together
classify rational homotopy fibrations $(S^{2n+1})^r \to X \to B$ up to rational
equivalence.  The first of these is the monodromy representation $\rho:\pi_1X
\to GL(r_i,\mathbb{Q})$.  The second is the obstruction to extending a section
from the $(2n+1)$-skeleton to the $(2n+2)$-skeleton, a generalization of the
Euler class for sphere bundles.  This class, which is defined more thoroughly in
section 6, lies in the twisted cohomology group $H^{2n+2}(B;M_\rho)$, where
$M_\rho$ is the $\mathbb{Q}\pi_1B$-module corresponding to $\rho$.  In turn, the
universal covering $\tilde B \to B$ pulls this Euler class back to the
$L_\infty$ cohomology $H^{2n_i+2}_{(\infty)}(\tilde B;\mathbb{Q}^{r_i})$.  While the
definition of $L_\infty$ cohomology relies on a metric on $\tilde B$, it is
insensitive to small-scale geometry, and so only depends on the topology of $B$.
\begin{thmA} \label{thmA}
  A finite CW complex $X$ has no volume distortion in $\pi_n(X)$ for any $n \geq
  2$ if and only if the following conditions hold:
  \begin{enumerate}
  \item $X$ is rationally equivalent to the total space of a fibration over
    $B\pi_1X$ with fiber $\Pi_{i=1}^s (S^{2n_i+1})^{r_i}$, which further decomposes
    as a tower of fibrations
    $$X=X_m \to X_{m-1} \to \cdots \to X_0=B\pi_1X$$
    with fibers of the form $(S^{2n_i+1})^{r_i}$, $1 \leq i \leq m$;
  \item for each $i$, the monodromy representation $\rho:\pi_1X \to
    GL(r_i,\mathbb{Q})$ is elliptic, i.e.\ its image is contained in a conjugate
    of $O(r_i,\mathbb{R})$;
  \item and for each $i$, the Euler class $\eu \in H^{2n_i+2}(X_{i-1};M_\rho)$
    vanishes when considered in the $L_\infty$ cohomology
    $H^{2n_i+2}_{(\infty)}(\tilde X_{i-1};\mathbb{Q}^{r_i})$ of the universal cover.
  \end{enumerate}
\end{thmA}
One way of interpreting this result intuitively is as follows: spaces with no
distortion must have universal covers which are \emph{coarsely trivial}
fibrations, that is they are in some sense finite distance from being a metric
product $\prod_i S^{2n_i+1} \times \widetilde{B\pi_1X}$.  In certain cases, this
assertion may be interpreted literally---there is a bilipschitz map from
$\tilde X$ to this product space.  In general, however, the author has not been
able to turn this intuition into a compelling piece of mathematics.

We say volume distortion is \emph{infinite} if there is a finite-dimensional
subspace of $\pi_n(X) \otimes \mathbb{Q}$ in which arbitrarily large vectors
have representatives of bounded volume.  Our second complete characterization,
however, is of when a complex has \emph{weakly infinite} volume distortion.
Here we take the minimum volume functional on $\pi_n(X)$ to be a restriction of
that on $H_n(\tilde X)$; this is natural because for simply connected spaces,
least volume functionals are the same for homology and homotopy, at least for
$n \geq 3$ (see Lemma \ref{lem:BBFS}.)  We say that (a finite-dimensional
subspace of) $\pi_n(X) \otimes \mathbb{Q}$ is weakly infinitely distorted in (a
finite-dimensional subspace of) $H_n(\tilde X;\mathbb{Q})$ if arbitrarily large
vectors in the latter at a bounded distance from the former are represented by
integral chains of bounded volume.
\begin{thmA} \label{thmB}
  A finite CW complex $X$ has no weakly infinite volume distortion in $\pi_n(X)$
  for any $n \geq 2$ if and only if the following conditions hold:
  \begin{enumerate}
  \item $X$ is rationally equivalent to the total space of a fibration over
    $B\pi_1X$ with fiber $\Pi_{i=1}^s (S^{2n_i+1})^{r_i}$;
  \item the monodromy representation $\pi_1X \to GL(\pi_*(X) \otimes
    \mathbb{Q})$ is elliptic;
  \item and for every $n \geq 2$, the group $H_n(\tilde X;\mathbb{Q})$ splits as
    a $\mathbb{Q}\pi_1X$-module into the image of the Hurewicz map and its
    complement.
  \end{enumerate}
\end{thmA}
Which weakly infinitely distorted classes are in fact infinitely distorted
remains largely open, although it is certainly not all of them, as evidenced by
Theorem \ref{thm:infcyc}.

\subsection*{Examples and methods}
In this section, we give an overview of the proof of Theorem \ref{thmA}.  In
essence, each of the three conditions is aimed at eliminating a particular
potential source of volume distortion; we give examples of each of these three
types in this section.  The proof works by showing that if a space has no volume
distortion of these three types, then it must have no volume distortion at all.
Note that this does not preclude the existence of other, yet undiscovered
sources of volume distortion in spaces that do have volume distortion.

Certain homotopy classes in $\pi_n(X)$ have representatives which retract to the
$(n-1)$-skeleton of the space $X$; this is the simplest possibility, since all
such maps have zero volume.  This case is exemplified by the Hopf map $S^3 \to
S^2$; indeed, in a simply-connected space, all homotopy classes which are zero
homologically are distorted in this way.  Thus to eliminate this effect, our
space $X$ must have the property that the Hurewicz homomorphism $\pi_*(X)
\otimes \mathbb{Q} \to H_*(\tilde X;\mathbb{Q})$ on its universal cover is
injective.  This is equivalent to condition (1).

Distortion may also be caused by the action of the fundamental group on
$\pi_n(X)$.  Thus, suppose $X$ is the mapping torus of a degree 2 map $f:S^n \to
S^n$.  Then for every $k$, $2^k$ times the generator $[\id_{S^n}] \in \pi_n(X)$
has a representative with volume 1, to be thought of as a balloon on a string of
length $k$.  The key observation here is that this is because the action of
$\pi_1(X)$ on $\pi_n(X) \otimes \mathbb{Q}$ sends the generator to an unbounded
set.  Condition (2) serves to eliminate sources of distortion of this type.

Finally and most subtly, distortion can be induced by the Euler class of a
sphere bundle.  For example, suppose that $S^3 \to X \xrightarrow{p} T^4$ is a
bundle with Euler class $[T^4]$, so that the obstruction to extending a section
$s:(T^4)^{(3)} \to X$ to the 4-cell is the generator $\alpha \in \pi_3(X)$.  In
other words, the attaching map of the 4-cell lifts along this section to a map
homotopic to the inclusion of the fiber $S^3 \hookrightarrow X$.  From the
perspective of the universal cover $\tilde X \to \mathbb{R}^4$, this gives us a
Lipschitz section of the 3-dimensional grid which does not extend to all of
$\mathbb{R}^4$.  Consider the map $f_k:S^3 \to \tilde X$ which is the lift along
this section of the boundary of $[0,k]^4$.  This map has volume $O(k^3)$, but
is homotopic to the sum of $k^4$ boundaries of cubes of side length 1, each of
which is in turn homotopic to the inclusion of the fiber.  In other words,
$[f_k]=k^4\alpha$, and so the volume distortion of $\alpha$ is at least
$Ck^{4/3}$.

Indeed, one can see that this is the \emph{only} possible source of distortion
in this space.  Suppose now that $f:S^n \to \tilde X$ is a map of volume $V$.
At the expense of a multiplicative constant increase in volume, we can assume
that its image is in $\tilde s^{-1}\left(\widetilde{(T^4)^{(3)}}\right)$.  It
turns out that $f$ differs from $s \circ p \circ f$ by an amount of twisting
around the fiber which must be linear in $V$.  Thus the element $[f] \in
\pi_3(X)$ has the form
$$[f]=\Fill(p \circ f)+O(V),$$
where $\Fill(p \circ f)$ is the volume of a 4-chain in $\mathbb{R}^4$ filling
$p \circ f$.  The isoperimetric inequality in $\mathbb{R}^4$ leads us to
conclude that $[f]=O(V^{4/3})$.

This key geometric idea can easily be extended to other bundles of the form $S^n
\to X \to M^{n+1}$.  For example, if $M$ is hyperbolic, the linearity of its
isoperimetric inequality means that there is no distortion in $X$.  Condition
(3) of Theorem \ref{thmA}, which turns out to be equivalent to the linearity of
a certain isoperimetric inequality, is aimed at eliminating this source of
distortion.  To prove the theorem in full generality one needs not only to prove
this correspondence, but also to provide an argument of the type we gave for the
bundle over $T^4$---that distortion in the total space of a rational homotopy
fibration corresponds to an isoperimetric function in the base---for a larger
class of \emph{delicate} spaces, that is, those that satisfy conditions (1) and
(2).  However, because fibrations are hard to construct in the world of finite
complexes, this argument relies instead on cofibrations with the same
homotopy-theoretic properties.

\subsection*{Results in specific dimensions}
Theorems \ref{thmA} and \ref{thmB} lay out conditions that are necessary and
sufficient for $\pi_n(X)$ to be undistorted and not weakly infinitely distorted,
respectively, for every $n$.  What can we say about whether a specific
$\pi_n(X)$ is distorted?

Here is an alternate sketch of the proof of Theorem \ref{thmA} which highlights
those aspects which are most relevant for this question.  Let $X$ be a finite
complex.  If for some $n$, $\pi_n(X) \otimes \mathbb{Q}$ is an
infinite-dimensional vector space, then, as we will show in Corollary
\ref{cor:fin}, there is a distorted element in $\pi_N(X)$ for some $N$; but
beyond this, such a situation is hard to analyze.  We can say much more when
$\pi_n(X) \otimes \mathbb{Q}$ is finite-dimensional.  Indeed, more generally, we
will give a criterion for whether a finite-dimensional
$\mathbb{Q}\pi_1(X)$-submodule $V \subseteq \pi_n(X) \otimes \mathbb{Q}$, that
is, a vector subspace invariant under monodromy, is distorted.

Given such a submodule $V$, we can add a finite number of $(n+1)$-cells to get
an inclusion $p:X \to B$ with $V=\ker(p_*:\pi_n(X) \to \pi_n(B)) \otimes
\mathbb{Q}$.  In addition, $p$ is rationally $n$-connected; for our purposes,
this makes $p$ close enough to one step in a Postnikov tower and we call the
resulting system an \emph{almost Postnikov pair}.  In fact, in this situation,
if $\rho:\pi_1(X) \to GL(V)$ is the monodromy representation on $V$, the Euler
class $\eu \in H^{n+1}(B;M_\rho)$ simply sends all $(n+1)$-cells of $X$ to $0$
and each of the extra $(n+1)$-cells to the element of $V$ it kills.  We can now
state our criterion.
\begin{thmA} \label{thmD}
  Fix a finite complex $X$ and an $n \geq 3$, and let $V \subseteq \pi_n(X)
  \otimes \mathbb{Q}$ be a submodule which is finite-dimensional as a vector
  space.  Then $V$ has no volume distortion if and only if the following
  conditions hold:
  \begin{enumerate}
  \item the monodromy representation $\rho:\pi_1(X) \to GL(V)$ is elliptic;
  \item for some (or any) almost Postnikov pair $X \to B$ with
    $V=\ker(\pi_n(X) \to \pi_n(B)) \otimes \mathbb{Q}$, the Euler class
    $\eu \in H^{n+1}(B;M_\rho)$ vanishes when considered in the $L_\infty$
    cohomology $H^{n+1}_{(\infty)}(\tilde B;V)$ of the universal cover.
  \end{enumerate}
\end{thmA}
This is the closest we can easily get to isolating individual elements of
$\pi_n(X)$ and testing whether they are distorted.  It's hard to say anything
about the distortion of individual elements of $\pi_nX$ on which $\rho$ acts
nontrivially; even when $V$ is distorted, its elements may not be.  This
distinction is explored at some length in Section 3.

This theorem gives considerable information on top of Theorem \ref{thmA}.
Nevertheless, the condition that $V$ be finite-dimensional excludes many natural
situations.  For example, all higher homotopy groups of $S^1 \vee S^2$ are
infinite-dimensional with $\mathbb{Z}$ acting freely.  So suppose the
$\pi_1$-orbit of $\alpha \in \pi_n(X) \otimes \mathbb{Q}$ is
infinite-dimensional.  What can we say about the distortion of $\alpha$?  If the
Hurewicz homomorphism on the universal cover $\tilde X$ sends $\alpha \mapsto
0$, then $\alpha$ is certainly volume-distorted.  If instead its Hurewicz image
in $\tilde X$ is detected by an $L_\infty$ cocycle, then it is certainly not.
These conditions, however, are by no means exhaustive, and the latter can also
be difficult to determine, as Example \ref{ex:BS} demonstrates.
Infinite-dimensional homotopy groups, then, remain somewhat of a mystery.

\subsection*{Corollaries and other results}
The isoperimetric inequality mentioned above turns out, in the case of
aspherical spaces, to define a new kind of higher-dimensional filling function
for groups.  Any cohomology class $\omega \in H^{n+1}(X;\mathbb{Q})$ gives a
directed isoperimetric function $\FV_{X,\lvert\langle\omega,\cdot\rangle\rvert}^n$.
Groups with interesting directed isoperimetric functions then give rise to
spaces with interesting distortion functions.  For example, we construct a
sequence of groups $\dmd_n$, related to the Baumslag-Solitar group $BS(1,2)$,
which have finite classifying spaces and for which
$\FV_{\dmd_n,\lvert\langle\omega,\cdot\rangle\rvert}^n(k) \cong 2^{\sqrt[n]{k}}$ for
any nonzero $\omega \in H^{n+1}(\dmd_n;\mathbb{Q})$; indeed, for these groups,
this coincides with the usual higher-dimensional Dehn function.  This implies
that nontrivial fibrations $S^n \to X \to B\dmd_n$ have volume distortion of the
form $\exp(\sqrt[n]{k})$ as well.

This correspondence between distortion and filling functions provides more
general connections between this subject and geometric group theory.  Thus one
corollary of Theorem \ref{thmA} is that if $X$ is the unit tangent bundle of a
closed $n$-manifold which is aspherical or has non-amenable fundamental group,
then the class in $\pi_n(X)$ corresponding to the fiber is never distorted.
Conversely, for amenable groups such filling functions are never linear for any
nonzero cohomology class.

\subsection*{Related work and further directions}
Because many of our bounds rely on cellular maps, the results are inherently
asymptotic---we do not aspire to minimize volume of specified maps in specified
geometries.  Others, however, have worked in this direction.  \cite{Gluck} and
\cite{Wen} have found that well-known maps uniquely minimize the Lipschitz
constant of maps between round spheres and from products of spheres to spheres,
respectively.  Larry Guth in \cite{Guth2} and \cite{Guth3} provides bounds,
depending on the metric in the domain and range, on Lipschitz constants of
Hopf-like maps of spheres, as well as on their $k$-dilation for various $k$, a
more general measure of the size of a map that includes both volume and
Lipschitz constant.

Asking about asymptotic growth only makes sense for rational homotopy classes.
For torsion homotopy classes, however, one can also ask a homotopy invariant
question: are certain geometric quantities positive or are they zero?  In a
recent work in this vein \cite{Guth}, Guth addresses the minimal $k$-dilation of
certain torsion homotopy classes of spheres.

One can ask whether results similar to ours hold for Lipschitz distortion or
more generally for distortion with respect to $k$-dilation.  Since Lipschitz
distortion is never infinite, this precludes a result of the genre of Theorem
\ref{thmB}.  On the other hand, there is some promise for a result similar to
Theorem \ref{thmA}.  The reduction to delicate spaces works just as well for
Lipschitz distortion, as does a result similar to Theorem \ref{thm:sph},
relating distortion in a sphere bundle to a certain filling function in the base
space.  However, as far as we know this only holds for actual sphere bundles,
not up to rational homotopy.  Example \ref{ex:Lip} gives pause to attempts to
extend this approach to delicate spaces in general, and so the problem seems
less tractable.  Still, there are a number of rich phenomena related to
Lipschitz distortion which have yet to be explored.

\subsection*{Outline of the paper}
Here we describe the main methods and results of each section of the paper.

Section 1 concerns the development of a rational homotopy theory of
non-simply-connected CW complexes.  We show a rational version of Wall's result
on finiteness properties for CW complexes \cite{Wall} and several other
finiteness results which become useful later for applying results from algebraic
topology in the context of finite complexes.

In section 2, we discuss metric structures on compact spaces and give detailed
definitions for the various types of distortion.  We also prove Theorem
\ref{thmQ}.

In section 3, we discuss distortion in simply-connected spaces as well as
distortion from monodromy, and show that spaces with no infinite distortion must
satisfy conditions (1) and (2) of Theorems \ref{thmA} and \ref{thmB}.  We also
discuss the applicability of our methods to infinite-dimensional spaces of
finite type.

Section 4 reviews previous results surrounding higher-dimensional filling
functions of groups and spaces and introduces a new type of filling function.
We give examples in which this kind of filling function demonstrates various
behaviors, most notably the sequence $\dmd_n$ described above, whose
$n$-dimensional filling functions are asymptotically equivalent to
$2^{\sqrt[n]{k}}$.

In section 5, we show that the filling functions defined in section 4 are
equivalent to certain cohomological isoperimetric inequalities which may be
thought of as dual.  This turns out to generalize several known instances of
this type of duality.

Finally, in section 6 we tackle the class of delicate spaces.  This means first
defining almost Postnikov pairs and their Euler classes.  We then use the
results of sections 4 and 5 to relate distortion in one of the spaces in the
pair to a filling function in the other.  As a special case of this relationship
we complete the proofs of Theorems \ref{thmA} and \ref{thmD} and discuss various
examples and applications.

Section 7 provides a different, more algebraic look at delicate spaces, yielding
a proof of Theorem \ref{thmB}.

\subsection*{Acknowledgements}
I would like to thank my advisor, Shmuel Weinberger, for introducing me to this
area of research and guiding me throughout the work.  Many thanks also to Katie
Mann, Jenny Wilson, Chris Beck, Bena Tshishiku, Kevin Whyte, Daniel Studenmund,
and others for their support and fruitful conversations.  Special thanks to
Jeremy Rickard on MathOverflow for suggesting the counterexample \ref{rickard}
and to Zeb Brady for walking me through some basic algebraic number theory.
Many thanks also to the anonymous referee for a multitude of corrections and
suggestions.

This paper is based on a portion of the author's PhD thesis \cite{thesis}.

\setcounter{tocdepth}{1}
\tableofcontents

\section{Algebraic preliminaries}

We start this section by proving two simple propositions that will find their
uses in later sections, before moving on to a more self-contained discussion
developing the rational homotopy theory of finite non-simply-connected
complexes.  We will tacitly assume all spaces to be connected, and we will
frequently and tacitly make use of the natural $\mathbb{Q}\pi_1X$-module
structure on the group $\pi_n(X) \otimes \mathbb{Q}$.  As a matter of notation,
we write $h_n:\pi_n(X) \otimes \mathbb{Q} \to H_n(X;\mathbb{Q})$ to mean the
rational Hurewicz homomorphism.  We also write $X^{(k)}$ to refer to the
$k$-skeleton of a CW complex $X$.
\begin{prop} \label{prop:i-1}
  Suppose $X$ is a simply connected complex and $[\alpha] \in \pi_n(X)$ is zero
  under the Hurewicz map.  Then $[\alpha]$ has a representative $f:S^n \to
  X^{(n-1)}$.  More generally, if $h_n([\alpha])=0$, then there is some $k>0$ so
  that $k[\alpha]$ has a representative $f:S^n \to X^{(n-1)}$.
\end{prop}
\begin{proof}
  The relative Hurewicz theorem and exact sequences for pairs give us the
  diagram
  $$\xymatrix{
    \pi_n(X^{(n-1)}) \ar[r] \ar[d] & \pi_n(X) \ar[r] \ar[d] &
    \pi_n(X,X^{(n-1)}) \ar[d]^\wr \\
    0 \ar[r] & H_n(X) \ar@{^{(}->}[r] & H_n(X,X^{(n-1)}).
  }$$
  Thus the sequence
  $$\pi_n(X^{(n-1)}) \to \pi_n(X) \to H_n(X)$$
  is exact.  This gives us the integral result.  The rational version follows.
\end{proof}
\begin{prop} \label{lem:injpi}
  Suppose $(X,A)$ is a CW pair of finite $n$-complexes such that the inclusion
  $\iota:A \to X$ induces isomorphisms on $\pi_1$ and on $\pi_i \otimes
  \mathbb{Q}$ for $2 \leq i \leq n-1$.  Then the induced map $\iota_*:\pi_n(A)
  \otimes \mathbb{Q} \to \pi_n(X) \otimes \mathbb{Q}$ is injective.
\end{prop}
\begin{proof}
  For any simply connected CW complex $Y$, consider the Serre spectral sequence
  applied to the fibration
  $$K(\pi_k(Y),k) \to Y_{(k)} \to Y_{(k-1)}$$
  within the Postnikov tower of $Y$.  The entry $E^2_{k,0}=H_k(Y_{(k-1)})$ has
  all differentials going to zero groups and thus survives to the $E^\infty$
  page.  This implies that the map
  \begin{equation} \label{obs}
    H_k(Y;\mathbb{Q})=H_k(Y_{(k)};\mathbb{Q}) \twoheadrightarrow
    H_k(Y_{(k-1)};\mathbb{Q})
  \end{equation}
  is a surjection.

  By the rational relative Hurewicz theorem, $\iota_*:H_i(\tilde A;\mathbb{Q})
  \to H_i(\tilde X;\mathbb{Q})$ is also an isomorphism for $i \leq n-1$.
  Moreover, since $H_{n+1}(\tilde X,\tilde A)=0$, $\iota_*:H_n(\tilde A) \to
  H_n(\tilde X)$ is injective.  Consider now the Serre spectral sequence for the
  fibration
  $$K(\pi_n(X),n) \to \tilde X_{(n)} \to \tilde X_{(n-1)}.$$
  In the $E^2$ page, rows $1$ through $n-1$ are zero, which allows us to
  construct the Gysin-like exact sequence illustrated below.  Here, the boxed
  groups are assembled from the terms of the $E^\infty$ page of the spectral
  sequence, giving rise to the arrows going to and from them.
  $$
  \begin{tikzpicture}
    \matrix (m) [matrix of math nodes,
      nodes in empty cells,nodes={minimum width=5ex,
        minimum height=5ex}
    ]{
      \qquad\quad & & & & \\
      n      & \pi_n(X) \otimes \mathbb{Q} &     &     & \\
      n-1    & 0      &     &     & \\
      \vdots & \vdots &  &  & \\
      1      & 0      & \cdots & 0 & & \\
      0      & H_0(\tilde X_{(n-1)};\mathbb{Q}) & \cdots &
      H_n(\tilde X_{(n-1)};\mathbb{Q}) & H_{n+1}(\tilde X_{(n-1)};\mathbb{Q}) \\
      \qquad\quad\strut & 0 & \cdots & n & n+1 & \strut \\};
    \node[draw,rectangle] (Hnp) at (4,1.5) {$H_{n+1}(\tilde X_{(n)};\mathbb{Q})$};
    \node[draw,rectangle] (Hn) at (0,0.5) {$H_n(\tilde X_{(n)};\mathbb{Q})$};
    \draw[-stealth] (Hnp) to [bend left=30] (m-6-5);
    \draw[-stealth] (m-6-5) edge [bend right=40]
      node[anchor=south west,inner sep=1pt] {$\partial_{n+1}$} (m-2-2);
    \draw[-stealth] (m-2-2) to [bend left=15] (Hn);
    \draw[-stealth] (Hn) to [bend left=45] (m-6-4);
    \draw[-stealth] (m-6-4) edge [bend right=10]
      node[anchor=north east,inner sep=0pt] {$\partial_n$} (m-3-2);
    \draw[thick] (m-1-1.east) -- (m-7-1.east);
    \draw[thick] (m-7-1.north) -- (m-7-6.north);
  \end{tikzpicture}
  $$
  Moreover, the observation that the map \eqref{obs} is surjective for $k=n+1$
  implies that $H_{n+1}(\tilde X_{(n)};\mathbb{Q})=0$, reducing the sequence to
  four nonzero terms.

  The same analysis holds replacing $X$ with $A$.  The whole picture is
  functorial, and so we get a homomorphism of exact sequences
  $$
  \xymatrix{
    0 \ar[r] & H_{n+1}(\tilde A_{(n-1)};\mathbb{Q}) \ar[r]^-{\partial_{n+1}}
    \ar[d]^{\text{\rotatebox[origin=c]{-90}{$\cong$}}} &
    \pi_n(A) \otimes \mathbb{Q} \ar[r] \ar[d]^{\iota_*} &
    H_n(\tilde A_{(n)};\mathbb{Q}) \ar[r] \ar@{^{(}->}[d] &
    H_n(\tilde A_{(n-1)};\mathbb{Q}) \ar[r] \ar[d]^{\wr} & 0 \\
    0 \ar[r] & H_{n+1}(\tilde X_{(n-1)};\mathbb{Q}) \ar[r]^-{\partial_{n+1}} &
    \pi_n(X) \otimes \mathbb{Q} \ar[r] & H_n(\tilde X_{(n)};\mathbb{Q}) \ar[r] &
    H_n(\tilde X_{(n-1)};\mathbb{Q}) \ar[r] & 0.
  }
  $$
  Now a diagram chase shows that $\ker(\iota_*)=0$.
\end{proof}

\subsection*{Finiteness conditions and rational homotopy}
In his seminal paper \cite{Wall}, Wall proves the following theorem, stated here
in abridged form:
\begin{thm*}
  The following conditions on a CW complex $X$ are equivalent:
  \begin{enumerate}
  \item $X$ is homotopy equivalent to a complex with finite $n$-skeleton.
  \item The group $\Gamma:=\pi_1(X)$ is finitely presented, and for every $k
    \leq n$, the condition $F_k$ holds: for every finite complex $K^{k-1}$ and
    $(k-1)$-connected map $\ph:K \to X$, $\pi_k(\ph) \otimes \mathbb{Q}$ is a
    finitely generated $\mathbb{Z}\Gamma$-module.
  \end{enumerate}
\end{thm*}
Intuitively, one can understand this as follows.  The homotopy of finite
complexes, including their rational homotopy, can a priori be very complicated.
At the very least, there are many such spaces with complicated fundamental
groups $\Gamma$ for which higher homotopy groups are infinitely generated as
modules over $\mathbb{Z}\Gamma$.  This complexity, however, is not mere anarchy,
but in some sense is dictated by the fundamental group, give or take a few
cells, i.e.~a finite number of extra generators or relations.

In this section, we develop a similar theory in a rational setting.  Moreover,
we show that the rational isomorphisms produced are effective, in the sense that
the torsion in the difference has bounded exponent, despite perhaps being vastly
infinitely generated.  Our main technical tools are a generalized Hurewicz
theorem introduced by Serre and an invertibility result for rational
equivalences on the level of chain homotopy.

First, we define the appropriate notions of equivalence.
\begin{defn}
  For two CW complexes $X$ and $Y$, we say a map $f:X \to Y$ is
  \emph{rationally $n$-connected} if it induces an isomorphism on $\pi_1$ and
  $\pi_k(f) \otimes \mathbb{Q}:=\pi_k(M_f,X) \otimes \mathbb{Q}=0$, where $M_f$
  is the mapping cylinder, for $2 \leq k \leq n$.  If this is true for every
  $k$, we say $f:X \to Y$ is a \emph{rational equivalence}.

  We say that $X$ is \emph{rationally equivalent} to $Y$ ($X \simeq_{\mathbb{Q}}
  Y$) if there is a CW complex $Z$ with rational equivalences $X \to Z$ and $Y
  \to Z$.  $X$ is \emph{rationally $n$-equivalent} to $Y$ if there is a CW
  complex $Z$ and maps $X \to Z$ and $Y \to Z$ which induce isomorphisms on
  $\pi_1$ and $\pi_k \otimes \mathbb{Q}$ for $2 \leq k \leq n$.
\end{defn}
Note that a map inducing a rational $n$-equivalence is rationally $n$-connected,
and a rationally $n$-connected map induces a rational $(n-1)$-equivalence, but
the converses of these statements are false.

It is not obvious that rational equivalence and $n$-equivalence as we have
defined them are in fact equivalence relations.  In other words, we have yet to
show that any zigzagging sequence of equivalences
$$
\begin{tikzpicture}
  \node (a) at (0,0) {$X$};
  \node (b) at (0.9,-0.75) {$Z_1$};
  \node (c) at (1.8,0) {$Z_2$};
  \node (d) at (2.7,-0.75) {$Z_3$};
  \node (e) at (3.6,0) {$Z_4$};
  \node (f) at (4.5,-0.75) {};
  \node (x) at (5.4,0) {};
  \node (y) at (6.3,-0.75) {$Z_M$};
  \node (z) at (7.2,0) {$Y$};
  \node (dots) at (4.95,-0.375) {$\cdots$};
  \draw[->] (a) -- (b);
  \draw[->] (c) -- (b);
  \draw[->] (c) -- (d);
  \draw[->] (e) -- (d);
  \draw[->] (e) -- (f);
  \draw[->] (x) -- (y);
  \draw[->] (z) -- (y);
\end{tikzpicture}
$$
collapses into one with just two maps, $X \to Z \leftarrow Y$.  To show this, it
is enough to prove the following lemma.
\begin{lem}
  Suppose that $W$, $X$, and $Y$ are CW complexes with rational
  ($n$-)equivalences $X \xleftarrow{f} W \xrightarrow{g} Y$.  Then there is a
  complex $Z$ with rational ($n$-)equivalences $X \xrightarrow{f^\prime} Z
  \xleftarrow{g^\prime} Y$.
\end{lem}
\begin{proof}
  By perhaps replacing $X$ and $Y$ with the mapping cylinders of $f$ and $g$
  respectively, we can assume that $W$ is a subcomplex of both $X$ and $Y$.
  Then let $Z=X \cup_W Y$, with $f^\prime:X \to Z$ and $g^\prime:Y \to Z$ the
  corresponding injections.  Consider now the universal covers of all these
  spaces.  By excision, $H_k(\tilde Z,\tilde X) \cong H_k(\tilde Y,\tilde W)$
  for all $k$.  The rational relative Hurewicz theorem then implies that if the
  pair $(Y,W)$ is rationally $n$-connected, then so is the pair $(Z,X)$ and that
  moreover
  $$\pi_{n+1}(Z,X) \otimes \mathbb{Q} \cong H_{n+1}(\tilde Z,\tilde X;\mathbb{Q})
  \cong H_{n+1}(\tilde Y,\tilde W;\mathbb{Q}) \cong \pi_{n+1}(Y,W) \otimes
  \mathbb{Q}.$$
  This argument is symmetric with respect to $X$ and $Y$, so in the case of
  rational equivalence we are done.  In the case that $f$ and $g$ are rational
  $n$-equivalences, this gives us a homomorphism of exact sequences
  $$
  \xymatrix{
    \pi_{n+1}(X,W) \otimes \mathbb{Q} \ar[r] \ar[d]^{\wr} &
    \pi_n(W) \otimes \mathbb{Q} \ar[r]^{\sim} \ar[d]^{\wr} &
    \pi_n(X) \otimes \mathbb{Q} \ar[r] \ar[d]^{f^\prime_*} & 0 \\
    \pi_{n+1}(Z,Y) \otimes \mathbb{Q} \ar[r] &
    \pi_n(Y) \otimes \mathbb{Q} \ar[r]^{g^\prime_*} &
    \pi_n(Z) \otimes \mathbb{Q} \ar[r] & 0.
  }
  $$
  This allows us to conclude that $f^\prime_*$ and $g^\prime_*$ are also
  isomorphisms.  Thus we see that $f^\prime$ and $g^\prime$ are indeed rational
  $n$-equivalences.
\end{proof}

An important technical tool we will use is a wide-ranging generalization of the
Hurewicz theorem introduced by Serre.  Our reference for this is section 9.6 of
Spanier's algebraic topology textbook \cite{Spa}.  First, we need to introduce
some terminology.
\begin{defn}
  A class $\mathcal{C}$ of abelian groups is an \emph{acyclic Serre ideal} if
  \begin{enumerate}
  \item it is closed under subquotients and extensions;
  \item if $A \in \mathcal{C}$ and $B$ is any abelian group, then $A \otimes B,
    \Tor_1(A,B) \in \mathcal{C}$;
  \item if $A \in \mathcal{C}$, then $H_n(A) \in \mathcal{C}$ for every $n>0$.
  \end{enumerate}
  We say that two groups $G$ and $H$ are \emph{isomorphic modulo} $\mathcal{C}$
  if there is a homomorphism $G \to H$ whose kernel and cokernel are in 
  $\mathcal{C}$.
\end{defn}
\begin{thm*}[generalized relative Hurewicz theorem, 9.6.21 in \cite{Spa}]
  Let $\mathcal{C}$ be an acyclic Serre ideal of abelian groups and $(X,A)$ be
  a pair of simply connected CW complexes.  Then the following are equivalent:
  \begin{enumerate}
  \item $\pi_k(X,A) \in \mathcal{C}$ for $2 \leq k \leq n$.
  \item $H_k(X,A) \in \mathcal{C}$ for $2 \leq k \leq n$.
  \end{enumerate}
  Moreover, each of these implies that the Hurewicz homomorphism $\pi_{n+1}(X,A)
  \to H_{n+1}(X,A)$ is an isomorphism modulo $\mathcal{C}$.
\end{thm*}
The most obvious examples of acyclic Serre ideals are the class containing only
the group 0 and the class of torsion groups.  However, there is another such
class that we would like to use.  Suppose that $(X,K)$ is a rationally highly
connected CW pair.  This means that for each element $\alpha \in \pi_n(X)$,
there is some $p>0$ such that $p\alpha$ can be homotoped into $K$.  However, a
priori this $p$ can grow without bound depending on the $\alpha$ we select.  On
the other hand, for a comparison between the metric behavior of $\pi_n(X)$ and
$\pi_n(K)$ to be meaningful, we would prefer to find a universal such $p$, and
indeed this turns out to be possible for finite complexes.  To this end, we
define the class of groups of bounded exponent.
\begin{defn}
  An abelian group $A$ has \emph{bounded exponent} if $pA=0$ for some $p \in
  \mathbb{N}$.  We call the smallest such $p$ the \emph{period} of the group.
\end{defn}
\begin{prop} \label{prop:Serre}
  The class $\mathcal{BE}$ of abelian groups of bounded exponent is an acyclic
  Serre ideal.
\end{prop}
\begin{proof}
  Closure under subquotients, extensions, and tensor products is immediate.

  Suppose $A$ has period $p$.  Since $\Tor_1(\mathbb{Z}/k\mathbb{Z},B)=
  \{x \in B:kx=0\}$, it has period $k$.  Now, the Tor functor commutes with
  filtered colimits and direct sums; in particular,
  $$\Tor_1(A,B)=\bigcup_{G \subset A\text{ finite}} \Tor_1(G,B)=
  \bigcup_{G \subset A\text{ finite}} \bigoplus_{G=\oplus_i \mathbb{Z}/k_i\mathbb{Z}}
  \Tor_1(\mathbb{Z}/k_i\mathbb{Z},B)$$
  has period $p$.

  Since $H_n(\mathbb{Z}/k\mathbb{Z})=0$ for even $n$ and $\mathbb{Z}/k
  \mathbb{Z}$ for odd $n$, the K\"unneth formula tells us that for finite $G$,
  the period of $H_n(G)$ is at most that of $G$.  According to Theorem 9.5.9 of
  \cite{Spa}, for any group $A$, $H_*(A) \cong \varinjlim \{H_*(G): G \subset A
  \text{ finite}\}$.  Hence if $A$ has period $p$ then $H_n(A)$ has period at
  most $p$.
\end{proof}
For the future, we need one more property of groups of bounded exponent.
\begin{prop} \label{prop:bo_mod}
  Let $\Gamma$ be a group and $M$ a finitely generated torsion
  $\mathbb{Z}\Gamma$-module.  Then $M$ is of bounded exponent as a group.
\end{prop}
\begin{proof}
  Let $M=\langle a_1,\ldots,a_m\rangle$.  Then $k_ia_i=0$ for some $k_i>0$.  Let
  $r=\lcm\{k_i:1 \leq i \leq n\}$.  A general element $a \in M$ is given by $a=
  \sum_j t_jg_ja_{i_j}$, where $t_j \in \mathbb{Z}$ and $g_j \in \Gamma$.  Thus
  $ra=0$.
\end{proof}
Now we come to our second technical tool.  From an algebraic-topological
perspective, any map can be replaced by the inclusion into its mapping cylinder.
Thus the next proposition may be thought of as a very weak invertibility result
for rational $n$-equivalences: they have homotopy inverses in the category of
chain complexes with rational coefficients.  In fact, this will be our main tool
for proving rational invariance both in this section and in the rest of the
paper.  Furthermore, a map is rationally $n$-connected if and only if its
rational homotopy fiber is $(n-1)$-connected; thus this proposition also turns
out to be an important tool for analyzing rational homotopy fibrations.
\begin{prop} \label{prop:chain_ho}
  Let $(X,K)$ be a rationally $n$-connected CW pair.  Then the inclusion of
  cellular chain complexes $i_k:C_k(\tilde K;\mathbb{Q}) \to
  C_k(\tilde X;\mathbb{Q})$ is a chain homotopy equivalence through dimension
  $n$.  In particular, for $0 \leq k \leq n$, there is a $\pi_1$-equivariant
  homotopy inverse $j_k:C_k(\tilde X;\mathbb{Q}) \to C_k(\tilde K;\mathbb{Q})$
  such that $j_k \circ i_k=\id$ and $i_k \circ j_k$ is homotopic to the identity
  via a $\pi_1$-equivariant chain homotopy $u_k:C_k(\tilde X;\mathbb{Q}) \to
  C_{k+1}(\tilde X;\mathbb{Q})$ which is zero on the image of $i$.  We call
  $j_k$ a \emph{lifting homomorphism}.
\end{prop}
\begin{proof}
  We produce $j_k$ and $u_k$ by induction on $k$.  For 0-cells $e$ of
  $\tilde X$ outside $\tilde K$, we can set $j_0(e)$ to be any 0-cell of
  $\tilde K$ in an equivariant way, and $u_0(e)$ to be a path between $e$ and
  $j_0(e)$.  Now suppose that we have constructed $j_{k-1}$ and $u_{k-1}$.  Let
  $e$ be a $k$-cell of $\tilde X$.  Then $e+u_{k-1}(\partial e)$ represents an
  element of $H_k(\tilde X,\tilde K;\mathbb{Q})$ and so there's a chain $j_e \in
  C_k(\tilde K;\mathbb{Q})$ such that $i_k(j_e)-e-u_{k-1}(\partial e)=
  \partial u_e$ for some $u_e \in C_{k+1}(\tilde X;\mathbb{Q})$.  We set $j_k(e)=
  j_e$ and $u_k(e)=u_e$.  We then extend equivariantly over the equivalence
  class of $e$.  If $e$ is a cell of $\tilde K$, we can take $j_e=e$ and $u_e=0$
  by induction.
\end{proof}
With the two main tools in place, we can show that rationally highly connected
finite CW pairs have homotopy groups with bounded exponent.
\begin{cor} \label{cor:bo_pi}
  Suppose $n \geq 2$ and $(X,K)$ is a rationally $n$-connected finite CW pair.
  Then for $i \leq n$, $H_i(\tilde X,\tilde K)$ and $\pi_i(X,K) \cong
  \pi_i(\tilde X,\tilde K)$ have bounded exponent.  Moreover, $H_{n+1}(\tilde X,
  \tilde K) \cong \pi_{n+1}(X,K) \mod \mathcal{BE}$.
\end{cor}
\begin{proof}
  Let $j_k:C_k(\tilde X;\mathbb{Q}) \to C_k(\tilde K;\mathbb{Q})$ and $u_k:
  C_k(\tilde X;\mathbb{Q}) \to C_{k+1}(\tilde X;\mathbb{Q})$ be a lifting
  homomorphism and the corresponding chain homotopy.  There are a finite number
  of equivalence classes of $k$-cells $e$ in $\tilde X$, and for each such cell,
  there are elements $\alpha_e \in C_k(\tilde K;\mathbb{Z})$ and $\beta_e \in
  C_{k+1}(\tilde X;\mathbb{Z})$ and integers $q_e$ and $q^\prime_e$ such that
  $$j_k(e)=\frac{1}{q_e}\alpha_e\text{ and }
  u_k(e)=\frac{1}{q_e^\prime}\beta_e.$$
  Taking the least common denominator $q_k$ of all the $q_e$'s and
  $q_e^\prime$'s, we get that for any $k$-chain $c$ with boundary in $\tilde K$,
  $q_kc$ is integrally homologous to a $k$-chain in $\tilde K$.  This proves
  that $H_i(\tilde X,\tilde K)$ have bounded exponent.  The rest of the
  conclusion follows by the generalized relative Hurewicz theorem.
\end{proof}
We can also now prove a rational version of Wall's theorem.  We will use it to
prove two corollaries on the finiteness of certain complexes which will in turn
be applied later in the paper.  Corollary \ref{lem:QFin} shows that rationally
equivalences can be kept within the category of complexes with finite skeleta,
while Theorem \ref{cor:XandY} relates the finiteness properties of the three
spaces in a rational homotopy fibration.
\begin{thm}[rational Wall theorem] \label{thm:QWall}
  Let $X$ be a CW complex and $n \geq 2$.  Then the following are equivalent:
  \begin{enumerate}
  \item $X$ is rationally equivalent to a CW complex $Y$ with finite
    $n$-skeleton.
  \item There is a CW complex $Y$ with finite $n$-skeleton and a rational
    equivalence $Y \to X$.  (Equivalently, there is an $n$-complex with an
    $n$-connected map to $X$.)
  \item The group $\Gamma:=\pi_1(X)$ is finitely presented, and for every $k
    \leq n$, the condition $F_k(\mathbb{Q})$ holds: for every finite complex
    $K^{k-1}$ and rationally $(k-1)$-connected map $\ph:K \to X$, $\pi_k(\ph)
    \otimes \mathbb{Q}$ is a finitely generated $\mathbb{Q}\Gamma$-module.
  \end{enumerate}
\end{thm}
\begin{proof}
  (2) clearly implies (1).

  Suppose (1) is true.  That is, there are a finite complex $Y$ and a complex
  $Z$ such that $f:X \to Z$ and $g:Y \to Z$ are rational equivalences.  Since
  $\Gamma=\pi_1(Y)$, it must be finitely presented.  Now suppose that $K$ is a
  $(k-1)$-complex and $\ph:K \to X$ is a rationally $(k-1)$-connected map.  Then
  $\psi:=f \circ \ph$ is as well, so that, by Hurewicz,
  $$\pi_k(\psi) \otimes \mathbb{Q}=\pi_k(Z,K) \otimes \mathbb{Q} \cong
  \pi_k(\tilde Z,\tilde K;\mathbb{Q}) \cong H_k(\tilde Z,\tilde K;\mathbb{Q}).$$
  Moreover, we can assume by taking mapping cylinders that $K$ and $Y$ are
  subcomplexes of $Z$.  Let $j_\bullet:C_\bullet(\tilde Z;\mathbb{Q}) \to
  C_\bullet(\tilde K;\mathbb{Q})$ and $u_\bullet:C_\bullet(\tilde Z;\mathbb{Q})
  \to C_{\bullet+1}(\tilde Z;\mathbb{Q})$ be the lifting homomorphism and
  associated chain homotopy for the pair $(\tilde Z,\tilde K)$, and let
  $j^\prime_\bullet$ and $u^\prime_\bullet$ be the lifting homomorphism and chain
  homotopy for the pair $(\tilde Z,\tilde Y)$.
  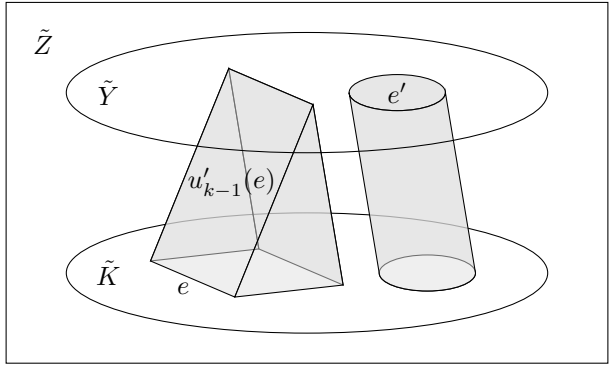
\begin{figure}
    \centering
    \begin{tikzpicture}[scale=0.8]
      \tikzstyle{chainfill}=[fill=gray!25, fill opacity=0.5];
      \draw (0,0) circle [y radius=1, x radius=4];
      \node (K) at (-3.3,0) {$\tilde K$};
      \draw (-5,-1.5) -- (5,-1.5) -- (5,4.5)
      -- (-5,4.5) node[anchor=north west,inner sep=10pt] {$\tilde Z$} -- cycle;
      \fill[chainfill] (0.7,3) arc (180:0:0.8 and 0.3) -- (2.8,0)
      arc (0:180:0.8 and 0.3) -- cycle;
      \filldraw[chainfill] (1.5,3) circle [y radius=0.3, x radius=0.8]
      node[opacity=1]{$e^\prime$};
      \draw (2,0) circle [y radius=0.3, x radius=0.8];
      \filldraw[chainfill] (0.7,3) arc (180:360:0.8 and 0.3) -- (2.8,0)
      arc (360:180:0.8 and 0.3) -- cycle;
      \coordinate (L1) at (-2.6,0.2);
      \coordinate (R1) at (-1.2,-0.4);
      \coordinate (L2) at (-1.3,3.4);
      \coordinate (R2) at (0.1,2.8);
      \coordinate (L3) at (-0.8,0.4);
      \coordinate (R3) at (0.6,-0.2);
      \filldraw[chainfill] (L1) -- (L2) -- (L3) -- cycle;
      \filldraw[chainfill] (L2) -- (R2) -- (R3) -- (L3) -- cycle;
      \filldraw[chainfill] (L1) -- (R1)
      node[pos=0.5, anchor=north east, inner sep=1pt, opacity=1]{$e$}
      -- (R2) -- (L2) -- cycle;
      \filldraw[chainfill] (R1) -- (R2) -- (R3) -- cycle;
      \node (a) at (-1.25,1.5) {$u_{k-1}^\prime(e)$};
      \draw (0,3) circle [y radius=1, x radius=4];
      \node (Y) at (-3.3,3) {$\tilde Y$};
    \end{tikzpicture}

    \caption{
      A schematic illustrating $F(e^\prime,e)$ for an arbitrary pair of chains
      $e^\prime \in C_k(\tilde Y;\mathbb{Q})$ and $e \in
      C_{k-1}(\tilde K;\mathbb{Q})$.
    } \label{fig:cycle}
  \end{figure}
  Then there is a homomorphism
  $$F:C_k(\tilde Y;\mathbb{Q}) \oplus C_{k-1}(\tilde K;\mathbb{Q}) \to
  H_k(\tilde Z,\tilde K;\mathbb{Q}),$$
  given on cells by
  $$F(e^\prime,e)=[e^\prime+u_{k-1}(\partial e^\prime)-u_{k-1}^\prime(e)-
  u_{k-1}(\partial u_{k-1}^\prime(e))].$$
  The boundary of this chain is
  $$j_{k-1}(\partial e^\prime)-j_{k-1}(\partial u_{k-1}^\prime(e)) \in
  C_{k-1}(\tilde K;\mathbb{Q}),$$
  so it is in fact a cycle in $(\tilde Z,\tilde K)$.  Moreover, if $c$ is a
  chain in $\tilde Z$ with boundary in $\tilde K$, then
  \begin{align*}
    F(j^\prime_k(c),\partial c) &= [(j^\prime_k(c)-u_{k-1}^\prime(\partial c))
      +u_{k-1}(\partial j^\prime_k(c))
      -u_{k-1}(\partial u_{k-1}^\prime(\partial c))] \\
    &= [(c+\partial u_k^\prime(c))+u_{k-1}(j^\prime_{k-1}(\partial c))+
      (-u_{k-1}(j_{k-1}^\prime(\partial c))+u_{k-1}(\partial c))] \\
    &=[c+\partial u_k^\prime(c)]=[c] \in H_k(\tilde Z,\tilde K;\mathbb{Q}),
  \end{align*}
  showing that $F$ is onto.  Since the domain of $F$ is a finitely generated
  free $\mathbb{Q}\Gamma$-module, this means that
  $H_k(\tilde Z,\tilde K;\mathbb{Q})$ is finitely generated, proving (3).

  On the other hand, suppose (3) is true.  We will inductively construct a $Y$
  with finite $n$-skeleton with a rational equivalence $Y \to X$.  Since
  $\Gamma$ is finitely presented, we can construct a finite 2-complex $A$ with
  $\pi_1(A)=\Gamma$.  This gives a 1-connected map $\ph:A \to X$, so $\pi_2(\ph)
  \otimes \mathbb{Q}$ is finitely generated.  We can add a finite number of
  2-cells to kill it, building $Y^{(2)}$ and extending $\ph$.  By induction, we
  can build a finite complex $Y^{(n)}$ and extend $\ph$ to it.  Finally, we can
  add cells to build the rest of $Y$.
\end{proof}
\begin{cor} \label{lem:QFin}
  Suppose, for some $n \geq 2$, that $K$ and $L$ are CW-complexes with finite
  $n$-skeleta and $K \xrightarrow{f} Y \xleftarrow{g} L$ are rational
  equivalences.  Then indeed there are rational equivalences $K \xrightarrow{f}
  Z \xleftarrow{g} L$ where $Z$ has finite $n$-skeleton.
\end{cor}
\begin{proof}
  As usual, let $\Gamma$ refer to $\pi_1K=\pi_1L=\pi_1Y$.

  By taking a double mapping cylinder, we can assume that $f$ and $g$ are both
  inclusions.  By inducting on dimension, we will build a complex $Z$ with maps
  $$
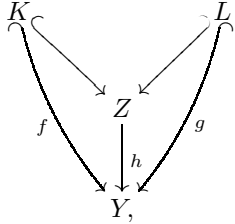
\xymatrix{
    K \ar@{_{(}->}@/_/[rdd]_f \ar@{^{(}->}[rd] && L \ar@{^{(}->}@/^/[ldd]^g
    \ar@{_{(}->}[ld] \\
    & Z \ar[d]^h & \\
    & Y, &
  }$$
  where $h$ is a rational equivalence, and hence so are the inclusions into $Z$.
  For the base case, let $\ph:L^{(1)} \times [0,1] \to Y$ be a homotopy taking
  $g|_{L^{(1)}}$ into $K^{(1)}$, and let $Z^{(2)}=K^{(2)} \cup_{\ph|L^{(1) \times \{1\}}}
  L^{(1)} \times [0,1] \cup_{\ph|L^{(1) \times \{0\}}} L^{(2)}$.  Then the inclusion
  $Z^{(2)} \hookrightarrow Y$ induces an isomorphism on $\pi_1$.  Moreover, any
  element in $\pi_2(K)$ has a representative which maps to $Z^{(2)}$, so
  $\pi_2(Z) \otimes \mathbb{Q}$ surjects onto $\pi_2(Y) \otimes \mathbb{Q}$.  In
  particular, $Z^{(2)} \hookrightarrow Y$ is rationally $2$-connected.

  Now suppose we have constructed $Z^{(k)}$ with rationally $k$-connected
  $h_k:Z^{(k)} \to Y$.  Then since $Y$ has the rational homotopy type of $K$,
  the rational Wall theorem tells us that $\pi_{k+1}(h_k) \otimes \mathbb{Q}$ is
  a finitely generated $\mathbb{Q}\Gamma$-module, and so we can add in a finite
  number of $(k+1)$-cells to kill it, in such a way that the map $h_k$ extends
  to an $h_{k+1}$ which is then rationally $(k+1)$-connected.  To ensure that
  $K$ and $L$ are included in our final $Z$, we make sure that all their
  $(k+1)$-cells are on the list.

  By induction, $h$ is a rational equivalence, and hence so are the inclusions
  $K \hookrightarrow Z \hookleftarrow L$.
\end{proof}
\begin{thm} \label{cor:XandY}
  Let $n \geq 2$, and suppose $f:X \to Y$ is a $\pi_1$-isomorphic map of CW
  complexes such that $\pi_k(f) \otimes \mathbb{Q}$ is finite-dimensional for
  $2 \leq k \leq n+1$ and $\pi_2(f) \otimes \mathbb{Q}=0$.   Then $Y$ is
  rationally equivalent to a complex with finite $n$-skeleton if and only if $X$
  is.
\end{thm}
\begin{proof}
  Suppose first that $Y$ is rationally equivalent to a complex with finite
  $n$-skeleton.  We will show that the homotopy fiber of $f$ is also rationally
  equivalent to a complex with finite $n$-skeleton, and use this to construct an
  $n$-complex with an $n$-connected map to $X$.

  By adding cells in dimensions 3 and 4 to $Y$, we can kill the torsion subgroup
  of $\pi_2Y$ without affecting its rational homotopy type.  Thus we can assume
  that $\pi_2(f)=0$.  Using the standard path space construction, we can assume
  that $f$ is a fibration; let $F$ be its fiber, which is simply connected.  By
  assumption, $\pi_k(F) \otimes \mathbb{Q}$ is finite-dimensional for $k \leq
  n$, so by the generalized Hurewicz theorem, so is $H_k(F;\mathbb{Q})$.  This
  means that $F$ is rationally equivalent to a complex with finite $n$-skeleton,
  for example by Theorem 9.11 of \cite{FHT}; by the rational Wall theorem we can
  find a finite complex $F^\prime$ such that the map $F^\prime \to F$ is
  rationally $n$-connected.

  Let $B \to Y$ be a rationally $n$-connected map from a finite complex with one
  0-cell to $Y$.  We will construct finite rational approximations $A_k$ to
  $E_k:=f^{-1}(B^{(k)})$ by induction on $k$.  We write $\tilde B$ for the
  universal and let $\tilde E_k$, etc., be the corresponding covers of our other
  spaces mapping to $B$; note that these are not always universal covers.

  Clearly, $F^\prime \to E_0 \cong F$ is a rationally $n$-connected map, so we
  can set $A_0=F^\prime$.  Moreover, we can construct $A_1$ as follows.  For each
  1-cell $c:[0,1] \to B$, we have two rationally $n$-connected maps $F^\prime
  \to c^*X$ corresponding to the two endpoints of the interval.  We can use the
  proof of Corollary \ref{lem:QFin} to find a complex which includes them both
  and which maps to $c^*X$ via a rationally $n$-connected map.  To construct
  $A_1$ as desired, we construct such a complex for each 1-cell and glue them
  together along all the copies of $F^\prime$.  Then the map $A_1 \to E_1$ is
  rationally $n$-connected since the universal covers $\tilde A_1$ and
  $\tilde E_1$ live in a commutative square
  $$\xymatrix{
    F^\prime \ar[r] \ar[d] & \tilde A_1 \ar[d] \\
    F \ar[r] & \tilde E_1
  }$$
  where the other three maps are rationally $n$-connected.

  Now let $k \geq 2$, and suppose we have constructed a rationally $n$-connected
  map $A_{k-1} \to E_{k-1}$.  We will add cells to $A_{k-1}$, again by induction
  on dimension, in order to build $A_k$.  We will use the following fact, which
  is a standard step in deriving the Serre spectral sequence, as for example in
  \cite{HatSS}:
  \begin{equation} \label{eq:HxH}
    H_i(\tilde E_k,\tilde E_{k-1}) \cong H_k(\tilde B^{(k)},\tilde B^{(k-1)})
    \otimes H_{i-k}(F).
  \end{equation}
  We start with the base case, which is slightly different for $k=2$ and $k>2$.
  For $k>2$, the map $A_{k-1} \to E_k$ is already rationally $(k-1)$-connected,
  and so by the generalized relative Hurewicz theorem,
  $$\pi_k(E_k,A_{k-1}) \otimes \mathbb{Q} \cong H_k(\tilde E_k,\tilde A_{k-1};
  \mathbb{Q}) \cong H_k(\tilde E_k,\tilde E_{k-1};\mathbb{Q}).$$
  By \eqref{eq:HxH}, this is a finitely generated module, so it can be killed by
  adding a finite number of $k$-cells.  This gives us a finite complex $A_k(0)$
  with a rationally $k$-connected map $A_k(0) \to E_k$.

  In the case $k=2$, we instead have that $A_1 \to E_2$ is an (integrally)
  1-connected map, and we can use another form of the Hurewicz theorem (Theorem
  4.37 in \cite{Hatc}) to show in a similar way that $\pi_2(E_2,A_1)$ is
  finitely generated over $\mathbb{Z}\pi_1E_1$.

  Now let $0<i \leq n-k$, and suppose we have constructed a finite complex
  $A_k(i-1)$ with a rationally $(k+i-1)$-connected map $A_k(i-1) \to E_k$.
  Consider the exact sequence
  $$H_{k+i}(\tilde E_k,\tilde A_{k-1};\mathbb{Q}) \to
  H_{k+i}(\tilde E_k,\tilde A_k(i-1);\mathbb{Q}) \to
  H_{k+i-1}(\tilde A_k(i-1),\tilde A_{k-1};\mathbb{Q}).$$
  The last module is finitely generated because $A_k(i-1)$ and $A_{k-1}$ are
  finite complexes; the first, by \eqref{eq:HxH}.  Therefore
  $H_{k+i}(\tilde E_k,\tilde A_k(i-1);\mathbb{Q}) \cong \pi_{k+1}(\tilde E_k,
  \tilde A_k(i-1)) \otimes \mathbb{Q}$ is finitely generated.  Killing it gives
  us $A_k(i)$ as desired.

  In the end, we obtain a rationally $n$-connected map $A_n \to X$.  We can
  complete $A_n$ to a complex rationally equivalent to $X$ by adding cells in
  dimensions $n+1$ and higher.  This completes one direction of the proof.

  Now we tackle the other direction.  Suppose that $X$ is rationally equivalent
  to a complex with finite $n$-skeleton, and suppose that $Y$ is not.  Let $3
  \leq k \leq n$ be the first dimension in which a complex rationally equivalent
  to $Y$ necessarily has infinitely many cells, and so $\pi_k(Y,Y^{(k-1)})
  \otimes \mathbb{Q}$ is infinitely generated.  Consider the homotopy pullback
  $g:Z \to Y^{(k-1)}$ of the homotopy fibration $f:X \to Y$; in particular,
  $\pi_i(g)=\pi_i(f)$ for every $i$ and the map $Z \to X$ is rationally
  $(k-1)$-connected.  By the $Y \Rightarrow X$ direction, $Z$ is rationally
  equivalent to a complex $K$ with finite skeleta and indeed there is a rational
  equivalence $K \to Z$.  Let $\ph:K^{(k-1)} \to X$ and $\psi:K^{(k-1)} \to
  Y^{(k-1)}$ be the maps which factor through $Z$, so that $\pi_k(\ph) \otimes
  \mathbb{Q}$ is finitely generated by Theorem \ref{thm:QWall}, whereas
  $\pi_k(\psi) \otimes \mathbb{Q}$ is finitely generated by the exact sequence
  of triples
  $$\pi_k(K,K^{(k-1)}) \otimes \mathbb{Q} \to \pi_k(\psi) \otimes \mathbb{Q}
  \to \pi_k(f) \otimes \mathbb{Q} \to 0.$$
  The exact sequence of triples
  $$\pi_k(\ph) \otimes \mathbb{Q} \to \pi_k(Y,K^{(k-1)}) \otimes \mathbb{Q}
  \to \pi_k(f) \otimes \mathbb{Q} \to 0$$
  shows that $\pi_k(Y,K^{(k-1)}) \otimes \mathbb{Q}$ is finitely generated, while
  the exact sequence of triples
  $$\pi_k(\psi) \otimes \mathbb{Q} \to \pi_k(Y,K^{(k-1)}) \otimes \mathbb{Q} \to
  \pi_k(Y,Y^{(k-1)}) \otimes \mathbb{Q} \to \pi_{k-1}(\psi) \otimes \mathbb{Q}$$
  shows that $\pi_k(Y,K^{(k-1)}) \otimes \mathbb{Q}$ is infinitely generated,
  since $\pi_{k-1}(\psi) \otimes \mathbb{Q} \cong \pi_{k-1}(f) \otimes
  \mathbb{Q}$.  Thus we have a contradiction.
\end{proof}

\section{Basic properties}

In this section, we define distortion functions and discuss relationships
between different definitions.  As mentioned in the introduction, what we mean
by ``distortion'' can be defined for any measure of ``size'' in a group, that
is, for any subadditive functional on that group.  We will start with this
purely formal definition before specifying it to Lipschitz and volume distortion
in higher homotopy and homology groups.

The Lipschitz and volume functionals we use are only defined up to additive and
multiplicative constants, forcing us to discuss distortion functions only up to
asymptotic equivalence.  That is, when comparing functions $\mathbb{N} \to
\mathbb{R} \cup \{\infty\}$, we will use the relations
\begin{align*}
  f \lesssim g &\iff \text{for some $A, B, C, D$, }f(n) \leq Ag(Bn+C)+D \\
  f \sim g &\iff f \lesssim g\mbox{ and }f \gtrsim g.
\end{align*}
We now give a number of formal definitions relating to distortion functions.
\begin{defn}
  Let $G$ be an abelian group, let $\ph:G \to G \otimes \mathbb{Q}$ be the
  rationalization homomorphism, and $F:G \to  \mathbb{R}^+$ a subadditive
  functional.  Define the $F$-\emph{distortion function} of $\alpha \in G$ to be
  the functon $\delta_{\alpha,F}:\mathbb{N} \to \mathbb{R}^+ \cup \{\infty\}$
  given by
  $$\delta_{\alpha,F}(k)=\sup\{m \mid F(m\alpha) \leq k\}.$$
  We say $\alpha$ is \emph{distorted} if $F(k\alpha)/k \to 0$ as $k
  \to \infty$, and that $\alpha$ is \emph{infinitely distorted} if there is a
  $k$ such that $\delta_{\alpha,F}(k)=\infty$.

  Similarly, given a norm $\lVert\cdot\rVert$ on a finite-dimensional vector
  subspace $V \subseteq G \otimes \mathbb{Q}$, we can define the $F$-distortion
  function of $V$ to be
  $$\delta_{V,F}(k)=\sup\{\lVert\ph(\alpha)\rVert \mid \alpha \in G\text{ with }
  \ph(\alpha) \in V\text{ and }F(\alpha) \leq k\}.$$
  Note that the asymptotics of $\delta_V$ do not depend on the norm.  We say
  that $V$ is \emph{distorted} if $k/\delta_{V,F}(k) \to 0$ as $k \to \infty$,
  and \emph{infinitely distorted} if there is a $k$ such that
  $\delta_{V,F}(k)=\infty$.

  Finally, we will say that $V$ is \emph{weakly distorted} in another
  finite-dimensional subspace $U$ if for some constant $C$,
  $$\sup \{\lVert\ph(\alpha)\rVert \mid \dist(\ph(\alpha),V)<C\text{ and }F(
  \alpha) \leq k\} \gnsim k.$$
  Similarly, $V$ is \emph{weakly infinitely distorted} in $U$ if there are
  constants $C$ and $k$ such that
  $$\sup \{\lVert\ph(\alpha)\rVert \mid \dist(\ph(\alpha),V)<C\text{ and }F(
  \alpha) \leq k\}=\infty.$$
  These conditions also clearly do not depend on $\lVert\cdot\rVert$.
\end{defn}
That is, a subspace is weakly distorted if there are vectors which close to it,
but not necessarily in it, on which $F$ is small.  It turns out, as we show
below, that for purely formal reasons weak distortion always implies distortion,
at least when distortion in $V$ is a meaningful idea at all.  This contrasts
with the infinite version; as we will show in Example \ref{exs:oddD}(\ref{qqu}),
weak infinite distortion \emph{does not} imply infinite distortion.  An
identical argument shows that if the one-dimensional subspace
$\mathbb{Q}\ph(\alpha)$ is $F$-distorted, then so is the element $\alpha$.
However, here again the distortion functions may be different; such a case will
be demonstrated in Example \ref{exs:oddD}(1).
\begin{lem}
  Let $G$ be an abelian group and $F$ a subadditive functional, and let $V
  \subset U \subset G \otimes \mathbb{Q}$ be finite-dimensional subspaces.  If
  $\ph(G)$ contains a basis of $V$ and $V$ is weakly $F$-distorted in $U$, then
  $V$ is $F$-distorted.
\end{lem}
\begin{proof}
  Fix a finite set $\{\alpha_i\} \subset G$ such that $\{\ph(\alpha_i\}$
  generates $\ph(\ph^{-1}(U))$, and let $C^\prime=\max_i F(\alpha_i)/
  \lVert\ph(\alpha_i)\rVert$.  Then for every $\gamma \in G$ with $\ph(\gamma)
  \in U$, there is a sufficiently large $M$ that $M\gamma$ is in the subgroup of
  $G$ generated by the $\{\alpha_i\}$ and therefore $F(\gamma) \leq C^\prime M
  \lVert\Gamma\rVert$.

  The assumptions guarantee that $\ph(G) \cap V$ is ``coarsely dense'' in $V$:
  that is, there is a constant $K$ such that every point in $V$ is at most
  distance $K$ from $\ph(G)$.  Since $V$ is weakly $F$-distorted, there is a $C$
  such that for every $\epsi>0$ and $N>0$ there is a $\alpha \in G$ such that
  $\dist(\alpha,V)<C$ and $N<F(\alpha) \leq \epsi\lVert\ph(\alpha)\rVert$.  We
  can then write $\alpha=\beta+\gamma$ where $\ph(\beta) \in V$ and
  $\lVert\ph(\gamma)\rVert \leq K+C$.  Let $M$ be such that $M\gamma$ is in
  the subgroup generated by $\{\alpha_i\}$.  Then
  $$F(M\beta) \leq \epsi M\lVert\ph(\alpha)\rVert+C^\prime M(K+C) \leq M
  (\epsi\lVert\ph(\alpha)\rVert+\text{const}).$$
  Supposing that for a given $\epsi>0$ we choose $N=C^\prime(K+C)/\epsi$, this
  gives us
  $$F(M\beta) \leq \text{const} \cdot \epsi M\lVert\ph(\beta)\rVert=
  \text{const} \cdot \epsi \lVert\ph(M\beta)\rVert.$$
  Since we can satisfy this inequality for every $\epsi>0$, $V$ is
  $F$-distorted.
\end{proof}
Note that this proof is utterly ineffective: the fact that $V$ is weakly
distorted means only that it has \emph{some} superlinear distortion function,
but its divergence from linearity could be arbitrarily slow.

In these definitions, $G$ may be rather complicated even when $G \otimes
\mathbb{Q}$ is a finitely generated vector space.  For example, $G=
\mathbb{Z}\left[\frac{3}{5}+\frac{4}{5}i\right] \subset \mathbb{Q}[i] \cong
\mathbb{Q}^2$ is infinitely generated as a group; if we see it as generated by
$\left\{\left(\frac{3}{5}+\frac{4}{5}i\right)^n: n \in \mathbb{N}\right\}$, then
each generator has the same 2-norm in $\mathbb{Q}^2$.  Thus the universe of
conceivable behaviors of the functional $F$ is very large indeed.

It is worth giving a little bit of flavor as to the kind of functionals $F$ that
we will deploy.  A basic example is as follows: let $G=\mathbb{Z}$, and let
$F(z)$ denote the least number of powers of 2 one needs to add or subtract to
get to $z$.  Thus for example $F(1023)=F(2^{10}-2^0)=2$.  Then $F$ is unbounded,
but so is $\{z \in \mathbb{Z}: F(z)=1\}$; therefore $\delta_{1,F}(n)=\infty$ for
every $n \geq 1$.  In other words, the element 1 is infinitely distorted, even
though some of its multiples have large minimal representations.

More generally, to any group $G$ and any generating set $\{g_i\}_{i \in I}$ we
may associate the functional $F$ which, for any $g \in G$, gives the minimal sum
of the coefficients of a way of representing $g$ as $\sum_i C_ig_i$.  In
Examples \ref{exs:oddD} and the later parts of section 3 we discuss several
examples of this type; the complexity revealed there is as much a consequence of
the linear algebra in the definitions we have given thus far as it is of the
geometry of the relevant spaces.

We now specify these ideas to the study of homotopy groups and their metric
properties.  Let $X$ be a compact Riemannian manifold with boundary or a finite
CW-complex with a piecewise Riemannian metric.  Rademacher's theorem, applied to
some piecewise smooth local embedding of $X$ into $\mathbb{R}^N$, tells us that
a Lipschitz map $f:S^n \to X$ is almost everywhere differentiable.  In
particular, one can define its \emph{volume}:
$$\vol f:=\int_{S^n} |\Jac(f)|d{\vol}.$$
\begin{defn}
  Given $\alpha \in \pi_n(X)$, write
  $$\lvert\alpha\rvert_{\Lip}:=\sup\{\Lip f \mid f:S^n \to X
  \text{ is a Lipschitz representative of }\alpha\}.$$
  The \emph{Lipschitz distortion} of $\alpha$, written $\LD_\alpha(k)$, is the
  distortion of $\alpha$ with respect to the functional $F(\alpha)=\lvert\alpha
  \rvert_{\Lip}$.  Similarly, write
  $$\lvert\alpha\rvert_{\vol}:=\min\{\vol f \mid f:S^n \to X
  \text{ is a Lipschitz representative of }\alpha\}.$$
  The \emph{volume distortion} of $\alpha$, written $\VD_\alpha(k)$, is the
  distortion of $\alpha$ with respect to the volume functional..

  Given a finite-dimensional vector subspace $V \subseteq \pi_n(X) \otimes
  \mathbb{Q}$ together with a norm $\lVert\cdot\rVert$, we define the distortion
  functions
  \begin{align*}
    \LD_V(k) &= \sup\{\lVert\vec v\rVert \mid \exists \alpha \in \pi_n(X)
    \text{ such that }\alpha\overset{\otimes \mathbb{Q}}{\mapsto}\vec v \in V
    \text{ with }\lvert\alpha\rvert_{\Lip} \leq k\},\\
    \VD_V(k) &= \sup\{\lVert\vec v\rVert \mid \exists \alpha \in \pi_n(X)
    \text{ such that }\alpha\overset{\otimes \mathbb{Q}}{\mapsto}\vec v \in V
    \text{ with }\lvert\alpha\rvert_{\vol} \leq k\}.
  \end{align*}
\end{defn}
If $\alpha$ is torsion, then $|k\alpha|_{\Lip}$ and $|k\alpha|_{\vol}$ are
bounded, and so the corresponding distortion functions are eventually infinite.
In all cases, $|k\alpha|_{\Lip} \lesssim k^{1/n}\alpha$ because we can get a
representative of $k\alpha$ by precomposing a map $f \in \alpha$ with a map $S^n
\to S^n$ of degree $k$ and Lipschitz constant $k^{1/n}$.  Thus for any $\alpha$,
$\LD_\alpha(k) \gtrsim k^n$.  Moreover, a $k$-Lipschitz map has volume at most
$k^n$, so for any $\alpha$, $\VD_\alpha(k) \leq \LD_\alpha(k)^{1/n}$.  In
particular, $\VD_\alpha(k)$ is at least linear.  This motivates another
definition:
\begin{defn}
  A class $\alpha$ or subspace $V$ in $\pi_n(X) \otimes \mathbb{Q}$ is
  \emph{Lipschitz undistorted} if $\LD_\alpha(k) \sim k^n$ (respectively
  $\LD_V(k) \sim k^n$), and \emph{undistorted} if in addition its volume
  distortion is linear.
\end{defn}

Clearly, for any subspace $V$ and any $0 \neq \alpha \in V$, $\LD_V(k) \gtrsim
\LD_\alpha(k)$ and $\VD_V(k) \gtrsim \VD_\alpha(k)$.  On the other hand, we will
later see an example of an $X$ such that $\pi_n(X) \otimes \mathbb{Q}$ is
distorted although none of its elements are.  This provides extra motivation for
the definitions of $\LD_V$ and $\VD_V$.

A Lipschitz homotopy equivalence $X \underset{g}{\overset{f}{\rightleftarrows}}
Y$ induces an asymptotic equivalence on distortion functions, since
$$|\alpha|_{\Lip} \leq \Lip g|f_*\alpha|_{\Lip} \leq \Lip f\Lip g|\alpha|_{\Lip}$$
and
$$|\alpha|_{\vol} \leq (\Lip g)^n|f_*\alpha|_{\vol} \leq (\Lip f)^n(\Lip g)^n
|\alpha|_{\vol}.$$
In particular, we can speak of distortion functions of finite CW complexes and
compact manifolds without specifying a metric.  Another way of simplifying our
object of study is to restrict to a more combinatorial class of functions $S^n
\to X$.
\begin{defn}[\cite{BBFS}]
  Given a CW-complex $X$ and an $n$-manifold $M$, call a map $f:M \to X$
  \emph{admissible} if $f(M) \subseteq X^{(n)}$ and for every interior $U$ of an
  $n$-cell of $X$, $f^{-1}(U)$ is a disjoint union of balls which map
  homeomorphically to $U$.  If $M$ is compact, define the \emph{cellular volume}
  $\vol_C(f)$ to be the total number of these balls.
\end{defn}
If $X$ has a piecewise Riemannian metric and a finite $n$-skeleton, then for any
admissible map $f$,
$$c\vol(f) \leq \vol_C(f) \leq C\vol f,$$
where $c$ and $C$ are the least and greatest volume of a cell, respectively.
Thus asymptotically speaking, it doesn't matter which notion of volume we
consider.  Before we can use this fact, however, we need to make sure that the
asymptotics of distortion functions remain the same if we only consider
admissible maps.  For this, we use a theorem originally from geometric measure
theory.
\begin{thm}[Federer--Fleming deformation theorem \cite{ECHLP}]
  For any finite-dimensional simplicial complex $Y$, for example for a
  triangulated manifold, there is a $c>0$ with the following property.  Any
  Lipschitz $k$-cycle $T$ can be decomposed as $T=Q+\partial R$, where $R$ is a
  Lipschitz $(k+1)$-chain and $Q$ is a smooth $k$-cycle whose simplices are
  cellular, such that $\mass_{k+1}R \leq c\mass_k Q \leq c^2\mass_k T$ and $Q$
  and $R$ are contained in the smallest subcomplex of $Y$ containing $T$.

  Indeed, if $N$ is a triangulated manifold, then a Lipschitz map $f:N
  \to Y$ is Lipschitz homotopic, via a homotopy of mass bounded by $c^2\vol f$,
  to an admissible map $g$ with $\vol_C(g)\leq c\vol f$.
\end{thm}
The second statement, though not stated as such by \cite{ECHLP}, falls out of
the proof they give.  Every CW complex $X$ has a simplicial approximation $Y$
with a homotopy equivalence $Y \to X$ that sends each $k$-simplex either
homeomorphically to a cell or to $X^{(k-1)}$, and when $X$ is compact this can
be made Lipschitz.  Therefore, the same statement holds for CW complexes.  Thus
we have the following consequences.
\begin{cor}
  The minimal volume of a representative of $\alpha \in \pi_n(X)$ is
  approximated to within a multiplicative constant, depending on $n$ and $X$, by
  an admissible representative.  To find the asymptotic behavior of the
  distortion function of an element of $\pi_n(X)$, it is enough to consider
  admissible representatives.  In particular, the asymptotic behavior of the
  volume distortion functions of $\pi_n(X)$ depends only on the topology of the
  $(n+1)$-skeleton of $X$.
\end{cor}
These homotopy invariance results are also true for Lipschitz constants, using a
Lipschitz version of the deformation theorem which is beyond the scope of this
paper.  But the following lemma gives a stronger independence result which only
seems to work for volume distortion.  Namely, it shows that all multiples have
maps that pull back through rational equivalences in a volume-preserving way.

To demonstrate some of the intuition behind this, here is a warmup problem.
Consider the rational equivalence $Z \to T^3$, where $Z$ is defined as follows.
The 3-cell of $T^3$ has an attaching map $a:S^2 \to (T^3)^{(2)}$; define $Z=
(T^3)^{(2)} \cup_{a \circ u} D^3$, where $u:S^2 \to S^2$ is a map of degree 2.
Then an extension of $u$ to the interior of $D^3$ defines a map $\ph:Z \to T^3$;
since $\widetilde{T^3}$ is contractible and $\tilde Z$ has torsion homology in
all dimensions $n>0$, this is a rational equivalence.  Let $f:(D^3,S^2) \to
\widetilde{T^3}$ be the composition
$$D^3 \xrightarrow{\text{degree }2} [0,k]^3 \hookrightarrow \mathbb{R}^3.$$
We know this map is homotopic rel boundary to $\ph \circ g$ for some map
$g:(D^3,S^2) \to Z$.  But how do we construct such a map?  The two preimages of
each 3-cell are quite far away from each other, but we need to join their
boundaries into a single map of degree 2.

It turns out that one way to do this is to take a path from one preimage to the
other and nullhomotope its image through $\widetilde{(T^3)^{(2)}}$, which is
simply connected.  Although this forces the map to become very distorted, the
cellular volume is not affected.  An observation of this sort was originally
made by Brian White in \cite{White}.

In this example, our job is made easier by the fact that the 2-skeleta of the
two complexes happen to coincide.  When that is not the case, forcing the map to
behave correctly on lower skeleta takes some rather intuition-free wrangling.
\begin{lem} \label{lem:riv}
  Let $(X,K)$ be a CW pair, and $n \geq 2$, such that the inclusion $K
  \hookrightarrow X$ is rationally $n$-connected.  Then there are constants
  $p_n(X,K)>0$ and $\kappa_n(X,K)$ such that for any map $f:S^n \to X$ of volume
  $k$ there is an admissible map $g:S^n \to K$ for which $p_nf \simeq g$ and $g$
  has volume $\kappa_nk+\kappa_n$.  If $f$ is admissible, then as cellular
  chains, $g_\#([S^n])=p_nj_n(f_\#([S^n]))$, where $j_n:C_n(X) \to C_n(K)$ is a
  lifting homomorphism.
\end{lem}
\begin{proof}
  By passing to a homotopy equivalent situation, we assume that $X$ and $K$ have
  the same 1-skeleton and that all boundary maps are admissible.  At the cost of
  increasing $\kappa_n$, we assume $f$ is admissible.

  First, fix some notation.  For every $i$ and $r$, fix degree $r$ maps
  $d_r:(D^i,S^{i-1}) \to (D^i,S^{i-1})$ and $d_r:S^i \to S^i$.  In an abuse of
  notation, we distinguish these only by context.  Given maps $a:I \times
  S^{i-1} \to X$ and $b:D^i \to X$ such that $a|_{\{0\} \times S^{i-1}}=
  b|_{\partial D^i}$, define the map $a \vee b:D^i \to X$ to be ``$a$ on the
  outside and $b$ on the inside.''  Finally, let $q_n$ be the period of the
  group $H_n(\tilde X,\tilde K)$, which has bounded exponent by Corollary
  \ref{cor:bo_pi}.

  Suppose first that $n=2$.  Let $f:S^2 \to X^{(2)}$ be an admissible map.
  Since $\pi_1(X) \cong \pi_1(K)$, given a 2-cell $e$ with attaching map
  $\gamma_e:S^1 \to X^{(1)}$, we can extend $\gamma_e$ to a map $h_e:(D^2,S^1)
  \to (K,K^{(1)})$.  Define a map $f^\prime:S^2 \to K^{(2)} \subset X$ which
  agrees with $f$ on $f^{-1}(X^{(1)})$ and where every homeomorphic preimage of a
  cell $e$ is replaced with $h_e$.  Then $f^\prime$ differs from $f$ by a
  torsion element of $\pi_2(X)$.  Therefore $f \circ d_{q_2}$ is homotopic to
  $f^\prime \circ d_{q_2}$.  This shows the lemma for $n=2$, with $p_2=q_2$.

  The proof for $n \geq 3$ is much more complicated and the homotopy we
  construct involves several steps.  On the interval $[0,1/3]$, we homotope $f
  \circ d_{q_n}$ to an admissible map which has degree zero on all $n$-cells of
  $\tilde X \setminus \tilde K$.  On the interval $[1/3,2/3]$, we cancel out
  preimages with opposite orientations; this leaves us with a map $S^n \to
  X^{(n-1)} \cup K$.  Finally, we show that a particular multiple of this map
  can be homotoped into $K$ without changing the volume.  This last homotopy is
  itself quite involved and involves several intermediate maps, whose
  relationships to each other are sketched in Figure \ref{fig:Phi}.

  We take $f:S^n \to \tilde X$ to be a lift of the original map to the universal
  cover.  To begin the first step, let $\alpha \in C_n(\tilde X)$ be the
  cellular cycle $f_\#([S^n])$.  Fix a lifting homomorphism
  $j_n:C_n(\tilde X;\mathbb{Q}) \to C_n(\tilde K;\mathbb{Q})$ and a
  corresponding chain homotopy $u_n:C_n(\tilde X;\mathbb{Q}) \to
  C_{n+1}(\tilde X;\mathbb{Q})$, and choose a listing $\{e_r\}_{1 \leq r \leq V}$
  of the cells of $q_nu_n(\alpha)$, with multiplicity.  We build a homotopy
  $h:\left[0,\frac{1}{3}\right] \times S^n \to \tilde X$ starting with $h_0=f
  \circ d_{q_n}$.  As $t$ increases, we homotope through each cell $e_r$ so that
  $$h_\#([\{t_r\} \times S^n])=h_\#([\{t_{r-1}\} \times S^n])+\partial e_r.$$
  Thus $h_\#\left(\left[\left\{\frac{1}{3}\right\} \times S^n\right]\right)=
  q_nj_n(\alpha)$.  We can also ensure that $h$ is admissible by leaving it
  constant for a time $\epsi$ between homotoping through cells.  Moreover, we
  can assume that in $h_{1/3}$, the closures of preimages of open cells in
  $\tilde X$ are disjoint closed disks; in other words,
  $h_{1/3}^{-1}(\tilde X^{(n-1)})$ is a path-connected compact manifold.

  On the interval $\left[\frac{1}{3},\frac{2}{3}\right]$, we cancel out all
  preimages of cells with opposite orientations; this way, the cellular volume
  of $h_{2/3}$ will be $\lVert q_nj_n(\alpha) \rVert_1$.  Choose two preimages of
  an $n$-cell under $h_{1/3}$ which have opposite orientation; call these
  $a_1,a_2:B^n \to S^n$ and the interior of the cell $B \subset \tilde X$.  We
  can find a simple smooth path $\gamma$ through the interior of
  $h_{1/3}^{-1}(\tilde X^{(n-1)})$ whose endpoints are a point $b_1$ on the
  boundary of $\overline{a_1(B^n)}$ and a point $b_2$ on the boundary of
  $\overline{a_2(B^n)}$ with $f(b_1)=f(b_2)$.  Moreover, $X^{(n-1)}$ is simply
  connected, so $h_{1/3} \circ \gamma$ is a nullhomotopic loop in $X^{(n-1)}$.
  Let $N$ be a tubular neighborhood of $\gamma$ containing a smaller tubular
  neighborhood $N_{1/2}$.  We homotope $h_{1/3}|_N$ in such a way that $N \cap
  h_{1/3+\epsi}^{-1}(B)=N_{1/2}$.  This gives us a disk $D=\overline{a_1(B^n)}
  \cup \overline{N_{1/2}} \cup \overline{a_2(B^n)}$ which maps to our cell with
  degree 0.  Thus we can homotope $h_{1/3+\epsi}|_D$ so that it maps to
  $\tilde X^{(n-1)}$.  After iterating this process for every pair of preimages
  with opposite orientations, we get an $h_{2/3}$ which is admissible and maps
  to each $n$-cell $c$ exactly $\langle q_nj_n(\alpha), c \rangle$ times.

  However, at this point, it's not necessarily true that $h_{2/3}$ maps to
  $\tilde K$.  We have ensured that the image of $h_{2/3}$ is in $\tilde K \cup
  \tilde X^{(n-1)}$; moreover, by homotoping neighborhoods of paths through
  $\tilde X^{(n-1)}$ into $\tilde K^{(1)}$, we can ensure that $h_{2/3}^{-1}(
  \tilde K)$ is connected, and $h_{2/3}^{-1}(\tilde X \setminus \tilde K)$ is
  contained in a ball $B$ with $h_{2/3}(B) \subset \tilde X^{(n-1)}$.  We can
  assume that $B$ is the upper hemisphere and $d_r$ preserves it.
  \begin{prop}
    There is an $r=r(X,K,n)>0$ such that $h_{2/3}|_B \circ d_r$ deforms rel
    boundary to $K^{(n-1)}$ via a homotopy $\Phi:\left[\frac{2}{3},1\right]
    \times B \to X$.
  \end{prop}
  \begin{proof}
    By Corollary \ref{cor:bo_pi}, $\pi_n(\tilde X,\tilde K)$ also has bounded
    exponent, while $\pi_{n+1}(\tilde X,\tilde K) \cong
    H_{n+1}(\tilde X,\tilde K) \mod \mathcal{BE}$.

    \begin{figure}
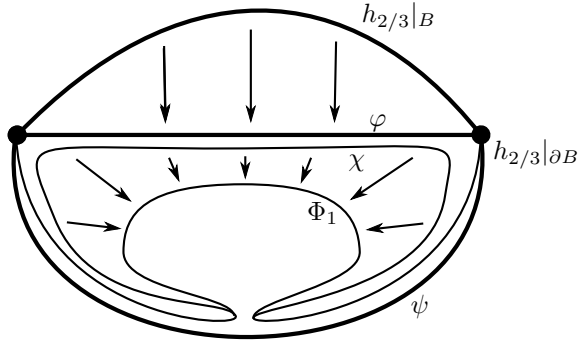

      \begin{center}
        \include{skeleton_lemma}
        \caption{The stages of the homotopy $\Phi$.} \label{fig:Phi}
      \end{center}
    \end{figure}

    Since $h_{2/3}|_B$ defines an element of $\pi_n(\tilde X,\tilde K)$, there's
    some $r_1(X,K,n)$ such that $h_{2/3}|_B \circ d_{r_1}$ deforms rel boundary
    to a map $\ph:B \to K$.

    By Proposition\ \ref{lem:injpi}, the inclusion $\tilde K^{(n-1)}
    \hookrightarrow \tilde X^{(n-1)}$ induces a map on $\pi_{n-1}$ with torsion
    kernel.  Call $\Gamma:=\pi_1(X)$ and consider the diagram
    $$
    \xymatrix{
      \mathbb{Z}\Gamma^{\#\{n\text{-cells of }K\}} \ar[r] \ar@{^{(}->}[d] \ar[rd]^\theta&
      \pi_{n-1}(\tilde K^{(n-1)}) \ar@{->>}[r] \ar[d]^{\iota_*} &
      \pi_{n-1}(\tilde K) \ar[r] \ar[d]^{\wr}_{\text{\rotatebox[origin=c]{-90}{$\mathcal{BE}$}}} & 0 \\
      \mathbb{Z}\Gamma^{\#\{n\text{-cells of }X\}} \ar[r] &
      \pi_{n-1}(\tilde X^{(n-1)}) \ar@{->>}[r] & \pi_{n-1}(\tilde X) \ar[r] & 0
    }
    $$
    induced by the inclusions $K^{(n-1)} \hookrightarrow K$ and $X^{(n-1)}
    \hookrightarrow X$.  From a diagram chase, we get that $\ker \iota_*$ is an
    extension of a bounded exponent group by $A:=
    \mathbb{Z}\Gamma^{\#\{n\text{-cells of }K\}}/\ker\theta$; since $\ker \iota_*$
    is known to be torsion, $A$ must be torsion, and thus by Proposition
    \ref{prop:bo_mod} it has bounded exponent.  Therefore $\ker \iota_*$ itself
    has bounded exponent, and there is an $r_2(X,K,n)$ such that
    $h_{2/3}|_{\partial B} \circ d_{r_2}$ is nullhomotopic in $\tilde K^{(n-1)}$
    via a nullhomotopy $\psi$.

    Together, $\psi \circ d_{r_1}$ and $\ph \circ d_{r_2}$ give us a map
    $$\chi:=\psi \circ d_{r_1} \vee \ph \circ d_{r_2}:S^n \to \tilde K$$
    which is homotopic in $\tilde X$ to
    $$\psi \circ d_{r_1} \vee h_{2/3}|_B \circ d_{r_1} \circ d_{r_2}:S^n \to
    \tilde X^{(n-1)}.$$
    Now, we have the diagram (with decorations mod $\mathcal{BE}$)
    $$
    \xymatrix{
      \pi_{n+1}(\tilde X,\tilde K) \ar[r]
      \ar[d]^{\text{\rotatebox[origin=c]{-90}{$\cong$}}} &
      \pi_n(\tilde K) \ar@{->>}[r]^{\iota_*} \ar[d]^{h_K} &
      \pi_n(\tilde X) \ar[d]^{h_X} \\
      H_{n+1}(\tilde X,\tilde K) \ar[r] & H_n(\tilde K) \ar@{->>}[r] &
      H_n(\tilde X),
    }
    $$
    where the rows are exact.  We know that $h_X(\iota_*[\chi])=0$; on the other
    hand, a diagram chase confirms that for some $r_3(X,K,n)$, there is an
    $\alpha \in \pi_n(\tilde K)$ with $\iota_*\alpha=0$ and $h_K(\alpha)=
    r_3h_K[\chi]$.  By Prop.\ \ref{prop:i-1}, $r_3[\chi]-\alpha$ has a
    representative $\hat\chi$ whose image is in $\tilde K^{(n-1)}$.  In
    particular, $\hat\chi \simeq \chi \circ d_{r_3}$ in $\tilde X$.

    Set $r=r_1r_2r_3$ and let $\Phi$ be a homotopy that takes $h_{2/3}|_B \circ
    d_r$ to $\psi \circ d_{r_1} \circ d_{r_3} \vee \chi \circ d_{r_3}$ and then
    deforms $\chi \circ d_{r_3}$ to $\hat\chi$.
  \end{proof}
  By precomposing with $d_r$ on the interval $[0,\frac{2}{3}]$, we then get a
  homotopy $H:[0,1] \times S^{n-1} \to \tilde X$ with
  $$H(t,x)=\left\{
  \begin{array}{l l}
    h_t(d_r(x))&\mbox{if }t \in \left[0,\frac{2}{3}\right]\\
    h_{2/3}(d_r(x))&\mbox{if }t \in \left[\frac{2}{3},1\right]\mbox{ and }x \in
    S^{n-1} \setminus B \\
    \Phi(x)&\mbox{if }t \in \left[\frac{2}{3},1\right]\mbox{ and }x \in B
  \end{array}\right.$$
  and $H_1$ lands in $\tilde K$ with $H_1([S^n])=rq_nj_n(\alpha)$, is admissible
  and has no cells of opposite orientations.  In particular, if we let $t_n=
  \max\{|j_n(e)|: e\text{ is an $n$-cell of }X\}$, $\vol g \leq rq_nt_n\vol f$.
  Thus we can set $p_n=rq_n$, and in our homotopy equivalent setup $\kappa_n=
  rq_nt_n$ and $g=H_1$.  The homotopy equivalence may impose a penalty on the
  constant.
\end{proof}
More generally, if a map $\ph:Y \to X$ obeys the same conditions on homotopy,
then using the mapping cylinder, we can prove that for any $f:S^n \to X$, $rf$
lifts to a map of volume $Ck$ for constants $C(\ph,n)$ and $r(\ph,n)$.  As a
corollary, we have Theorem A.
\begin{thm}[Rational invariance of volume distortion] \label{thm:QInv}
  Let $X$ and $Y$ be rationally $n$-equivalent finite complexes, with
  $n$-equivalences $X \xrightarrow{f} Z \xleftarrow{g} Y$.  Then for any $\alpha
  \in \pi_n(X)$ and $\beta \in \pi_n(Y)$ such that $f_*\alpha=g_*\beta$,
  $\VD_\alpha \sim \VD_\beta$.  Similarly, for any finite-dimensional $V
  \subseteq \pi_n(X) \otimes \mathbb{Q}$ and $W \subseteq \pi_n(Y) \otimes
  \mathbb{Q}$ which are sent to the same subspace of $\pi_n(Z) \otimes
  \mathbb{Q}$, $\VD_V \sim \VD_W$.
\end{thm}
\begin{proof}
  By Corollary \ref{lem:QFin}, we can assume that $Z$ is also finite; it is not
  necessarily the case then that $f_*\alpha=g_*\beta$, but certainly this is
  true for some multiple of $\alpha$ and $\beta$.

  For any $k$,
  $$\lvert kf_*\alpha \rvert_{\vol} \leq (\Lip f)^k\lvert k\alpha\rvert_{\vol}.$$
  Conversely, by Lemma \ref{lem:riv}, there are a $p_n(X,Z)$ and a
  $\kappa_n(X,Z)$ such that
  $$\lvert p_nk\alpha \rvert_{\vol} \leq \kappa_n\lvert kf_*\alpha \rvert_{\vol}+
  \kappa_n.$$
  Therefore the distortion functions of $\alpha$ and $f_*\alpha$ are
  asymptotically equivalent.  The same holds for $\beta$ and $g_*\beta$.  The
  proof of the subspace case follows.
\end{proof}

This fact allows us to ignore torsion information when studying volume
distortion functions.  Indeed, in our subsequent discussion, we will often speak
of spaces and maps ``up to rational equivalence''.  However, we conjecture that
such an approach is not sufficient for studying Lipschitz distortion.
\begin{ex} \label{ex:Lip}
  Let $X$ be the total space of a fibration $S^3 \to X \to T^4$ with Euler class
  equal to the fundamental class $[T^4]$.  Define a map $f_k:S^3 \to
  \widetilde{T^4}$ which sends $S^3$ to the sides of a $4$-cube with side
  length $k$.  This map is $Ck$-lipschitz, for $C$ independent of $k$, and as
  will be discussed it lifts to a $Ck$-lipschitz map $g_k:S^3 \to X$ which is a
  representative of $k^4\alpha$, where $\alpha$ is a generator of $\pi_3(X)$.
  Thus $\alpha$ is lipschitz distorted in $X$ with $\LD_\alpha(k) \gtrsim k^4$.

  On the other hand, let $Z=(T^3)^{(2)} \cup_{a \circ u} D^3$ be the space
  mentioned above which has a rational equivalence $i:Z \to T^3$.  Then the map
  $\id \times i:S^1 \times Z \to T^4$ is also a rational equivalence.  Thus we
  can define a pullback fibration with total space $Y=(\id \times i)^*X$, and
  $Y \to X$ is also a rational equivalence.  On the other hand, suppose that
  $\LD_\alpha(k) \gtrsim k^4$.  It is possible to show that this is equivalent
  to the existence of admissible maps $f_k:(D^4,S^3) \to \tilde Z$ of volume
  $k^4$ with admissible boundary such that $\Lip(f_k|_{S^3}) \lesssim k$.  By
  composing $f_k|_{S^3}$ with a projection onto $\tilde Z$, we get a sequence of
  $Ck$-Lipschitz admissible maps $D^3 \to \tilde Z$ with area $C^\prime k^3$.
  Although we don't know a proof, we suspect that such maps do not exist.

  This would mean that $Y$ and $X$ have asymptotically different Lipschitz
  distortion functions.  In other words, torsion matters for Lipschitz
  distortion, at least for spaces with a nontrivial fundamental group.
\end{ex}

\section{Sources of volume distortion}
In this section, we discuss two ways volume distortion can be induced in finite
complexes: homological triviality and actions by the fundamental group.  The
ultimate goal of this is to show that conditions (1) and (2) of Theorem
\ref{thmA} are necessary, but we will allow ourselves various detours along the
way to explore these phenomena in greater detail.

Both of these phenomena induce a particularly coarse kind of distortion.
Throughout the section, let $X$ be a finite CW complex with universal cover
$\tilde X$ and fundamental group $\Gamma$.  Restating the previous section's
definition, a class $\alpha$ or finite-dimensional subspace $V$ in $\pi_n(X)
\otimes \mathbb{Q}$ is \emph{infinitely distorted} if there is some constant $C$
such that there are arbitrarily large $k$ (respectively, arbitrarily large
vectors $\vec v \in V$) such that $|k\alpha|_{\vol} \leq C$ (respectively,
$|\vec v|_{\vol} \leq C$.)  As discussed in the previous section, this does not
necessarily mean that the volume of all multiples of $\alpha$ is uniformly
bounded.  Thus, by defining distortion the way we do rather than simply studying
the function $f(k)=\lvert k\alpha \rvert_{\vol}$, we are obscuring a certain
amount of complexity.

Lipschitz distortion in simply connected spaces is fairly subtle.  It has been
previously studied by Gromov in \cite{GrDil}, \cite{GrQHT}, and \cite{GrMS}.  On
the other hand, volume distortion is trivial in the simply connected case: all
elements either have all multiples with uniformly bounded volume, or are
undistorted, depending on whether they are homologically trivial.

Recall that we use the notation $h_n:\pi_n(X) \otimes \mathbb{Q} \to
H_n(X;\mathbb{Q})$ to refer to the rational Hurewicz homomorphism. 
\begin{thm} \label{thm:sc}
  Suppose $X$ is a simply connected complex and $\alpha \in \pi_n(X) \otimes
  \mathbb{Q}$.  Then $k\alpha$ has a representative of zero volume for some $k$
  if and only if $h_n(\alpha)=0$; otherwise $\alpha$ is undistorted.
\end{thm}
\begin{proof}
  If $h_n(\alpha) \neq 0$, then there is a cohomology class $\omega \in H^n(X;
  \mathbb{Q})$ which pairs nontrivially with any representative of $k\alpha$,
  giving some $kC$.  Thus $k\alpha$ cannot have zero volume and in fact the
  cellular volume
  $$\vol_C(k\alpha) \geq kC/\min_{n\text{-cells $c$ of }X} \omega(c).$$
  In particular, $\alpha$ is undistorted.

  On the other hand, if $h_n(\alpha)=0$ then some $k\alpha$ has a representative
  of zero volume by Proposition \ref{prop:i-1}.
\end{proof}
We can make the same observation for piecewise Riemannian complexes using
differential forms and the metric definition of volume.   Indeed, we can apply
it even to non-compact spaces by requiring our differential forms to be bounded.
This prefigures the use of $L_\infty$ cohomology later in the paper.  For the
moment, however, we restrict ourselves to exploring how this relates distortion
to the homology of covering spaces.
\begin{defn}
  Given a $k$-form $\omega$ on a piecewise Riemannian space $(X,g)$, let
  $\lVert\omega\rVert_\infty=\sup \langle \omega, (v_1,\ldots,v_k) \rangle$,
  where the supremum ranges over all orthonormal frames $(v_1,\ldots,v_k)$.  We
  say $\omega$ is \emph{bounded} if $\lVert \omega \rVert_\infty<\infty$.  It is
  clear that a form on a compact space is bounded, and that the boundedness of
  a form on $\tilde X$ with respect to a Riemannian metric lifted from a compact
  $X$ is independent of the metric.  We may therefore speak of bounded forms on
  $X$ and $\tilde X$ without specifying a metric.
\end{defn}
\begin{lem} \label{lem:duh}
  If $\alpha \in \pi_n(X)$ and there is a bounded closed form $\omega \in
  \Omega^n(\tilde X)$ such that $\langle \omega,\alpha \rangle=c \neq 0$, then
  $\alpha$ is volume-undistorted.
\end{lem}
\begin{proof}
  If $f:S^n \to X$ is a map representing $k\alpha$, then $\langle \omega, f
  \rangle=ck$, and so $\vol f \geq ck/\lVert \omega \rVert_\infty \sim k$.
\end{proof}
In particular, if a class is has a nonzero image under the Hurewicz homomorphism
in some finite cover of $X$, then this condition is satisfied.  One might even
hope---as it turns out in Example \ref{ex:BS}, incorrectly---that this is the
only way such bounded cocycles can be generated.
\begin{cor} \label{cor:cover}
  Let $\alpha \in \pi_n(X)$.  If for some finite cover $\ph:Y \to X$ and lift
  $\tilde\alpha \in \pi_n(Y)$ of $\alpha$, $0 \neq h_n(\tilde\alpha) \in
  H_n(Y)$, then $\alpha$ is volume-undistorted.
\end{cor}
\begin{proof}
  Let $\pi:\tilde X \to X$ be the universal cover.  It is enough to show that
  there is a bounded form $\omega$ on $\tilde X$ and lift $\tilde{\tilde\alpha}$
  of $\tilde\alpha$ to $\tilde X$ such that $\langle \omega,
  h_n(\tilde{\tilde\alpha}) \rangle \neq 0$.  We can define a form on $Y$ which
  integrates on cells to the cellular cochain which is cell-wise dual to
  $h_n(\tilde\alpha)$.  The pullback of this form to $\tilde X$ is the form we
  are looking for.
\end{proof}
We have an example in which this applies nontrivially:
\begin{ex}
  Let $A=S^2 \vee S^1 \vee S^1$, with $x$, $a$, and $b$ denoting the identity
  maps on the three spheres and $\cdot$ denoting the action of $\pi_1A$ on
  $\pi_2A$.  For some function $f \simeq x+a \cdot x+b \cdot x$, set $X=A \cup_f
  D^3$.  Then in $H_2(X)$, $x$ represents an element of order 3, and so
  $h_2([x])=0$.

  On the other hand, let $Y$ be the threefold cover of $X$ corresponding to the
  normal subgroup $\lpair a^3, ab \rpair$; this has three 2-cells which we may
  call $x_1$, $x_2$, and $x_3$.  In the cellular homology of $Y$, the relation
  induced by each lift of the 3-cell is $x_1+x_2+x_3=0$; therefore the $x_i$ are
  homologically nontrivial, and $x$ is undistorted by Corollary \ref{cor:cover}.
\end{ex}
But also an example demonstrating that it is not a necessary condition:
\begin{ex} \label{ex:BS}
  According to \cite{Mesk}, the Baumslag-Solitar group $BS(2,3)=\langle a,b \mid
  ab^2a^{-1}b^{-3} \rangle$ is not residually finite; in particular every
  surjection onto a finite group sends $g:=[a^b,a] \mapsto 1$.  Moreover, $g$
  and $b$ generate a free subgroup of $BS(2,3)$.  Let $X$ be a 2-complex with
  $\pi_1X=BS(2,3)$, and let $Y=(X \vee S^2) \cup_f D^3$, for some $f$ such that
  $[f]=y+g \cdot y+b \cdot y$, where $y$ is the generator of $\pi_2(S^2)$.  To
  show that $y$ is undistorted in $Y$, we need to find a cellular cocycle in
  $H^2(\tilde Y;\mathbb{Q})$ which is nonzero on some lift $\tilde y$ of $y$.
  By the argument in the previous example, such a cocycle can be defined on the
  cells $\langle g,b \rangle \cdot \tilde y$; we can extend it to the other
  cosets of $\langle g,b \rangle$ by defining it to be zero there.

  On the other hand, in any finite cover, $y$ does not have any nontrivial lifts
  from the point of view of rational homology.  Indeed, suppose $\ph:Z \to Y$ is
  a finite cover and $\tilde y$ any lift of $y$ to $Z$.  Then the disk attached
  via $f$ induces the relation $(2+b)h_2(\tilde y)=0$.  Moreover, there's an $r$
  for which $b^r$ acts trivially on $H_2(Z)$, so that
  $$0=(2^r-(-b)^r) h_2(\tilde y)=(2^r+(-1)^{r-1})h_2(\tilde y).$$
  Thus $\tilde y$ is sent via the Hurewicz map to a torsion element of $H_2(Z)$.
\end{ex}

\subsection*{Spaces with injective Hurewicz homomorphisms}
For the time being, we will expand our discussion to possibly infinite spaces
with finite skeleta.  Although these are not compact, distortion is still
well-defined if we stick to admissible maps, independent of the definition of
volume used, or indeed maps $S^n \to X^{(N)}$ for any fixed $N(n)$.  In all
these situations, it is still the case that if $\alpha \in \pi_n(X)$ goes to 0
under the Hurewicz homomorphism, then its volume is zero.  Therefore, in order
for a space to have no distortion, its rational Hurewicz homomorphisms must all
be injective.  Here we classify spaces with this property.
\begin{prop} \label{prop:PiQ}
  Let $X$ be any simply-connected CW complex.  Then the rational Hurewicz
  homomorphism $h_n:\pi_n(X) \otimes \mathbb{Q} \to H_n(X;\mathbb{Q})$ is
  injective for all $n \geq 2$ if and only if $X$ is rationally equivalent to a
  product of Eilenberg--MacLane spaces.
\end{prop}
\begin{proof}
  If $X$ is a $K(G,n)$, then $\pi_n(X) \to H_n(X)$ is an isomorphism and all
  other homotopy groups vanish.  This injectivity is preserved under arbitrary
  products, so the backwards implication is true.

  For the converse, we note that every simply-connected complex can be
  rationalized; that is, it has a rational equivalence to a complex whose
  reduced homotopy and homology groups are rational vector spaces.  One may for
  example see Chapter 7 of \cite{GrMo}.  Thus it is enough to prove this for
  rational spaces.

  So let $X$ be a rational space, and consider a step of its Postnikov tower,
  $K(\pi,n) \to X_{(n)} \to X_{(n-1)}$; $X_{(n)}$ and $X_{(n-1)}$ are also rational
  spaces.  If this fibration is not a product, then there is a nonzero
  $k$-invariant $k \in H^{n+1}(X_{(n-1)};\pi)$ giving the obstruction to
  extending a section $(X_{(n-1)})^{(n)} \to X_{(n)}$.  Choose a cellular cycle
  $c \in C_{n+1}(X_{(n-1)})$ which pairs nontrivially with $k$, and let $f:S^n
  \to X_{(n-1)}$ be the sum of the boundary maps of the cells in $c$.  Then in
  our partial section, $f$ lifts to a representative $\tilde f:S^n \to 
  (X_{(n)})^{(n)}$ of $\langle k,c \rangle \in \pi_n(X)$.  Since $c$ is a cycle,
  $f$ has degree 0 on each $n$-cell; thus $[\tilde f]$ is also homologically
  trivial.

  Therefore, if the Hurewicz homomorphism is injective, every stage of the
  Postnikov tower of $X$ must be a product.
\end{proof}
Now, for even $n$, $K(\mathbb{Q},n)$ has an infinite number of nonzero
cohomology groups; for odd $n$, $K(\mathbb{Q},n) \simeq_{\mathbb{Q}} S^n$.
Moreover, an infinite product or a $K(V,n)$ where $V$ is an infinite-dimensional
rational vector space will always have infinitely many nonzero cohomology
groups.  This gives us
\begin{cor} \label{cor:fin}
  If $X$ is a finite complex with no distortion in its homotopy groups, then
  $\tilde X$ is rationally equivalent to a finite product of odd-dimensional
  spheres.
\end{cor}
We will show later that this condition is equivalent to condition (1) of
Theorems \ref{thmA} and \ref{thmB}.

One could hope that with a result as powerful as Proposition \ref{prop:PiQ}, we
could proceed through a classification like that of Theorem \ref{thmA} for
spaces with finite skeleta.  However, the rest of our analysis relies heavily on
the fact that the rational homotopy groups of our spaces are finite-dimensional;
dealing with larger homotopy groups would require radically new techniques.  How
could one hope to carry this analysis through to infinite complexes with finite
skeleta?  Consider the following restatement of Corollary \ref{cor:fin}: if $X$
is finite, then if one of its rational homotopy groups is infinite-dimensional,
some higher homotopy group has homologically trivial rational elements.  One
approach would be to try to demonstrate that this still holds for complexes with
finite skeleta.  The following example demonstrates that such an approach
cannot work: there is an undistorted space with finite skeleta whose universal
cover is homotopy equivalent to $(S^3)^\infty$.  This means that condition (1)
of Theorem \ref{thmA} cannot hold true in the same form for such complexes.
\begin{ex} \label{rickard}
  Consider Thompson's group $F$.  Brown and Geoghegan \cite{BG_F} give a
  $K(F,1)$ with two cells in each dimension; call this space $X$.  The group $F$
  acts transitively on the set $A$ of dyadic rationals strictly between 0 and 1,
  as a subgroup of $\Homeo_+([0,1])$; the stabilizer $S$ of a point $a \in A$
  under this action is isomorphic to $F \times F$, corresponding to
  homeomorphisms that fix $[a,1]$ and those that fix $[0,a]$.  These two copies
  of $F$ are generated by four elements; call these $g_1,g_2,g_3,g_4$.  Then as
  a module,
  $$\mathbb{Q}[A]=\mathbb{Q}F/\langle g_i-1: i=1,2,3,4\rangle.$$
  Thus we can attach four 4-cells to $X \vee S^3$ along $g_i \cdot \id_{S^3}
  -\id_{S^3}$, to get a space $Y_4$ with finite skeleta and with $\pi_3(Y) \cong
  \mathbb{Q}[A]$.  One may think of $Y_4$ as
  $$X \cup_{K(S,1)^{(1)}} (K(S,1)^{(1)} \times S^3).$$
  Continuing to add four cells in every dimensions gives us a space $Y=X
  \cup_{K(S,1)} (K(S,1) \times S^3)$, whose universal cover is homotopy
  equivalent to $\bigvee_{a \in A} S^3$.

  We now require a second induction to embed $Y$ in a space whose universal
  cover is homotopy equivalent to $\prod_{a \in A} S^3$.  In general, $F$ acts
  transitively on unordered $n$-element subsets of $A$ with stabilizer
  $F^{n+1}$.  Suppose, then, that we have constructed a space $Z_{n-1}$ such that
  $\tilde Z_{n-1}$ is homotopy equivalent to $K^{(3(n-1))}$, where $K \simeq
  \prod_{a \in A} S^3$ is a complex with one $3k$-cell corresponding to each
  $k$-tuple of elements in $A$.  To construct $Z_n$, we add a single $3n$-cell
  corresponding to an $n$-tuple of elements of $A$; then proceed as in the
  construction of $Y$, using $2(n+1)$ $(3n+1)$-cells to identify those of its
  translates which correspond to the same $n$-tuple, and so on with $2(n+1)$
  cells in each subsequent dimension.

  At the end of this induction, we have a complex $Z$ with $O(n^2)$ cells in
  dimension $n$ whose universal cover is homotopy equivalent to $\prod_{a \in A}
  S^3$.  Moreover, the $S^3$ we attached originally is homologically nontrivial
  in $Z$; therefore $\pi_3(Z)$ is undistorted.
\end{ex}

\subsection*{Distortion via monodromy}
Another potential source of infinite distortion is the action by the fundamental
group on $\pi_n(X)$: volume is preserved under this action, but a norm on a
finite-dimensional subspace of $\pi_n(X) \otimes \mathbb{Q}$ need not be.  The
most basic example is when $X$ is the mapping torus of a degree 2 map on $S^2$;
here $\pi_2(X) \cong \mathbb{Z}[\frac{1}{2}]$, and so $\pi_2(X) \otimes
\mathbb{Q} \cong \mathbb{Q}$.  Let $\alpha=[\id_{S^2}] \in \pi_2(X)$; then
$2^k\alpha$ has a cellular representative of volume 1 for every integer $k$, and
indeed any integer multiple $k\alpha$ has a representative of volume
$\leq \log_2 k$.

More generally, let $\alpha \in \pi_n(X)$, $\gamma \in \Gamma$, and suppose that
the $\mathbb{Q}[\mathbb{Z}]$-module generated by $\gamma^i \cdot \alpha$ is an
$m$-dimensional $\mathbb{Q}$-vector space $V$ for some finite $m$.  Then
$\gamma$ acts on $V$ via a linear transformation $T \in GL(V)$.  Then either $T$
is conjugate to an element of $O(m,\mathbb{R})$, or there is some vector
$\vec v$ such that $T^k(\vec v)$ increases without bound and thus $V$ is
infinitely distorted.  A similar dichotomy holds if $\mathbb{Q}\Gamma \cdot
\alpha$ is a finite-dimensional submodule of $\pi_n(X) \otimes \mathbb{Q}$:
either $\gamma \cdot \alpha$ gets arbitrarily large, or the entire submodule
conjugates into $O(m,\mathbb{R})$.
\begin{defn}
  We say a transformation $T \in GL(m,\mathbb{Q})$, or more generally a
  representation $\rho:\Gamma \to GL(m,\mathbb{Q})$, is \emph{elliptic} if it
  preserves a norm on $\mathbb{Q}^m$, or equivalently if as a representation
  over $\mathbb{R}$ it conjugates into $O(m,\mathbb{R})$.
\end{defn}
\begin{prop}
  Any representation $\rho:\Gamma \to GL(m,\mathbb{Q})$ which is bounded in
  operator norm is elliptic.
\end{prop}
\begin{proof}
  Let $K \subset GL(m,\mathbb{R})$ be the closure of $\rho(\Gamma)$; this is a
  closed subgroup and hence a compact Lie group.  Averaging any norm over $K$
  gives us a $\rho$-invariant norm on $\mathbb{Q}^m$.
\end{proof}
\begin{cor} \label{cor:md}
  For any finite complex $X$ such that $\pi_nX \otimes \mathbb{Q}$ is
  finite-dimensional, either $\pi_nX \otimes \mathbb{Q}$ is infinitely
  distorted, or the monodromy representation $\rho:\pi_1X \to GL(\pi_nX \otimes
  \mathbb{Q})$ is elliptic.
\end{cor}
In other words, we have just shown the necessity of condition (2) of Theorems
\ref{thmA} and \ref{thmB}.  The reader who is impatient to see the proofs of
these theorems is invited to skip forward to Section 4.   In the rest of this
section, we take a closer look at how the monodromy action of an individual
element of $\Gamma$ affects the distortion function of individual elements of
finite-dimensional subspaces of $\pi_n(X) \otimes \mathbb{Q}$; in particular, we
will show that such distortion functions need not be infinite.  We also
demonstrate distinctions between distortion functions of elements and subspaces,
and between weakly infinite and infinite distortion.

This case is relatively easy to analyze because we can use the fact that
$\mathbb{Q}[\mathbb{Z}]$ is a PID.  Thus, given a general $X$ with $\pi_1X=
\Gamma$ and a $\gamma \in \Gamma$, consider an finite-dimensional
$\mathbb{Q}[\gamma,\gamma^{-1}]$-submodule $V \subseteq \pi_n(X) \otimes
\mathbb{Q}$.  Then $V$ decomposes as
$$V=\oplus_{i=1}^r \mathbb{Q}[\gamma,\gamma^{-1}]/p_i(\gamma)^{q_i}=:\oplus_{i=1}^r
V_i,$$
where the $p_i$ are irreducible factors of the characteristic polynomial of the
action $T \in GL(V)$ of $\gamma$ on $V$.  The distortion we have described
depends on finding ways of writing vectors in $V$ as $\sum_{i=-\ell}^\ell
T^i\vec v_i$, where the $\vec v_i$ are in the lattice generated by some
generating set whose precise nature is irrelevant.  Thus we can assume that
generating set contains a basis for each $V_i$, and we can classify the possible
distortion effects by looking at irreducible actions.

Of course, in a completely general setting, we can only talk of the contribution
of monodromy to the distortion of an element.  But there is a simple way of
constructing examples of spaces $X$ for which all distortion in $\pi_nX$ is
caused by monodromy by a single element.  Given an integer $m \times m$ matrix
$A$, we can construct the mapping torus $X_{A,n}$ of a map $\bigvee_m S^n \to
\bigvee_m S^n$ which is modeled on it.  If $A$ has nonzero determinant, then it
corresponds to the monodromy action of the generator of $\pi_1X_{A,n} \cong
\mathbb{Z}$ on $\pi_nX_{A,n} \otimes \mathbb{Q} \cong \mathbb{Q}^m$.  What's
more, since all admissible maps $S^n \to X_{A,n}$ are combinations of translates
of the generators, the monodromy is the only source of distortion for elements
of $\pi_nX_{A,n}$.

More generally, given an element $B=\frac{1}{q}A \in GL(m,\mathbb{Q})$, where
$A$ is again an integer matrix, we can construct a similar ``mapping torus''
$X_{B,n}$, with one cell added to $\bigvee_m S^n \vee S^1$ for each column of $A$
which homotopes $q$ times the corresponding sphere to a map corresponding to
this column.  For example, in the simplest case when $m=1$, we can construct a
space $X_{(5/3),n}$ by gluing both sides of $S^n \times I$ to a copy of $S^n$,
the left side with degree 3 and the right side with degree 5.  Here again, the
monodromy determines the distortion.

First, we consider some examples of the form suggested above and see that the
distortion of a subspace need not correspond to that of any of its elements.
\begin{exs} \label{exs:oddD}
  \begin{enumerate}
  \item Consider the space $X_{A,n}$ where $A$ is the companion matrix of the
    polynomial $x^4-2x^3-2x+1$.  This polynomial is irreducible over
    $\mathbb{Q}$ and has four distinct complex roots, two of which are on the
    unit circle and two of which are real.  Call these roots $\xi$, $\bar\xi$,
    $\eta$ and $\eta^{-1}$, and let $\vec u_\xi \in \mathbb{C}^4$, etc., be the
    corresponding eigenbasis.  Fix $0 \neq \vec v \in \pi_n(X) \otimes
    \mathbb{Q}$; then the coordinates of $\vec v$ in each of these basis vectors
    are nonzero.  On the other hand, for any admissible map $f:S^n \to X$ of
    volume $1$, the element $[f]=T^ke_i \in \pi_n(X) \otimes \mathbb{Q}$ has
    $\vec u_\xi$- and $\vec u_{\bar\xi}$-coordinates at most $1$.  Therefore, any
    admissible map representing $k\vec v$ has to have volume proportional to
    $k$.  The deformation theorem then implies that $\vec v$ admits neither
    volume nor Lipschitz distortion in $X$.
  \item Consider the space $X_{(5/3),3}$ constructed above.  One can think of
    this as the ``mapping torus'' of multiplication by $5/3$.  Note that
    $\pi_n(X) \otimes \mathbb{Q} \cong \mathbb{Q}$ is infinitely distorted,
    since $(5/3)^k$ can be represented with volume 1.  However, any given
    element in $\pi_n(X)$ is \emph{not} infinitely distorted, since for any $V$
    there is a finite number of integers that can be expressed as the sum of $V$
    powers of $5/3$.  So a one-dimensional subspace need not have the same
    distortion function as one of its elements!
  \item \label{qqu}This example demonstrates that a subspace which is weakly
    infinitely distorted is not necessarily infinitely distorted.  Consider the
    space $X_{B,n}$ for $B=\begin{pmatrix}A&0\\I&A\end{pmatrix}$, where $A=
    \begin{pmatrix}3/5&-4/5\\4/5&3/5\end{pmatrix}$.  Then for $\vec v, \vec w
    \in \mathbb{R}^2$ the vector corresponding to $(\vec v,\vec w)$ shifted by
    $k$ is
    $$B^k\begin{pmatrix}\vec v\\\vec w\end{pmatrix}=
    \begin{pmatrix}A^k\vec v\\kA^{k-1}\vec v+A^k\vec w\end{pmatrix}.$$
    This indicates in particular that the subspace generated by the first two
    basis vectors is undistorted.  On the other hand, letting $\alpha$ be the
    irrational angle of rotation of $A$, we get
    $$(B^k-B^{-k})\begin{pmatrix}\vec v\\\vec 0\end{pmatrix}=
    \begin{pmatrix}2\cos(k\alpha)\vec v\\2k\sin[(k-1)\alpha]A^{-1}\vec v
    \end{pmatrix}.$$
    In two-thirds of cases, $\sin(k-1)\alpha>1/2$, and so $\lVert(B^k-B^{-k})
    \vec v\rVert \geq k\lVert\vec v\rVert$.  Thus $\pi_n(X)$ is infinitely
    distorted, and indeed any individual rational vector in $\langle\vec e_3,
    \vec e_4\rangle$ is weakly infinitely distorted.  In contrast,
    \begin{lem} \label{lem:oddD}
      For any $k$, elements of $\langle\vec e_3,\vec e_4\rangle$ with cellular
      volume at most $k$ have length bounded by some $L(k)$.
    \end{lem}
    To prove this, we need a purely algebraic lemma whose precise statement and
    proof are relegated to an appendix, but which generalizes the following
    observation.  Images of integer lattice points under $A$ are only themselves
    lattice points 1/5 of the time; $\vec v$ and $A\vec v$ are both integer
    lattice points if and only if $\vec v$ is in the lattice generated by
    $\begin{pmatrix}-2\\1\end{pmatrix}$ and $\begin{pmatrix}1\\2\end{pmatrix}$,
    so the length of such a vector is $\sqrt{5z}$ for some integer $z$.  More
    generally, if $A^i\vec v$ and $A^j\vec v$ are both integer points, then
    $\lVert\vec v\rVert^2 \in 5^{|i-j|}\mathbb{Z}$.  Similarly, if $A^i\vec u=
    A^j\vec v$ and $\vec u$ and $\vec v$ are both integer vectors, then they
    must be zero or exponentially large in $|i-j|$.  For this proof, we require
    a stronger statement involving linear combinations of $A^j$ for various $j$.
    \begin{proof}
      To see this, we induct on $k$.  For $k=1$, the options are $B^i\vec e_j$
      for $i \in \mathbb{Z}$ and $j=3,4$, which all have length 1.

      Now fix a $V>1$, and take a particular linear combination
      $$\vec u=B^{t_0}\vec u_0+B^{t_1}\vec u_1+\cdots+B^{t_r}\vec u_r$$
      with $\sum_i |\vec u_i| \leq V$ and the $u_i=\begin{pmatrix}\vec v_i\\
      \vec w_i\end{pmatrix}$ are integer vectors.  If $\vec u \in
      \langle\vec e_3,\vec e_4\rangle$, then $\sum_i A^{t_i}\vec v_i=\vec 0$.  We
      can assume that the $t_i$ are increasing and $t_0=0$ since multiplying by
      $B^n$ doesn't change the length of a vector in $\langle\vec e_3,
      \vec e_4\rangle$.  Moreover, we can assume that all the $\vec w_i$ are
      zero, since they contribute at most a linear amount of total length.

      Now, either (1) $\sum_{i=0}^\ell A^{t_i}\vec v_i \neq \vec 0$ for any $\ell<
      r$, or (2) we know that $\lVert u\rVert\leq L(k_1)+L(k_2)$ for some $k_1+
      k_2=k$.  In case (1), by Lemma \ref{lem:NT}, $t_k \leq kf(k)$ for some
      $f(k) \sim \log k$.  Therefore there is a finite number of choices of
      $\vec u$ satisfying (1), and their lengths are bounded by some number
      $L^\prime(k)$.  We can thus set
      $$L(k)=\max\{L^\prime(k),L(k-1)+L(1),L(k-2)+L(2),\ldots\},$$
      completing the proof.
    \end{proof}
    The logarithmic bound given in Lemma \ref{lem:NT} implies that a map of
    volume $k$ represents a vector of length $O((k\log k)^2)$.  With some more
    painstaking accounting, it should be possible to bring this bound done to a
    quadratic one.

    Since, e.g., $\vec e_3$ is weakly infinitely distorted, it is in particular
    distorted.  To find an estimate from below, we construct a sequence of maps
    of volume $O(k)$ representing $5k^2\vec e_3$.  This will be given as a sum
    of powers of $B$ applied to a certain sequence of length $k$ vectors; a
    cancellation trick will ensure that at each step we add $O(k)$ times
    $\vec e_3$ but only $O(1)$ volume.  Specifically, let $\vec v_0=
    \begin{pmatrix}5k\\0\end{pmatrix}$, and for $1 \leq i \leq k$ let $\vec v_i$
    be a point such that $\vec v_i$ and $A^{-1}\vec v_i$ are both lattice
    points and such that $\left\lVert\sum_{j=1}^i \vec A^jv_j-
    i\vec v_0\right\rVert<\sqrt{5}$.  Then
    $$\vec u=\sum_{i=0}^{k-1} B^{i+1}\begin{pmatrix}\vec v_i\\\vec 0\end{pmatrix}
    -B^i\begin{pmatrix}A\vec v_i\\\vec 0\end{pmatrix}=\sum_{i=0}^{k-1}
    \begin{pmatrix}\vec 0\\A^i\vec v_i\end{pmatrix}$$
    can be represented by a map $S^n \to X$ of volume $O(k)$, since
    $\lVert\vec v_i-A\vec v_{i+1}\rVert \leq 2\sqrt{5}$; on the other hand,
    since the distance between $\vec u$ and $5k^2\vec e_3$ is at most
    $\sqrt{5}$, we can convert the one into the other by adding correction terms
    of the form $B^i(p\vec e_3+q\vec e_4)$ and length at most $2\sqrt{5}$ for
    each $0 \leq i \leq k-1$.  Thus $5k^2\vec e_3$ can be represented with
    volume $O(k)$.  This gives us our lower bound on distortion, and so $k^2
    \lesssim \VD_{\vec e_3}(k) \lesssim (k\log k)^2$.
  \end{enumerate}
\end{exs}
However, all of the examples of undistorted elements depend on the
characteristic polynomial $p_i$ having roots on the unit circle.  If all the
roots of $p_i$ are off the unit circle, then in fact all elements of $V$ are at
least exponentially distorted.
\begin{lem} \label{lem:nonuni}
  Suppose that $\gamma \in \pi_n(X)$ and that $V \subseteq \pi_n(X) \otimes
  \mathbb{Q}$ is a finite-dimensional irreducible
  $\mathbb{Q}[\gamma,\gamma^{-1}]$-submodule $V \cong \mathbb{Q}[\mathbb{Z}]/
  (p_i)$, and all complex roots $\lambda_j$ of $p_i$ have $|\lambda_j| \neq 1$.
  Then any $\alpha \in V_i$ has distortion function $\delta_\alpha(k) \gtrsim
  \exp k$.
\end{lem}
\begin{proof}
  Let $T:V \to V$ be the transformation induced by the action of $\gamma$.  Our
  general strategy will be to express vectors as a sum of logarithmically many
  terms of the form $T^iv_i$, where all the $v_i$ are lattice points within a
  fixed ball.  To do this, we first need to find an appropriate lattice.

  Let $\lambda_1,\ldots,\lambda_r$ be the eigenvalues of $T$ with multiplicity,
  with $|\lambda_1|,\ldots,|\lambda_s|<1$ and $|\lambda_{s+1}|,\ldots,|\lambda_r|
  >1$.  Let $v_1,\ldots,v_r$ be a corresponding real Jordan basis for $V \otimes
  \mathbb{R}$ with respect to $T$, and $V \otimes \mathbb{R}=V_- \oplus V_+$,
  where $V_-$ is spanned by $v_1,\ldots,v_s$ and $V_+$ by $v_{s+1},\ldots,v_r$.

  On the other hand, let $u_1,\ldots,u_r$ be a (rational) basis for $V$
  consisting of elements of $\pi_n(X)$, and $\lVert \cdot \rVert$ the Euclidean
  norm with respect to this basis, with the property that $T$ is
  $\lVert \cdot \rVert$-increasing on $V_+$ and $\lVert \cdot \rVert$-decreasing
  on $V_-$.  Note that this is a nontrivial condition: for example, the matrix
  $\begin{pmatrix}1.1&1\\0&1.1\end{pmatrix}$ does not increase the standard
  Euclidean norm of every vector.  This can be remedied, however, by scaling one
  of the coordinates in order to decrease the off-diagonal term.  In general,
  one can, for example, choose sufficiently large multiples of a close rational
  approximation to a basis consisting of scalings of the $v_i$ in which the
  off-diagonal entries of $T$ are very small.

  We let $\Lambda$ be the lattice generated by the $u_i$.  Finally, we fix the
  following constants:
  \begin{itemize}
  \item $Q$ such that the matrices of $T$ and $T^{-1}$ with respect to the basis
    $\{u_i\}$ are both in $\frac{1}{Q}M_r(\mathbb{Z})$;
  \item $L=\min\left\{\min_{u \in V_+, \lVert u\rVert=1}
    \lVert Tu \rVert,\min_{u \in V_-, \lVert u\rVert=1}
    \lVert T^{-1}u \rVert\right\}$;
  \item $U=\max\{\max_{\lVert u\rVert=1} \lVert Tu \rVert,
    \max_{\lVert u\rVert=1} \lVert T^{-1}u \rVert\}$.
  \end{itemize}
  In particular, every vector in $Q\Lambda$ has a preimage in $\Lambda$ under
  both $T$ and $T^{-1}$.

  Take an element $\alpha \in V_i \subseteq \pi_n(X) \otimes \mathbb{Q}$.  It
  has a multiple which is a lattice point $p \in \Lambda$, which decomposes over
  $\mathbb{R}$ as $p=p_-+p_+$, with $p_- \in V_-$ and $p_+ \in V_+$.  Fix $M$
  such that $\lVert Mp \rVert \geq Q\sqrt{r}$.

  Suppose first that $V_-=0$, i.e.\ $T$ only has large eigenvalues.  Then for
  any $M>Q\sqrt{r}$, take $q=Tp^\prime$ to be the nearest point in
  $Q\Lambda$ whose coordinates are smaller than those of $Mp$.  Then $Mp=p_0+
  Tp^\prime$, where $p_0$ and $p^\prime$ are lattice points, $\lVert p_0 \rVert
  \leq Q\sqrt{r}$ and $\lVert p^\prime \rVert \leq M\lVert p \rVert/L$.
  Continuing this construction inductively gives us
  $$Mp=\sum_{i=0}^\ell T^ip_i,$$
  with $\lVert p_i \rVert \leq Q\sqrt{r}$ for every $i$ and $\ell \leq
  \log_L(M\lVert p \rVert)$.  Thus
  $$|M\alpha|_{\vol} \leq \ell Q\sqrt{r} \in O(\log \lVert p \rVert).$$
  If $V_+=0$, the same computation holds substituting $T^{-1}$ for $T$.

  Now suppose $T$ has both small and large eigenvalues.  In other words,
  multiplying by a power of $T$ shrinks a vector in certain irrational
  directions and stretches it in others.  If these directions were rational,
  then by the above we could distort a vector going in just one of them.  Our
  strategy will be to express our chosen vector as a sum of shrinking and
  expanding components as precisely as possible.

  To this end, we write $Mp=a+b+p_0$, where $a$ and $b$ are in $Q\Lambda$,
  $\lVert p_0 \rVert<Q\sqrt{r}$, $\lVert a-p_- \rVert<Q\sqrt r$ and
  $\lVert b-p_+ \rVert<Q\sqrt r$.  Applying $T^{-1}$ to this sort of
  decomposition, we get the following lemma.
  \begin{lem} \label{lem:bound}
    For any lattice point $q \in V_+ \oplus V_-$, we can write $q=a+Tb$,
    with $a$ and $b$ lattice points, $\lVert a_+ \rVert<Q\sqrt r$ and
    $\lVert b_- \rVert<UQ\sqrt r$, $\lVert a_- \rVert \leq \lVert q_- \rVert+
    Q\sqrt r$ and $\lVert b_+ \rVert \leq \frac{\lVert q_+ \rVert}{L}$.  The
    same thing holds switching $V_+$ and $V_-$ components if we substitute
    $T^{-1}$ for $T$.
  \end{lem}

  Applying Lemma \ref{lem:bound} to $Mp$, we get $Mp=a+p_0+Tp_1$ with the
  appropriate bounds.  Applying the lemma inductively to $p_i$ gives
  $$Mp=a+\sum_{i=0}^\ell T^ip_i,$$
  where $\lVert p_i \rVert \leq (U+1)Q\sqrt{r}$ and $\ell \leq 
  \log_L(M\lVert p \rVert)$.  We can now also apply the $T^{-1}$ case of Lemma
  \ref{lem:bound} to $a$ to get
  $$Mp=\sum_{i=-\ell}^\ell T^ip_i,$$
  and the same bounds on $p_i$ and $\ell$ hold.  Hence
  $$|M\alpha|_{\vol} \leq 2\ell (U+1)Q\sqrt{r}) \in O(\log M\lVert p \rVert).$$
  Thus for any $\alpha$, volume distortion is at least exponential.
\end{proof}
Conversely, generalizing Example \ref{exs:oddD}(1), one sees that if $T$ is
diagonalizable and has at least one eigenvalue on the unit circle, then the
action of $\gamma$ does not induce any distortion on $V$.

The situation is more complex when $T$ has eigenvalues on the unit circle but is
non-diagonalizable.  If it is quasiunipotent, i.e.\ the eigenvalues are roots of
unity, then we can take a power of it which is in fact unipotent.  For an
irreducible unipotent $T$, there is an eigenvector $\vec v$ and a $\vec u$ for
which $T\vec u=\vec u+\vec v$.  Then $k\vec v=T^k\vec u-\vec u$ and so $\vec v$
is infinitely distorted.  For all but one of the other generalized eigenvectors,
one can find similar polynomials, showing that they are infinitely distorted as
well.  The remaining generalized eigenvector is undistorted.  If $T$ has
eigenvalues with irrational angles, then certain individual vectors are
distorted because they are weakly infinitely distorted, as one sees in Example
\ref{exs:oddD}(\ref{qqu}).  The precise distortion functions are harder to
ascertain, though it seems reasonable to suppose that, as in Example
\ref{exs:oddD}(\ref{qqu}), they are polynomial.

\subsection*{Summary} In this section, we saw that if $X$ is a finite complex
and no subspace of $\pi_*(X)$ is infinitely distorted, then:
\begin{itemize}
\item $\pi_*(X) \otimes \mathbb{Q}$ is finite dimensional;
\item the universal cover $\tilde X$ is rationally equivalent to a product of
  odd-dimensional spheres;
\item and the action of $\Gamma$ on $\pi_*(X) \otimes \mathbb{Q}$ is elliptic.
\end{itemize}
We will refer to spaces that satisfy these conditions as \emph{delicate spaces}.

\section{Filling functions}

In this section, we study various notions of isoperimetry in higher dimensions,
aiming to build a library of general results and examples which we can deploy
later.  Section \ref{S:5} will relate these to a dual cohomological notion, and
we will apply this duality to the study of distortion in Section \ref{S:6}.

The template for most study of isoperimetry in higher dimensions is the
one-dimensional example of Dehn functions.  The Dehn function of a group
$\Gamma$ describes the difficulty of solving the word problem in that group;
specifically, $\delta_\Gamma(k)$ is the minimal number of conjugates of
relations required to trivialize a trivial word of length $k$.  This has a
geometric interpretation as the cellular volume of fillings of cellular loops in
the Cayley 2-complex of $\Gamma$.

There are several different ways to generalize this notion to higher dimensions.
Higher-dimensional Dehn functions were first defined by \cite{AWP}, and filling
volume functions by Gromov in \cite{GrAsym}; later, other equivalent and
non-equivalent definitions of filling functions in groups and spaces have been
given by \cite{BBFS}, \cite{Young}, and \cite{Groft}.

Robert Young \cite{Young} formalizes the distinction between \emph{homotopical}
Dehn functions, which measure the difficulty of extending maps from spheres to
disks, or more generally, maps from $\partial M$ to $M$ for a manifold $M$, and
\emph{homological} isoperimetric functions, which measure the difficulty of
filling chains by cycles.  We will additionally introduce \emph{directed}
isoperimetric functions, which measure pairings of fillings with cohomology
classes and which have not been discussed before in this guise.  Except in
dimension 2, homotopical filling functions bound homological filling functions
from above.  These, in turn, bound directed isoperimetric functions, creating a
kind of hierarchy of ``coarseness.''

Certain of these functions are harder or easier to compute in certain
situations.  To take advantage of the ability to make comparisons between them,
we give definitions of all three types.

We start with the most obvious Dehn function for fillings of spheres with disks.
\begin{defn}
  Let $X$ be a compact space with fundamental group $\Gamma$ and $n$-connected
  fundamental cover $\tilde X$.  Given a Lipschitz map $f:S^n \to \tilde X$,
  define the \emph{filling volume} of $f$ to be the minimal volume of an
  extension of $f$ to $D^{n+1}$:
  $$\delta_X^n(f)=\inf\{\vol(g) \mid g:D^{n+1} \to \tilde X\text{ s.t.\ }
  g|_{\partial D^{n+1}}=f\}.$$
  The \emph{$n$-dimensional Dehn function} of $X$ is
  $$\delta_X^n(k)=\sup\{\delta_X^n(f) \mid f:S^n \to \tilde X\text{ s.t.\ }
  \vol f \leq k\}.$$
\end{defn}
One can also ask, for another $(n+1)$-manifold with boundary $(M,\partial M)$,
how hard it is to fill a map $\partial M \to X$ with a map $M \to X$:
\begin{defn}
  Given a Lipschitz map $f:\partial M \to \tilde X$, define the
  \emph{filling volume} of $f$ to be
  $$\delta_X^{(M,\partial M)}(f)=\inf\{\vol(g) \mid g:M \to \tilde X\text{ s.t.\ }
  g|_{\partial M}=f\}$$
  and we can define the corresponding Dehn function
  $$\delta_X^{(M,\partial M)}(k)=\sup\{\delta_X^{(M,\partial M)}(f) \mid
  f:\partial M \to \tilde X\text{ s.t.\ }\vol f \leq k\}.$$
\end{defn}
Next, we define the filling volume or homological isoperimetric functions.
\begin{defn}
  Let $X$ be a compact space with fundamental group $\Gamma$ and $n$-connected
  fundamental cover $\tilde X$.  Given a Lipschitz boundary $\beta \in
  C_n(\tilde X)$, define the \emph{filling volume} of $\beta$ to be
  $$\FVol_X^n(\beta)=\inf\{\vol(\alpha) \mid \alpha \in C_{n+1}(\tilde X)
  \text{ s.t.\ }\partial\alpha=\beta\}$$
  and the filling volume function
  $$\FV_X^n(k)=\sup\{\FVol_X^n(\beta) \mid \beta \in
  C_n(\tilde X)\text{ s.t.\ }\vol \beta \leq k\}.$$
  One can also restrict to boundaries that look like a specific $n$-manifold $N$
  to get
  $$\FV_X^N(k)=\sup\{\FVol_X^n(\beta) \mid \beta \in
  C_n(\tilde X)\text{ s.t.\ }\vol \beta \leq k\text{ and }\beta=f_*[N]
  \text{ for some }f:N \to \tilde X\}.$$
\end{defn}

By restricting chains to be cellular and maps to be admissible, one can get
similar cellular definitions.  Note, however, that cellular and Lipschitz
filling functions of $X$ might not be asymptotically equivalent as defined
earlier; e.g.\ one may be linear and the other sublinear.  Instead, we need a
slightly weaker notion of \emph{coarse equivalence},
\begin{align*}
  f \lesssim_C g &\iff f(k) \leq Ag(Bk+C)+Dk+E \\
  f \sim_C g &\iff f \lesssim_C g\mbox{ and }f \gtrsim_C g,
\end{align*}
for arbitrary constants $A$, $B$, $C$, $D$, and $E$.  It is easy to apply the
deformation theorem to see that these functions are indeed coarsely equivalent.
One also then sees that the cellular versions depend only on the
$(n+1)$-skeleton of $X$, and that a homotopy equivalence up to dimension $n$
between $X$ and $Y$ induces a coarse equivalence of filling functions.  This
gives a well-defined notion of $\FV_\Gamma^n$ and $\FV_\Gamma^N$ for any group
$\Gamma$ of type $\mathcal{F}_{n+1}$, that is, which has a $K(\Gamma,1)$ with
finite $(n+1)$-skeleton.

We will now refine the notion of filling volume to define filling homology
classes and directed isoperimetric functions.
\begin{defn}
  Assume that $X$ is constructed via Lipschitz attaching maps and that
  $\tilde X$ is $(n+1)$-connected.  Let $p:\tilde X \to X$ be the universal
  covering map.  Given an $n$-manifold $M$ and an admissible map $f:M \to
  \tilde X$, define the \emph{chain evaluation} $[[f]] \in C_n(X)$ to be the
  cellular chain whose value on a cell $c$ is the degree of the map $p \circ f$
  in $H_*(X,X \setminus c)$.  Suppose $[[f]]=0$.  Then for any homological
  filling $G \in C_{n+1}(\tilde X)$ of $f$, its image $p_\#G \in C_{n+1}(X)$ is a
  cycle, and we can define the \emph{filling class} $\Fill(f) \in H_{n+1}(X)$ to
  be its homology class.  This is well-defined since two such fillings differ by
  an $(n+1)$-boundary in $\tilde X$.

  Now, given a seminorm $\lVert\cdot\rVert$ on $H_{n+1}(X;\mathbb{Q})=
  H_{n+1}(\Gamma;\mathbb{Q})$, define the \emph{directed isoperimetric function}
  of $\Gamma$ with respect to $\lVert\cdot\lVert$ to be
  $$\FV_{\Gamma,\lVert \cdot \rVert}^M(k)=\sup\{\lVert\Fill(f)\rVert \mid f:M
  \to \tilde X\text{ admissible s.t.\ $[[f]]=0$ and }\vol f \leq k\}.$$
  If $b \in C_n(\tilde X)$ is a cellular boundary with $p_\#b=0$, then we can
  similarly define a filling class $\Fill(b) \in H_{n+1}(X)$, and filling
  functions
  $$\FV_{\Gamma,\lVert \cdot \rVert}^n(k)=\sup\{\lVert\Fill(b)\rVert \mid b \in
  C_n(\tilde X)\text{ cellular s.t.\ $p_\#b=0$ and }\vol b \leq k^{1/n}\}.$$
  More generally, suppose $X$ is any finite complex with Lipschitz attaching
  maps, $p:\tilde X \to X$ is the universal covering map, and
  $\lVert\cdot\rVert$ is a seminorm on $H_{n+1}(X;\mathbb{Q})/
  p_*H_{n+1}(\tilde X;\mathbb{Q})$.  Then a map $f:M \to \tilde X$ with $[[f]]=
  0$, or a boundary $b \in C_n(\tilde X)$ with $p_\#b=0$, has a filling class
  $$\Fill f \in H_{n+1}(X;\mathbb{Q})/p_*H_{n+1}(\tilde X;\mathbb{Q}),$$
  and we can define the filling functions $\FV_{\Gamma,\lVert \cdot \rVert}^M$ and
  $\FV_{\Gamma,\lVert \cdot \rVert}^n$ as above.
\end{defn}
It is clear from the definitions that for any $X$, $n$, and seminorm
$\lVert\cdot\rVert$, $\FV_{X,\lVert\cdot\rVert}^n(k) \lesssim \FV_X^n(k)$.  We now
try to understand directed isoperimetric inequalities in their own right.
\begin{exs}
  Here are two examples of seminorms with respect to which we may take directed
  isoperimetric inequalities:
  \begin{enumerate}
  \item Define the \emph{cellular norm} of $h \in H_{n+1}(X)$ by
    $$\lVert h \rVert_{\cel}=\min\left\{\sum a_i \Big\vert \sum a_ic_i
    \text{ is a cellular representative of }h\right\}.$$
    Then $\FV_{X,\lVert \cdot \rVert_{\cel}}^M$ is defined whenever $H_{n+1}(\tilde X)
    =0$.  The cellular norm is maximal among the seminorms we might consider:
    for any other seminorm $\lVert \cdot \rVert$ on $H_{n+1}(X)$,
    $\lVert \cdot \rVert \leq C\lVert \cdot \rVert_{\cel}$ for some $C$.
  \item Suppose that $e \in H^{n+1}(X;\mathbb{Q})$ is a cohomology class such
    that $p^*e=0$.  Then $\lvert\langle e,\cdot \rangle\rvert$ defines a
    seminorm on $H_{n+1}(X;\mathbb{Q})/p_*H_{n+1}(\tilde X;\mathbb{Q})$.  The
    directed isoperimetric functions we will use most often are with respect to
    this seminorm.
  \end{enumerate}
\end{exs}
\begin{ex}
  For the most obvious examples, that is $\Gamma=\mathbb{Z}^2$ and other
  surface groups, $\FV_{\Gamma,\lVert\cdot\rVert}^1(k) \sim \FV_\Gamma^1(k)$ for any
  nonzero seminorm $\lVert\cdot\rVert$; in any case all such choices of seminorm
  differ by a constant.

  In order to get a sense of the difference between homological and directed
  isoperimetric functions, consider the famous Baumslag-Solitar group $B=BS(1,2)
  =\langle a,b \mid bab^{-1}=a^2 \rangle$.  The Dehn function of this group is
  exponential: the word $b^{k}ab^{-k}ab^{k}a^{-1}b^{-k}a^{-1}$ of length $4k+4$
  represents the trivial element but takes $O(2^k)$ cells to fill.  By the same
  token, $\FV_B^1(k)$ is also exponential in $k$.

  On the other hand, the filling class of this word in $H_2(B)$ is zero, because
  in the usual filling, for every cell of positive orientation, there is a cell
  of negative orientation.  Indeed, it is easy to see that $H_2(B)=0$, since the
  Cayley complex only has one 2-cell and its boundary is nonzero.  Thus this
  kind of cancellation must happen for every word whose chain evaluation is
  zero.  In particular, $\FV_{B,\lVert\cdot\rVert}^1(k) \cong 0$ for any norm
  $\lVert\cdot\rVert$.
\end{ex}
Looking at this example, one may wonder if we have been too restrictive in
defining directed isoperimetric functions.  We know that every map is close
enough to an admissible map, so there is no harm done in restricting to such
maps.  We would like to show in addition that every map or chain is close in the
same sense to a map or chain whose chain evaluation is zero.  Then we have
better reason to believe that directed isoperimetric inequalities tell us
something about all maps or chains.
\begin{lem} \label{lem:tohom}
  Let $X$ be a finite complex with admissible boundary maps an universal
  covering map $p:\tilde X \to X$, and $n \geq 1$.  Then there is a constant $C$
  depending on $n$ and $X$ such that the following holds.
  \begin{enumerate}
  \item Let $f:S^n \to \tilde X$ be an admissible map of cellular $n$-volume $k$
    such that $[[f]]$ is a boundary.  Then $f$ can be deformed via a homotopy
    with $(n+1)$-volume $Ck$ to an admissible map $g$ of volume $Ck$ with
    $[[g]]=0$.
  \item Let $b \in C_n(\tilde X)$ be a boundary of volume $k$.  Then $b$ is
    homologous via an $(n+1)$-chain of volume $Ck$ to a boundary $c$ of volume
    $Ck$ with $p_\#c=0$.
  \end{enumerate}
\end{lem}
\begin{proof}
  We only prove (1); the proof of (2) is similar.

  We work in $X$, since all the maps we are considering lift to $\tilde X$.  We
  know $\lVert[[f]]\rVert_{\cel} \leq \vol f=k$.  Since $X$ is finite, the
  boundaries $B_n(X)$ form a finitely generated group.  Thus there is a constant
  $A$ such that we can always find an $(n+1)$-chain $c$ with $\partial c=[[f]]$
  and $\lVert c \rVert_{\cel} \leq Ak$.

  Let $B$ be the maximal volume of an attaching map $f_i$ of an $(n+1)$-cell
  $c_i$.  Then a map $D^n \to X$ which takes the disk to a balloon starting at a
  basepoint and with head $f_i$ is has volume at most $B$.  By mapping the upper
  hemisphere of $S^n$ to $X$ via $Ak$ balloons corresponding to the cells of
  $-c$ and the lower hemisphere via $f$, we create a map $g$ of volume $(AB+1)k$
  with $[[g]]=0$.  Since each $f_i$ can be nullhomotoped through $c_i$, this map
  is homotopic to $f$ via a homotopy with volume $\lVert c \rVert_{\cel}=Ak$.
  Thus we have proven the lemma with $C=AB+1$.
\end{proof}
To further promote the admission of directed functions to the filling function
pantheon, we would like to prove that they are well-defined for groups with
appropriate finiteness properties.
\begin{prop}
  Suppose $\Gamma$ is a group of type $\mathcal{F}_{n+1}$.  Then given a norm
  $\lVert\cdot\rVert$ on $H_{n+1}(\Gamma;\mathbb{Q})$,
  $\FV_{\Gamma,\lVert\cdot\rVert}^n$ depends up to coarse equivalence only on
  $\Gamma$, justifying the notation.
\end{prop}
\begin{proof}
  Suppose $X$ and $Y$ are two complexes with fundamental group $\Gamma$ and
  $n$-connected universal cover such that $H_{n+1}(X)=H_{n+1}(Y)=
  H_{n+1}(\Gamma)$.  The last condition can always be satisfied given a complex
  which satisfies the first two by adding a finite number of $(n+2)$-cells.

  In particular, we can find a cellular map $h:X^{(n+2)} \to Y$ which induces an
  isomorphism on $H_{n+1}$.  Now suppose $f:S^n \to X$ is an admissible map of
  volume $k$ with $[[\tilde f]]=0$.  Then $h \circ f$ is a cellular map of
  volume $Ck$.  Although it may not be admissible, it has well-defined degrees
  on $n$-cells and is homotopic in $Y^{(n+1)}$ to a map $g$ with the same bound
  on volume and the same degrees on $n$-cells; in particular, $[[\tilde g]]=0$.
  Moreover, an admissible filling of $f$ gives a corresponding cellular filling
  of $g$.  Since $h_*:H_{n+1}(X) \to H_{n+1}(Y)$ is induced by the corresponding
  map of chain complexes, $h_*\Fill(g)=\Fill(f)$.
\end{proof}
These are the only general results for directed isoperimetric functions that we
will show at this time.  Before we move on to some more interesting examples, we
prove some comparison results about filling volumes and Dehn functions.

First, we show that in dimensions other than 2, every cellular boundary is
induced by a map $f:S^n \to X$; we say that every boundary is \emph{spherical}.
This means that isoperimetric functions are equivalent whether or not we require
boundaries to be spherical.
\begin{lem} \label{lem:BBFS}
  Let $X$ be a finite complex with universal cover $\tilde X$ and $n \geq 3$.
  Then for every integral boundary $c \in C_n(\tilde X)$, there is an admissible
  $f:S^n \to \tilde X$ with $f_\#([S^n])=c$ and no cells of opposite
  orientations.  In particular, $\FV_X^{S^n} \sim \FV_X^n$ and for every norm
  $\lVert\cdot\rVert$ on $H^{n+1}(X)$, $\FV_{X,\lVert\cdot\rVert}^{S^n}
  \sim \FV_{X,\lVert\cdot\rVert}^n$.
\end{lem}
\begin{proof}
  This proof generalizes Remark 2.6(4) in \cite{BBFS}.  Let $c$ be a boundary in
  $C_n(\tilde X)$, and take an admissible map $g:(D^n,S^{n-1}) \to (\tilde X,
  \tilde X^{(n-1)})$ such that $g([D^n])=c$ and with no cells of opposite
  orientations, for example by mapping the boundary to a sum of attaching maps.
  Then
  $$g_\#([D^n,S^{n-1}])=0 \in H_n\left(\tilde X,\tilde X^{(n-1)}\right),$$
  and thus, by the relative Hurewicz theorem, $[g]=0 \in
  \pi_n(\tilde X,\tilde X^{(n-1)})$.  This means $[g|_{\partial D^n}]=0 \in
  \pi_{n-1}(\tilde X^{(n-1)})$, that is, it is nullhomotopic within the
  $(n-1)$-skeleton.  Using such a nullhomotopy, and including $D^n \subset S^n$
  as the upper hemisphere, we can extend $g$ to a map $f:S^n \to X$ with the
  desired properties.
\end{proof}
Finally, it's worth remarking that for groups, homological filling functions are
always finite.  This is not the case for all spaces.  For example, if we set $X=
S^2 \times S^1$, then a 2-boundary in $\tilde X$ of volume 2 can have
arbitrarily large filling volume: just take two copies of $S^2$ arbitrarily far
apart with opposite signs.
\begin{lem} \label{lem:finite}
  Let $\Gamma$ be a group of type $\mathcal{F}_n$.  Then for every $1 \leq m
  \leq n$, $\FV_\Gamma^m(k)<\infty$.
\end{lem}
\begin{proof}
  Fix a CW complex $X$ with $\pi_1X=\Gamma$ and $n$-connected universal cover
  $\tilde X$.  Theorem 1 of \cite{AWP} states among other things that
  $\delta_\Gamma^m(k)=\delta_X^m(k)<\infty$ for $1 \leq m \leq n$.  By Lemma
  \ref{lem:BBFS}, for $m \geq 3$, $\FV_\Gamma^m(k)<\delta_\Gamma^m(k)$, which
  completes the proof in that case.  The same inequality is true for $m=1$,
  since every 1-cycle is a union of circles.  The only case in which we have
  something to prove is $m=2$; here it may be harder to fill a surface of
  positive genus than a sphere.

  We will now show that $\FV_X^2(k)$ is finite for every $k$.  We may assume
  that $X^{(2)}$ is a simplicial complex and use simplicial volume as our
  measure of 2-volume.  Then given a cocycle, and hence coboundary, $\alpha \in
  C_2(\tilde X)$ of volume $k$, we can glue the cells of $\alpha$ into a
  simplicial map $f:\Sigma:=\bigsqcup_{i=1}^r \Sigma_{g_i} \to \tilde X$, with
  each surface $\Sigma_{g_i}$ glued out of triangles.  Note that $r \leq k$ and
  the Euler characteristic of $\Sigma$ is bounded in terms of $k$:
  $$\chi\left(\bigsqcup_{i=1}^r \Sigma_{g_i}\right)=\sum_{i=1}^r (2-2g_i) \geq
  -\frac{k}{2}$$
  since the difference between these two quantities is the number of vertices of
  $\Sigma$.  Thus $G(\Sigma)=\sum_{i=1}^r g_i$ is also bounded in terms of $k$.

  From here, we will show by induction on $G$ that $\FV_X^\Sigma(k)$ is finite.
  Since $\delta_\Gamma^2(k)$ is finite, this is true when $G=r$, that is, when
  $g_i=0$ for all $i$.  Now, by Gromov's systolic inequality for surfaces
  (Theorem 11.3.1 in \cite{Katz}), for any Riemannian surface $V=\Sigma_g$ there
  is a nonseparating loop $\gamma:S^1 \to \Sigma_{g_i}$ of length at most
  $C\log g\sqrt{g^{-1}\vol V}$.  A simplicial version of this inequality is
  shown, for example, in \cite{VHM}.  Thus we can find a simple, nonseparating
  simplicial loop $\gamma:S^1 \to \Sigma_{g_i}$ of length $k^\prime \leq
  C\log g\sqrt{g^{-1}k}$.  Define a surface $\Sigma^\prime$ and a map
  $f^\prime:\Sigma^\prime \to \tilde X$ by cutting $\Sigma_{g_i}$ at $\gamma$ and
  gluing in two copies of a minimal disk filling $\gamma$.  Then
  $G(\Sigma^\prime)=G(\Sigma)-1$ and
  $$\vol \Sigma^\prime \leq k^\prime:=k+2\delta_X^1(C\log g\sqrt{g^{-1}k}).$$
  Moreover, any filling of $f^\prime$ is also a filling of $f$.  Thus
  $\FV_X^\Sigma(k) \leq \FV_X^{\Sigma^\prime}(k^\prime)$ is finite.

  Since for each $k$, there is a maximal possible $G$, this gives us an overall
  bound on $\FV_X^2(k)$.
\end{proof}
\begin{exs}
  To conclude the section, we give some examples of directed isoperimetric
  functions.
  \begin{enumerate}
  \item Suppose $\Gamma$ is a group of type $\mathcal{F}_{n+1}$ with homological
    Dehn function $\FV^1_\Gamma=f(k)$.  By a theorem of \cite{Young}, for $n
    \geq 2$, $\FV^n_{\Gamma^n}(k) \geq f(k)$ because if $\gamma$ is hard to fill
    in $\Gamma$, then $\gamma \times \cdots \times \gamma$ is hard to fill in
    $\Gamma^n$.  However this doesn't tell us anything about directed
    isoperimetric functions, because if $D$ is a chain filling some 1-chain
    $\gamma$ in $\Gamma$ whose chain evaluation is zero, then
    $$D \times \gamma \times \cdots \times \gamma$$
    is a filling of $\gamma \times \cdots \times \gamma$ whose chain evaluation
    is also zero.
  \item Suppose $M^{n+1}$ is a closed oriented smooth manifold with fundamental
    group $\Gamma$.  By dualizing a handle decomposition, we see that filling an
    $n$-boundary in $\tilde M$ is equivalent to finding a 0-cochain cobounding a
    compactly supported 1-cochain in the Cayley graph of $\Gamma$.  Thus any
    $n$-boundary $b$ has a unique filling; moreover, $b=b_++b_-$, where $\vol b=
    \vol b_++\vol b_-$ and $b_+$ has a filling by positively oriented copies of
    the top cell while $b_-$ has a filling by negatively oriented copies.  Thus
    $\FV_M^n(k)$, $\FV_{M,\lvert\langle[M],\cdot\rangle\rvert}^n(k)$, the
    isoperimetric problem for domains in $\tilde M$, and the problem of bounding
    from below the sizes of boundaries of subsets of $\Gamma$ are all
    equivalent.  In particular, $\FV_M^n(k) \sim
    \FV_{M,\lvert\langle[M],\cdot\rangle\rvert}^n(k)$ is linear if and only if
    $\Gamma$ is non-amenable.
  \item Define the $n$th diamond group
    $$\dmd_n=\left\langle b_1,c_1,\ldots,b_n,c_n,a \bigg\vert
    \begin{aligned}
      &b_i^{-1}ab_i=c_i^{-1}ac_i=a^2 \\ \relax
      &[b_i,b_j]=[b_i,c_j]=[c_i,c_j]=0\text{ for }i \neq j
    \end{aligned}\right\rangle.$$
    We can think of $\dmd_n$ as $F_2^n$ with an extra generator $a$ together
    with some relations involving it.  Alternatively, we can define $\dmd_n$
    inductively by setting $\dmd_0=\mathbb{Z}$ and $\dmd_n$ to be a multiple
    ascending HNN extension of $\dmd_{n-1}$, specifically, the fundamental group
    of the graph of groups with a single vertex $\dmd_{n-1}$ and two edges each
    labeled by the injective self-homomorphism $a \mapsto a^2$, $b_i \mapsto
    b_i$, $c_i \mapsto c_i$.  (Indeed, when $n \geq 2$, this is an
    automorphism.)  This last definition gives a construction for a
    $(n+1)$-dimensional classifying complex $X_n$ for $\dmd_n$, starting with an
    $S^1$ with one 1-cell and setting $X_n$ to be the appropriate quotient space
    of $X_{n-1} \times ([0,1] \sqcup [0,1])$ with the product cell structure.

    \tikzset{->-/.style={decoration={
          markings,
          mark=at position .5 with {\arrow{>}}},postaction={decorate}}}
    \tikzset{-<-/.style={decoration={
          markings,
          mark=at position .5 with {\arrow{<}}},postaction={decorate}}}
    \tikzset{->>-/.style={decoration={
          markings,
          mark=at position .25 with {\arrow{>}},
          mark=at position .75 with {\arrow{>}}},postaction={decorate}}}
    \tikzset{-<<-/.style={decoration={
          markings,
          mark=at position .75 with {\arrow{<}},
          mark=at position .25 with {\arrow{<}}},postaction={decorate}}}
    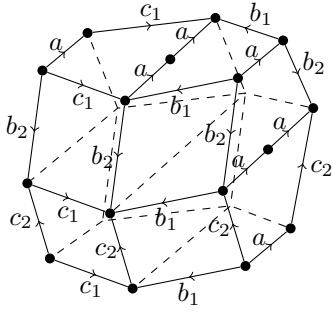
\begin{figure}
      \centering
      \begin{tikzpicture}
        \tikzstyle{frontN}=[circle,fill=black,scale=0.4];
        \coordinate (eee) at (0.7,0.8) {};
        \node[frontN] (eea) at (-0.2,0) {};
        \coordinate (eeA) at (1.6,1.6) {};
        \node[frontN] (bea) at (0,1.5) {};
        \node[frontN] (beA) at (1.2,2.6) {};
        \node[frontN] (bba) at (1.5,1.8) {};
        \node[frontN] (bbA) at (2.1,2.3) {};
        \node[frontN] (eba) at (1.3,0.3) {};
        \node[frontN] (ebA) at (2.5,1.4) {};
        \node[frontN] (cba) at (1.6,-0.7) {};
        \node[frontN] (cbA) at (2.2,-0.2) {};
        \node[frontN] (cea) at (0.1,-1) {};
        \coordinate (ceA) at (1.4,0.1) {};
        \node[frontN] (cca) at (-1,-0.6) {};
        \coordinate (ccA) at (-0.3,-0.1) {};
        \node[frontN] (eca) at (-1.3,0.4) {};
        \coordinate (ecA) at (-0.1,1.4) {};
        \node[frontN] (bca) at (-1.1,1.9) {};
        \node[frontN] (bcA) at (-0.5,2.4) {};
        \tikzstyle{behind}=[dashed];
        \draw[-<-] (bea)--(bba) node[pos=0.5,anchor=north,inner sep=2pt]{$b_1$};
        \draw[->-] (bba)--(eba) node[pos=0.5,anchor=east,inner sep=1pt]{$b_2$};
        \draw[-<-] (eba)--(cba) node[pos=0.5,anchor=east,inner sep=1pt]{$c_2$};
        \draw[->-] (cba)--(cea) node[pos=0.5,anchor=north,inner sep=2pt]{$b_1$};
        \draw[-<-] (cea)--(cca) node[pos=0.5,anchor=north,inner sep=3pt]{$c_1$};
        \draw[->-] (cca)--(eca) node[pos=0.5,anchor=east,inner sep=2pt]{$c_2$};
        \draw[-<-] (eca)--(bca) node[pos=0.5,anchor=east,inner sep=2pt]{$b_2$};
        \draw[->-] (bca)--(bea) node[pos=0.5,anchor=north,inner sep=3pt]{$c_1$};
        \draw[->-] (bca)--(bcA)
        node[pos=0.5,anchor=south east,inner sep=0.5pt]{$a$};
        \draw[->-] (bcA)--(beA) node[pos=0.5,anchor=south,inner sep=2pt]{$c_1$};
        \draw[-<-] (beA)--(bbA)
        node[pos=0.45,anchor=south west,inner sep=1pt]{$b_1$};
        \draw[->-] (bbA)--(ebA)
        node[pos=0.55,anchor=south west,inner sep=1pt]{$b_2$};
        \draw[-<-] (ebA)--(cbA) node[pos=0.5,anchor=west,inner sep=3pt]{$c_2$};
        \draw[-<-] (cbA)--(cba)
        node[pos=0.5,anchor=south east,inner sep=0.5pt]{$a$};
        \draw[behind] (cbA) -- (ceA) -- (ccA) -- (ecA) -- (bcA);
        \draw[-<<-] (beA) -- (bea) node[pos=0.5,frontN]{}
        node[pos=0.25,anchor=south east,inner sep=0.5pt]{$a$}
        node[pos=0.75,anchor=south east,inner sep=0.5pt]{$a$};
        \draw[->-] (bea)--(eea) node[pos=0.5,anchor=east,inner sep=1pt]{$b_2$};
        \draw[-<-] (eea)--(cea) node[pos=0.5,anchor=east,inner sep=1pt]{$c_2$};
        \draw[->-] (eca)--(eea) node[pos=0.5,anchor=north,inner sep=3pt]{$c_1$};
        \draw[-<-] (eea)--(eba) node[pos=0.5,anchor=north,inner sep=2pt]{$b_1$};
        \draw[->>-] (eba)--(ebA) node[pos=0.5,frontN]{}
        node[pos=0.25,anchor=south east,inner sep=0.5pt]{$a$}
        node[pos=0.75,anchor=south east,inner sep=0.5pt]{$a$};
        \draw[->-]
        (bba)--(bbA) node[pos=0.5,anchor=south east,inner sep=0.5pt]{$a$};
        \draw[behind] (beA) -- (eeA) -- (ceA) -- (cea);
        \draw[behind] (ccA) -- (cca) (eea) -- (eee) -- (eeA);
        \draw[behind] (ebA) -- (eeA) -- (ecA) -- (eca);
      \end{tikzpicture}

      \caption{
        A lift of the 3-chain $\sigma_2$ to the universal cover $\tilde X_2$.
        This induces a cycle in $X_2$ since opposite faces cancel out.
      } \label{fig:sigma}
    \end{figure}
    The space $X_n$ has $2^n$ $(n+1)$-cells $e_I$ corresponding to elements
    $I \in \{b,c\}^n$.  It's easy to see that the only $(n+1)$-cycles are
    multiples of
    $$\sigma_n=\sum_{I \in \{b,c\}^n} (-1)^{\text{number of $b$'s in }I}e_I \in
    C_{n+1}(X_n),$$
    and so $H_{n+1}(X_n) \cong \mathbb{Z}$.

    For now, we show that all the top-dimensional isoperimetric functions of
    $\dmd_n$ have the same superpolynomial growth.
    \begin{thm}
      Let $n \geq 1$, and let $h \neq 0 \in H^{n+1}(\dmd_n)$.  Then
      $$\FV_{\dmd_n,\lvert\langle h,\cdot \rangle\rvert}^n(k) \sim \FV_{\dmd_n}^n(k)
      \sim \delta_{\dmd_n}^n(k) \sim 2^{\sqrt[n]{k}}.$$
    \end{thm}
    \begin{proof}
      We start by setting some notation.  Let $p:\tilde X \to X$ be the
      universal covering and let $\rho_i:\dmd_i \to \dmd_i$ be the monodromy
      homomorphism $a \mapsto a^2$, $b_i \mapsto b_i$, $c_i \mapsto c_i$ used in
      the construction of $\dmd_{k+1}$.  We write $I_1$ and $I_2$ for the two
      intervals used to construct $X_n$ from $X_{n-1}$.  Note that the universal
      cover $\tilde X_{n+1}$ consists of glued-together copies of $\tilde X_n
      \times [0,1]$ indexed by edges of the Bass-Serre tree corresponding to the
      graph of groups, which we call \emph{layers}.

      As already discussed, for $n \neq 2$, it is automatic that
      $$\FV_{\dmd_n,\lvert\langle h,\cdot \rangle\rvert}^n(k) \lesssim
      \FV_{\dmd_n}^n(k) \lesssim \delta_{\dmd_n}^n(k).$$
      In the case $n=2$, the second inequality does not obviously hold, but
      $$\FV_{\dmd_2}^2(k) \lesssim \max \left\{\delta_{\dmd_2}^{(M,\partial M)}(k):
      (M,\partial M)\text{ 3-manifold with boundary}\right\}.$$
      Thus it is enough to show that
      $2^{\sqrt[n]{k}} \lesssim \FV_{\dmd_n,\lvert\langle h,\cdot \rangle\rvert}^n(k)$
      and that $\delta_{\dmd_n}^{(M,\partial M)}(k) \lesssim 2^{\sqrt[n]{k}}$ for
      every $(n+1)$-manifold with boundary $(M,\partial M)$.

      To show the first of these two inequalities, we construct chains with
      large directed fillings in $\tilde X_n$.  Specifically, by induction on
      $n$, we construct a chain $\tau_n(k) \in C_{n+1}(\tilde X_n)$ for which
      $p_\#\tau_n(k)=K\sigma_n$ for some $2^k<K<2^{n+k}$, and whose boundary has
      volume $O(k^n)$, if $n$ is viewed as a constant.  Notice that because from
      a homological point of view all fillings are equivalent, we do not need to
      show that $\tau_n(k)$ is the ``best'' filling of its boundary.

      In $\dmd_1$, the construction of $\tau_n(k)$ is similar to the usual
      demonstration that $BS(1,2)$ has exponential Dehn function.  Namely, we
      take $\tau_1(k)$ to be the disk bounded by $b_1^{-k}ab_1^kc_1^{-k}a^{-1}
      c_1^k$.  Notice for the purpose of the induction that $\rho_1(\tau_1(
      k-1))$ gives a disk that differs from $\tau_1(k)$ by only two cells.

      \begin{figure}
        \centering
        \begin{tikzpicture}
          \tikzstyle{chainfill}=[fill=gray!25];
          \fill[chainfill] (0.4,3) -- (4.9,3.9) -- (5.5,4.4) -- (2.8,5.3)
          -- (-2.3,4.7) -- (-2.9,4.2) -- cycle;
          \tikzstyle{frontN}=[circle,fill=black,scale=0.4];
          \foreach \x in {0,1,...,4} {
            \foreach \y in {0,1,...,4} {
              \node[frontN] (b\x b\y a) at (1.5*\x+0.2*\y,0.3*\x+1.5*\y) {};
              \node[frontN] (b\x c\y a) at (1.5*\x+0.3*\y,0.3*\x-\y) {};
              \node[frontN] (c\x b\y a) at (-1.1*\x+0.2*\y,0.4*\x+1.5*\y) {};
              \node[frontN] (c\x c\y a) at (-1.1*\x+0.3*\y,0.4*\x-\y) {};
              \edef\r{\number\numexpr\x+\y\relax};
              \ifnum \r=5
                \breakforeach
              \fi
            }
            \node[frontN] (b\x b\x A) at (2+1.3*\x,7.1-1.2*\x) {};
            \coordinate (b\x c\x A) at (7.2-1.2*\x,2.3-1.3*\x) {};
            \coordinate (c\x c\x A) at (2.4-1.4*\x,-2.9+1.4*\x) {};
            \coordinate (c\x b\x A) at (-3.2+1.3*\x,2.7+1.1*\x) {};
          }
          \foreach \x in {0,1,2,3} {
            \edef\xx{\number\numexpr\x+1\relax};
            \node[frontN] (b\x b\xx A) at (2.9+1.3*\x,6.8-1.2*\x) {};
            \node[frontN] (b\x c\xx A) at (6.9-1.2*\x,0.7-1.3*\x) {};
            \coordinate (c\x c\xx A) at (0.7-1.4*\x,-3.1+1.4*\x) {};
            \node[frontN] (c\x b\xx A) at (-3.6+1.3*\x,3.6+1.1*\x) {};
          }
          \coordinate (eeA) at (3.6,3.2) {};
          \tikzstyle{behind}=[dashed];
          \draw[behind] (b0b0A) -- (eeA) -- (c0c0A) (b0c0A) -- (eeA) -- (c0b0A);
          \draw[behind] (b0b0a) -- (eeA)
          node[pos=0.7,anchor=south east,inner sep=-0.5pt]{$a^{32}$};
          \draw[->>-] (b4b0a) -- (b4b4A) node[pos=0.5,frontN]{}
          node[pos=0.25,anchor=south east,inner sep=0.5pt]{$a$}
          node[pos=0.75,anchor=south east,inner sep=0.5pt]{$a$};
          \draw[->-] (b0c0A)--(b0c1A)
          node[pos=0.5,anchor=west,inner sep=2pt]{$c_2$};
          \draw[->-] (c3b4A)--(c4b4A)
          node[pos=0.5,anchor=south,inner sep=2pt]{$c_1$};
          \foreach \y in {0,1,2,3} {
            \edef\yy{\number\numexpr\y+1\relax};
            \foreach \x in {1,2,3,4} {
              \draw[->-] (b\x b\yy a) -- (b\x b\y a);
              \draw[->-] (b\x c\yy a) -- (b\x c\y a);
              \draw[->-] (c\x b\yy a) -- (c\x b\y a);
              \draw[->-] (c\x c\yy a) -- (c\x c\y a);
              \draw[->-] (b\yy b\x a) -- (b\y b\x a);
              \draw[->-] (b\yy c\x a) -- (b\y c\x a);
              \draw[->-] (c\yy b\x a) -- (c\y b\x a);
              \draw[->-] (c\yy c\x a) -- (c\y c\x a);
              \edef\r{\number\numexpr\x+\y\relax};
              \ifnum \r=4
                \breakforeach
              \fi
            }
            \edef\my{\number\numexpr4-\y\relax};
            \draw[->-] (b\yy b0a) -- (b\y b0a)
            node[pos=0.5,anchor=north,inner sep=3pt]{$b_1$};
            \draw[->-] (b0b\yy a) -- (b0b\y a)
            node[pos=0.5,anchor=east,inner sep=2pt]{$b_2$};
            \draw[->-] (c\yy b0a) -- (c\y b0a)
            node[pos=0.5,anchor=north,inner sep=3pt]{$c_1$};
            \draw[->-] (b0c\yy a) -- (b0c\y a)
            node[pos=0.5,anchor=east,inner sep=2pt]{$c_2$};
            \draw[->>-] (b\y b\my a) -- (b\y b\y A) node[pos=0.5,frontN]{}
            node[pos=0.25,anchor=south east,inner sep=0.5pt]{$a$}
            node[pos=0.75,anchor=south east,inner sep=0.5pt]{$a$};
            \draw[->-] (b\yy b\my a) -- (b\y b\yy A)
            node[pos=0.5,anchor=south east,inner sep=0.5pt]{$a$};
            \draw[-<-] (b\y b\y A)--(b\y b\yy A)
            node[pos=0.45,anchor=south west,inner sep=1pt]{$b_1$};
            \draw[->-] (b\y b\yy A)--(b\yy b\yy A)
            node[pos=0.55,anchor=south west,inner sep=1pt]{$b_2$};
            \draw[->-] (b\my c\yy a) -- (b\y c\yy A)
            node[pos=0.5,anchor=north west,inner sep=0.5pt]{$a$};
            \ifnum \y>0
              \draw[dash pattern=on 3pt off 3pt on 3pt off 3pt on 3pt off 3pt on
                3pt off 3pt on 100pt] (b\y c\y A)--(b\y c\yy A)
              node[pos=0.75,anchor=west,inner sep=2pt]{$c_2$};
            \fi
            \draw[behind] (b\y c\yy A) -- (b\yy c\yy A);
            \draw[behind] (b\y c\my a) -- (b\my c\my A);
            \draw[behind] (c\y c\y A) -- (c\y c\yy A) -- (c\yy c\yy A);
            \draw[behind] (c\y c\my a) -- (c\y c\y A);
            \draw[behind] (c\yy c\my a) -- (c\y c\yy A);
            \draw[->-] (c\my b\yy a) -- (c\y b\yy A)
            node[pos=0.5,anchor=south east,inner sep=0.5pt]{$a$};
            \ifnum \y<3
              \draw[dash pattern=on 3pt off 3pt on 3pt off 3pt on 3pt off 3pt on
                3pt off 3pt on 3pt off 3pt on 100pt] (c\yy b\yy A)--(c\y b\yy A)
              node[pos=0.75,anchor=south,inner sep=2pt]{$c_1$};
            \fi
            \draw[behind] (c\y b\y A) -- (c\y b\yy A);
            \draw[behind] (c\my b\y a) -- (c\y b\y A);
          }
        \end{tikzpicture}
        \caption{
          An illustration of the 3-chain $\tau_2(4)$ in the universal cover
          $\tilde X_2$.  It is an optimal filling of its boundary, which is a
          hard-to-fill 2-sphere in the group $\dmd_2$.  The highlighted plane is
          a layer corresponding to $\tau_1(3)$.
        }
      \end{figure}
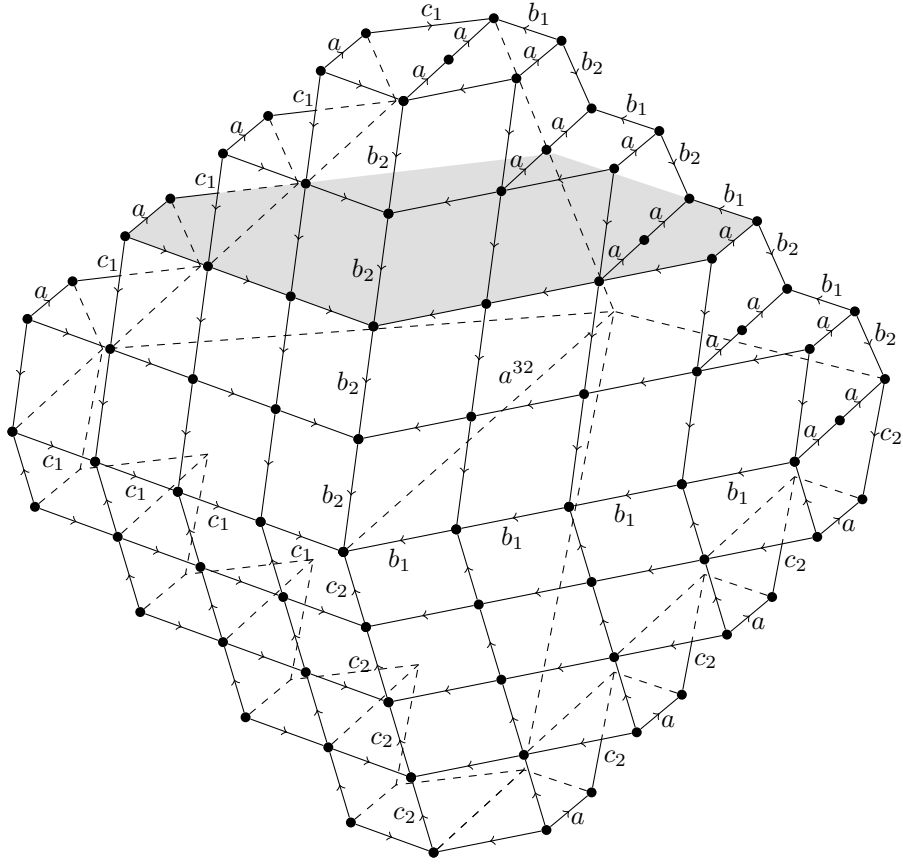
      Now suppose by induction that we have constructed $\tau_{n-1}(k) \in C_n(
      \tilde X_{n-1})$, for each $k \geq 0$, with the following properties:
      \begin{itemize}
      \item $p_\#\tau_{n-1}(k)=K\sigma_{n-1}$ for some $2^k \leq K \leq 2^{n+k}$;
      \item $\tau_{n-1}(0)=0$;
      \item $\vol(\partial\tau_{n-1}(k))=O(k^{n-1})$;
      \item for every $k \geq 1$, the $n$-chains $\rho_{n-1}\tau_{n-1}(k-1)$ and
        $\tau_{n-1}(k)$ differ by $O(k^{n-2})$ cells.
      \end{itemize}
      We construct $\tau_n(k)$ from $2k$ layers, that is, certain copies in
      $\tilde X_n$ of $(-1)^\ell\tau_{n-1}(j) \times I_\ell$ for $j=1,\ldots,k$
      and $\ell=1,2$.  This way, $p_\#\tau_{n-1}(k)=K\sigma_n$ with $K$ in the
      appropriate range.  Specifically, we pick these so that
      \begin{itemize}
      \item the two copies of $\tau_{n-1}(k-1) \times \{1\}$ cancel out;
      \item for each $2 \leq j \leq k-1$, $\tau_{n-1}(j) \times \{0\}$ cancels
        out with $\tau_{n-1}(j-1) \times \{1\}$, except for the aforementioned
        $O(k^{n-2})$ cells.
      \end{itemize}
      Then $\partial \tau_n(k)$ is the sum over $j$ and $\ell$ of copies of
      $(-1)^\ell\partial \tau_{n-1}(j) \times I_\ell$ and $(-1)^\ell[\tau_{n-1}(j)
      -\rho_{n-1}\tau_{n-1}(j-1)]$.  This means that in total,
      $$\vol(\partial \tau_n(k)) \leq 2kO(k^{n-1})+2kO(k^{n-2}).$$
      Moreover, since $\rho_n|_{\dmd_{n-1}}=\rho_{n-1}$, $\rho_n(\tau_{n-1}(j)
      \times I_\ell)$ differs from $\tau_{n-1}(j+1) \times I_\ell$ by $O(k^{n-2})$
      cells.  Thus $\rho_n\tau_n(k)$ differs from $\tau_n(k+1)$ by
      $$2kO(k^{n-2})+2^{n+1}=O(k^{n-1})$$
      cells.  This completes the inductive step.

      To show that $\delta_{\dmd_n}^n(k) \lesssim 2^{\sqrt[n]{k}}$, we do another
      induction.  It's clear that $\delta_{\dmd_1}^1(k) \leq 2^k$, and moreover,
      by Lemma 7.4 of \cite{BBFS}, $\delta_{\dmd_1}^{(M,\partial M)} \leq
      \delta_{\dmd_1}^1(k)$ for any surface with boundary $M$, giving us the base
      case.  For the inductive step when $n \geq 2$, we adapt and strengthen the
      argument of Theorem 7.2 of \cite{BBFS}, which concerns the top-dimensional
      Dehn functions of multiple ascending HNN extensions.
      \begin{thm*}[7.2 in \cite{BBFS}]
        Let $A_{n-1}$ be the $n$-dimensional classifying complex of a group $H$,
        and suppose that $F$ is a nondecreasing function such that
        $\delta_H^{(N,\partial N)}(k) \leq F(k)$ for all $n$-manifolds $N$ with
        boundary.  Let $A_{n+1}$ be the complex corresponding to a multiple
        ascending HNN extension $G$ of $H$.  Then for all $(n+1)$-manifolds with
        boundary $(M,\partial M)$, $\delta_G^{(M,\partial M)}(k) \leq F(k)$.
      \end{thm*}
      The idea of their proof is as follows.  First of all, since the dimension
      of the filling is equal to the dimension of the complex, the optimal
      filling will always be unique.  Moreover, for a map $f:(M,\partial M) \to
      A_{n+1}$, the volume of this filling is the sum of volumes of layers.

      Once again the universal cover $\tilde A_{n+1}$ consists of glued-together
      layers $\tilde A_n \times [0,1]$, indexed by edges $e$ of the Bass-Serre
      tree.  For each layer, let $Z_e$ be the corresponding copy of $\tilde A_n
      \times \{1/2\}$, and let $Z=\bigcup_e Z_e$.  Then given an admissible
      $f:(M,\partial M) \to \tilde A_{n+1}$ which is transverse to $Z$ (in the
      sense defined in \S VII.2 of \cite{BRS}) and setting, for each edge $e$ of
      the Bass-Serre tree, $N_e:=f^{-1}(Z_e)$ and $f_e=f|_{N_e}$, we have
      $$\vol f=\sum_e \vol(f|_{N_e}).$$
      On the other hand,
      $$\vol(f|_{\partial M})=\sum_e \vol(f|_{\partial N_e})+h,$$
      where $h$ stands for any ``horizontal'' volume which is the difference
      between fillings of the various $N_e$ meeting at a given vertex of the
      Bass-Serre tree.  Thus $\partial M$ is as easy to fill in $X_n$ as
      $\bigsqcup N_e$ is in $X_{n-1}$, and its volume is at least as large.

      In our case, however, this decomposition gives additional information.
      When certain boundaries have fillings which are much larger than
      $\vol(f|_{\partial M})$, this means that there must be adjacent boundaries
      that have very similar fillings.  This allows us to show that there is a
      large number of layers that have large intersections with $f$.  We will
      use this to prove the stronger statement that if
      $\delta_{\dmd_{n-1}}^{(N,\partial N)}(k) \lesssim 2^{\sqrt[n-1]{k}}$ for all
      pairs $(N,\partial N)$, then $\delta_{\dmd_{n-1}}^{(M,\partial M)}(k) \lesssim
      2^{\sqrt[n]{k}}$.

      So suppose $n \geq 2$, and let $M$ be an $(n+1)$-manifold with boundary
      and $f:(M,\partial M) \to \tilde X_n$ be a map of volume $2^k$ which is an
      optimal filling of its boundary.  We would like to show that
      $\vol f|_{\partial M} \gtrsim k^n$.  Our strategy will be to actually
      assume that $\vol f|_{\partial M} \leq k^n$, and then show that under that
      assumption $\vol f|_{\partial M} \gtrsim k^n$, because there are at least
      $\sim k$ layers each of which contributes volume at least $\sim k^{n-1}$
      to the boundary.

      So assume that $\vol(f|_{\partial M}) \leq k^n$, and that
      $$k>k_{\min}(n)=\max\{k_{\min}(n-1),3n\log_2 k+2\}.$$
      (We can choose $k_{\min}(1)=1$.)  Choose an edge $e_0$ for which
      $\vol(f|_{N_{e_0}})$ is maximal, so that
      $$\vol\left(f|_{N_{e_0}}\right) \geq \frac{2^k}{k^n} \geq 2^{2k/3+2}.$$
      By the inductive assumption, this means that $\vol(f|_{\partial N_{e_0}})
      \geq C_{n-1}k^{n-1}$.  Now let $v$ be a vertex incident to $e_0$, which has
      degree 4 in the Bass-Serre tree since $\rho_n$ is an automorphism.  Since
      $\rho$ multiplies areas by at most $2$, we know that for one of the edges
      incident to $v$, which we call $e_1$,
      $$\vol\left(f|_{N_{e_0}}\right) \leq 6\vol\left(f|_{N_{e_1}}\right)+
      \vol(f|_{\partial M}).$$
      Continuing in this vein, for $r=\lfloor k/9 \rfloor$, we can pick a path
      $e_0,e_1,\ldots,e_r$ such that
      $$\vol\left(f|_{N_{e_j}}\right) \geq \frac{1}{6}
      \left(\vol\left(f|_{N_{e_{j-1}}}\right)-k^n\right) \geq 2^{2k/3+2-3j} \geq
      2^{k/3+2}.$$
      By the inductive assumption, for each $0 \leq j \leq r$,
      $\vol(f|_{\partial N_{e_j}}) \geq (k/3)^{n-1}$, so
      $$\vol(f|_{\partial M}) \geq \frac{k}{9}C_{n-1}(k/3)^{n-1}=3^{-(n+1)}
      C_{n-1}k^n.$$
      By induction, we get that if $k>k_{\min}(n)$ and $\vol f=2^k$, then
      $\vol(f|_{\partial M}) \geq C_nk^n$, where $C_n=3^{-O(n^2)}$.  This completes
      the proof.
    \end{proof}
  \end{enumerate}
\end{exs}

\section{$L_\infty$ cohomology and fillings} \label{S:5}
Correspondences between isoperimetry and the cohomology theories that turn up in
coarse geometric settings have been noted a number of times in the literature,
notably by Block and Weinberger \cite{BWJAMS}, Attie, Block and Weinberger
\cite{ABW}, Gersten \cite{GerLost}, and Nowak and \v Spakula \cite{NoSpa}.  The
main technical theorem of this section generalizes most of these results, as
well as the classical max flow--min cut theorem from graph theory, using a
technique from the theory of algorithms, the duality theorem for linear
programming problems.  In effect, a linear programming problem seeks to optimize
a linear function subject to a number of linear constraints.  In the dual linear
program, the role of the constraints and variables is switched.  The theorem of
linear programming duality states that the optimum solution to the original and
dual programs is the same.  Our proof proceeds by translating our two conditions
into this formal setting and demonstrating that they generate dual linear
programs and are therefore equivalent.  For a more detailed discussion of these
ideas, see a textbook on algorithms, such as \cite{CLRS}.

Linear programming duality is widely applicable, including in geometry.  For an
example of a very differently flavored application to isoperimetric problems,
see \cite{KK}.

We will use the following form of linear programming duality:
\begin{thm*}[Linear programming duality, standard form; 29.10 in \cite{CLRS}]
  For vectors in $\mathbb{R}^q$ for any $q$, use $\leq$ to denote
  coordinate-wise comparison.  Let $A$ be an $m \times n$ matrix, $\vec b \in
  \mathbb{R}^m$, $\vec c \in \mathbb{R}^n$.  Then the maximal value of $\vec c
  \cdot \vec x$, $\vec x \in \mathbb{R}^n$, subject to the constraints $A\vec x
  \leq \vec b$ and $\vec x \geq \vec0$, is the same as the minimal value of
  $\vec b \cdot \vec y$, $\vec y \in \mathbb{R}^m$, subject to the constraints
  $A^T\vec y \geq \vec c$ and $\vec y \geq \vec0$.
\end{thm*}
\begin{thm}[Isoperimetric duality] \label{thm:mfmc}
  Suppose $Y$ is a metric CW complex in which balls intersect a finite number
  of cells, and write $E_n(Y)$ for the set of $n$-cells of $Y$.  For $\mathbb{F}
  =\mathbb{Q}$ or $\mathbb{R}$, let $\omega \in C^{n+1}(Y;\mathbb{F}^r)$ be any
  cochain.  For each $e \in E_n(Y)$, fix a (perhaps asymmetric) polyhedral norm
  $N_e$ on $\mathbb{F}^r$, and let $N^\prime_e$ be the dual norm on
  $(\mathbb{F}^r)^*$.  Then the following are equivalent:
  \begin{enumerate}
  \item for all chains $\sigma \in C_{n+1}(Y;(\mathbb{F}^r)^*)$,
    $\langle\omega,\sigma\rangle \leq \sum_{e \in E_n(Y)}
    N^\prime_e(\partial\sigma(e))$;
  \item $\omega=d\alpha$ for a cochain $\alpha \in C^n(Y;\mathbb{F}^r)$
    with $N_e(\langle\alpha,e\rangle) \leq 1$ for every $e \in E_n(Y)$.
  \end{enumerate}
\end{thm}
Note that (1) is an isoperimetric condition: what may be termed the
``$\omega$-content'' of any chain $\sigma$ is bounded by the $N$-volume of its
boundary.  On the other hand, (2) says that $\omega$ is the coboundary of an
$N$-bounded cochain.
\begin{proof}
  Fix a radius $R$, let $*$ be a basepoint in $Y$, and let $\{e_i: i \in I\}$ be
  an enumeration of the $(n+1)$-cells that intersect the open ball $B_R(*)$, and
  $\{f_j:j \in J\}$ be an enumeration of the $n$-cells which either intersect
  $B_R(*)$ or are incident to $e_i$ for some $i$.  We denote the coefficient
  of $f_j$ in $\partial e_i$ by $\partial_j e_i$.  For each $j$, we can write
  the norm $N_{f_j}$ as
  \begin{equation} \label{eqn:norm}
    N_{f_j}(v_1,\ldots,v_r)=\max \left\{\sum_{k=1}^r \vec c(j,\ell)_kv_k: \ell=1,
    \ldots,L_j\right\}
  \end{equation}
  for constants $L_j$ and constant vectors $\vec c(j,\ell)$.  With these
  notations in place, condition (2) holds restricted to $B_R(*)$ if and only if
  the linear programming problem
  \begin{equation} \label{eqn:max}
    \text{maximize }\sum_{i \in I} \sum_{k=1}^r x_{i,k}\text{ subject to }\left\{
    \begin{aligned}
      &\text{for }i \in I, 1 \leq k \leq r, && 0 \leq x_{i,k} \leq
      \lvert\langle\omega,e_i\rangle_k\rvert & (A_{i,k})\\
      &\text{for }i \in I, 1 \leq k \leq r, && x_{i,k} =
      \sign(\langle\omega,e_i\rangle_k)\sum_{j \in J} \partial_je_i\alpha_{j,k} &
      (C_{i,k})\\
      &\text{for }j \in J, 1 \leq \ell \leq L_j, &&
      \sum_{k=1}^r c(j,\ell)_k\alpha_{j,k} \leq 1, & (B_{j,\ell})
    \end{aligned}
    \right.
  \end{equation}
  has the maximal possible solution, $x=\sum_{i \in I} \lVert\langle\omega,
  e_i\rangle\rVert_1$.  Here, the vectors
  $$(\alpha_{j,1},\ldots,\alpha_{j,r})=\langle \alpha, f_j \rangle$$
  are the values on $n$-cells of a cochain $\alpha$ which the inequalities
  $(B_{j,\ell})$ constrain to have $N_{f_j}(\langle\alpha,f_j\rangle) \leq 1$.
  The equations $(C_{i,k})$ guarantee that
  $$(\sign(\langle\omega,e_i\rangle_1)x_{i,1},\ldots,
  \sign(\langle\omega,e_i\rangle_r)x_{i,r})=\langle d\alpha, e_i \rangle$$
  are the values on $(n+1)$-cells of $d\alpha$, and the equations $(A_{i,k})$
  guarantee that these are no greater than and have the same signs as the values
  of $\omega$.  Thus $x=\sum_{i \in I} \lVert\langle\omega,e_i\rangle\rVert_1$ is
  a solution if and only if one can find an appropriate $\alpha$ with $d\alpha=
  \omega$.

  Note that this linear programming problem is not quite in the standard form
  quoted above.  To convert it, we can replace the equations $(C_{i,k})$ with
  two inequalities $(C_{i,k})$ and $(\overline{C_{i,k}})$ with opposite signs,
  and replace the variables $\alpha_{j,k}$ with differences $(\alpha_{j,k}-
  \overline{\alpha_{j,k}})$ where both $\alpha_{j,k}$ and
  $\overline{\alpha_{j,k}}$ are nonnegative.  When we take the dual of the
  resulting problem, the mirrored variables become mirrored inequalities and
  vice versa, and we can turn them back into equalities and perhaps-negative
  variables, respectively.  Thus by linear programming duality, $x$ is the
  solution to
  \eqref{eqn:max} if and only if it is
  \begin{equation}\begin{aligned} \label{eqn:min}
    \text{the minimal value of }&\sum_{i \in I}\sum_{k=1}^r \lvert\langle\omega,
    e_i\rangle_k\rvert A_{i,k}+\sum_{j \in J} \sum_{\ell=1}^{L_j} B_{j,\ell} \\
    \text{ subject to }&\left\{
    \begin{aligned}
      &\text{for }i \in I, 1\leq k\leq r, & A_{i,k}+C_{i,k} &\geq 1 & (x_{i,k})\\
      &\text{for }j \in J, 1 \leq k \leq r, & \sum_{\ell=1}^{L_j} c(j,\ell)_k
      B_{j,\ell} &= \sum_{i \in I} \sign(\langle\omega,e_i\rangle_k)
      \partial_je_iC_{i,k} & (\alpha_{j,k}) \\
      &\text{for }i \in I, 1 \leq k \leq r, & A_{i,k} &\geq 0 \\
      &\text{for }j \in J, 1 \leq \ell \leq L_j, & B_{j,\ell} &\geq 0.
    \end{aligned}
    \right.
  \end{aligned}\end{equation}
  Set $\sigma \in C_{n+1}(Y;(\mathbb{F}^r)^*)$ to be
  \begin{equation} \label{eqn:sigma}
    \sigma=\sum_{i \in I} \left(\sign(\langle\omega,e_i\rangle_1)C_{i,1},\ldots,
    \sign(\langle\omega,e_i\rangle_r\right)C_{i,r})e_i
  \end{equation}
  in the dual basis to the standard basis.  Then the right side of
  $(\alpha_{j,k})$ adds up to $\partial\sigma(f_j)_k$.  So picking the
  $B_{j,\ell}$ amounts to adding together nonnegative multiples of the vectors
  $\vec c(j,\ell)$ for each $\ell$, and $\sum_\ell B_{j,\ell}$ is the sum of
  these coefficients, i.e.\ the dual norm $N^\prime_{f_j}(\partial\sigma(f_j))$.
  Thus, keeping in mind that $A_{i,k} \geq 1-C_{i,k}$, we can bound the quantity
  $M$ being minimized in \eqref{eqn:min} as
  $$M \geq \sum_{i \in I}\sum_{k=1}^r \lvert\langle\omega,e_i\rangle_k\rvert
  (1-C_{i,k})+\sum_{j \in J} \sum_{j=1}^{L_j} B_{j,\ell}=x-
  \langle\omega,\sigma\rangle+\sum_f N^\prime_f(\partial\sigma(f)).$$
  If (1) is true, then $M \geq x$ for every $R$ and choice of $\sigma$.  On the
  other hand, if (1) is not true, then for some $R$, take a counterexample
  $\sigma \in B_R(*)$ for which $\langle\omega,\sigma\rangle>\sum_{e \in E_n(Y)}
  N^\prime_e(\partial\sigma(e))$.  We can scale such a $\sigma$ so that
  \eqref{eqn:sigma} is satisfied with $|C_{i,k}| \leq 1$ for every $i$ and $k$,
  and then set $A_{i,k}=1-C_{i,k}$ and the $B_{j,\ell}$ accordingly.  This gives
  us a vector which satisfies the constraints and results in $M<x$; thus
  condition (1) is equivalent to the dual problem \eqref{eqn:min} satisfying
  $M=x$ for every $R$.

  In other words, this demonstrates that (2) implies (1), and that (1) implies
  that for every $R$ there is an $\alpha_R$ which satisfies (2) when restricted
  to $B_R(*)$.  To get an $\alpha$ as desired on all of $Y$, we take a weak-*
  accumulation point of the $\alpha_R$.
\end{proof}
Mutatis mutandis, the same proof shows a version of the theorem with the role of
chains and cochains reversed.  Both versions have a number of corollaries
obtained through specific choices of norms.  The one which will be most useful
to us refers to \emph{$L_\infty$ cohomology}, that is, the cohomology of the
complex of cellular cochains whose values on cells are uniformly bounded.  (Note
that the somewhat similarly defined singular theory known as
\emph{bounded cohomology} results in a very different invariant of spaces; see
for example \cite{GerLost2}.)
\begin{cor} \label{cor:IsoDu}
  The following are equivalent for a cochain $\omega \in
  C^{n+1}(Y;\mathbb{F}^r)$:
  \begin{enumerate}
  \item there is a constant $K$ such that for all chains $\sigma \in
    C_{n+1}(Y;\mathbb{F})$, $\lVert\langle \omega,\sigma \rangle\rVert_\infty
    \leq K\vol(\partial\sigma)$;
  \item $\omega$ is an $L_\infty$ coboundary.
  \end{enumerate}
\end{cor}
\begin{proof}
  The corollary is obtained by applying the theorem separately to each
  coordinate of $\omega=(\omega_1,\ldots,\omega_r)$, with $N_e(\cdot)=
  \lvert\cdot\rvert$ for every $e$ and $\ell=1,\ldots,r$.
\end{proof}
The similar theorem for 0-chains proved in \cite{BWJAMS} is a corollary in the
following formulation:
\begin{cor}
  The following are equivalent:
  \begin{enumerate}
  \item the isoperimetric inequality for subsets of $Y^{(0)}$ is linear;
  \item the chain in $C_0^{(\infty)}(Y;\mathbb{F})$ which takes the value 1 on
    every vertex is an $L^\infty$ boundary.
  \end{enumerate}
\end{cor}
Notably, \cite{BWJAMS} also prove that this is equivalent to the triviality of
the entire zeroth homology group in this $L^\infty$ theory.  The Poincar\'e dual
of this result used in \cite{ABW} is another corollary of our theorem.

Similarly, if we take $r=1$, choose a basepoint $* \in Y$ and define $N_e(x)=
|x/f(d(*,e))|$ for some non-decreasing function $f:[0,\infty) \to [0,\infty)$,
we recover an analogous theorem for the groups $H_n^f$ introduced by
\cite{NoSpa}, generalizing their Theorem 4.2:
\begin{cor}
  The following are equivalent for a chain $\sigma \in C_{n+1}(Y,*;\mathbb{F})$:
  \begin{enumerate}
  \item there is a constant $K$ such that for all cochains $\omega \in
    C_{n+1}(Y;\mathbb{F})$, $\lVert\langle \omega,\sigma \rangle\rVert_\infty
    \leq K\sum_e f(d(*,e))|d\omega(e)|$;
  \item $[\sigma]=0 \in H_n^f(Y;\mathbb{F})$.
  \end{enumerate}
\end{cor}

\section{Finite approximations of Postnikov towers} \label{S:6}

Suppose now that $X$ is a delicate space, and once again let $\Gamma=\pi_1X$.
We would like to show that $X$ is built out of simpler spaces in an
easy-to-analyze way; we will then use this decomposition to complete the proofs
of Theorems \ref{thmA} and \ref{thmB}.  First, note that the homotopy fiber of
the inclusion $X \to B\Gamma$ is $\tilde X$ and hence rationally equivalent to
$\prod_{i=1}^r S^{2n_i+1}$.  Next, we see that up to rational homotopy $B\Gamma$
also has finite skeleta.
\begin{lem}
  Suppose $X$ is a finite complex and $\tilde X$ is rationally
  equivalent to a complex $F$ with finite skeleta.  Then $\Gamma=\pi_1X$ is of
  type $\mathcal{F}_\infty(\mathbb{Q})$, i.e.\ $B\Gamma$ is rationally
  equivalent to a complex with finite skeleta.
\end{lem}
\begin{proof}
  Let $f:X \to B\Gamma$ be the canonical map.  Since for every $n$, $\pi_n(f)
  \otimes \mathbb{Q} \cong \pi_{n-1}(F) \otimes \mathbb{Q}$ is
  finite-dimensional, by Theorem \ref{cor:XandY} $B\Gamma$ is rationally
  equivalent to a complex with finite skeleta.
\end{proof}
Note that $B\Gamma$ need not itself have finite skeleta.  Thus for example, if
$L$ is a flag triangulation of $\Sigma\mathbb{R}\mathbf{P}^2$, then the
Bestvina-Brady group $H_L$ is finitely presented and the fundamental group of a
$\mathbb{Q}$-aspherical 3-complex, but not of an aspherical complex with finite
skeleta \cite{BeBr}.

When $X$ is a delicate space, for any $N$, we get a sequence of compact spaces
$$\prod_{i=1}^r S^{2n_i+1} \to X \to B,$$
which maps to a fibration $\prod_{i=1}^r K(\mathbb{Q},2n_i+1) \to \hat X \to
B\Gamma$ via rational $N$-equivalences.  We can pick $N$ sufficiently large that
in a desired range, these maps rationally obey the homotopy exact sequence of a
fibration and the Serre spectral sequence.  Moreover, the fibration splits into
a tower of rational homotopy fibrations by products of same-dimensional spheres
which is rationally equivalent to a Postnikov tower.  In other words,
\begin{prop}
  If $X$ is a delicate space, then it satisfies conditions (1) and (2) of
  Theorem \ref{thmA}.
\end{prop}
More generally, let $X$ be a finite complex such that $\pi_n(X) \otimes
\mathbb{Q}$ is finite-dimensional.  By adding a few cells, we can get a finite
complex $B$ such that the pair $(B,X)$ is rationally $n$-connected and $\pi_n(B)
\otimes \mathbb{Q}=0$.  By studying such a map, we can get a handle on the
distortion of $\pi_n(X)$.

\subsection*{Almost Postnikov pairs and the Euler class}
These last very weak conditions are strong enough to analyze a step of our
approximate Postnikov tower.  So in the rest of this section, an
\emph{almost Postnikov pair} will denote a system $X \xrightarrow{p} B$ of
finite complexes where the map $p$ is rationally $n$-connected and the vector
space $V_p:=\ker(\pi_n(X) \to \pi_n(B)) \otimes \mathbb{Q}$ is
finite-dimensional.  Note that by adding cells to $B$, perhaps infinitely many,
one obtains a map, well-defined up to rational homotopy, whose rational homotopy
fiber is $K(V_p,n)$.  In fact, although this definition is much more general,
one loses very little for the purposes of this paper by thinking of all almost
Postnikov pairs as $K(\mathbb{Q}^r,n)$-fibrations for some $r$.

We will say that an almost Postnikov pair is \emph{normal} if in addition the
Hurewicz map on $V_p$ is injective in the universal cover $\tilde X$ and the
action of $\pi_1(X)$ on $V_p$ is elliptic.  This terminology is motivated by the
fact that in this case, the corresponding fibration of universal covers is
trivial up to rational homotopy.  That is, it behaves as a product.

We would like to define invariants for almost Postnikov pairs.  As a warmup, we
define the same invariant in a more traditional context.
\begin{defn}
  Let $K(\mathbb{Q},n)^r=F \to X \to B$ be a fiber bundle with monodromy
  representation $\rho:\pi_1B \to GL(H_n(F;\mathbb{Q}))$, with a corresponding
  $\mathbb{Q}\pi_1B$-module $M_\rho$.  This defines a bundle of groups
  $\mathcal{H}(\rho)$ over $B$ with fiber $H_n(F;\mathbb{Q})$.  Define a class
  $\eu \in H^{n+1}(B;\mathcal{H}(\rho))=H^{n+1}(B;M_\rho)$ as the obstruction to
  lifting a singular $(n+1)$-simplex in $B$ to $X$, given a lifting of the
  singular chain complex through dimension $n$.  By analogy, we refer to this as
  the \emph{Euler class} of the bundle, although one could just as well call it
  its \emph{$k$-invariant}.
\end{defn}
We need to show that the Euler class is well-defined, that is,
\begin{enumerate}
\item the cochain is a cocycle;
\item the choice of lifts of $n$-simplices determines it up to a coboundary.
\end{enumerate}
We can assume fixed lifts for simplices in $C_*(B)$ for $* \leq n-1$, since such
a lifting is unique up to homotopy.  Notice then that two ways to lift
$n$-simplices giving representatives $c_1$, $c_2$ of the Euler class give us a
cochain $b \in C^n(B;M_\rho)$ with $\delta b=c_1-c_2$.  This proves (2).

Now we prove (1).  Let $f:\Delta^{n+2} \to B$ be a simplex, and choose a
lifting of the singular chain complex through dimension $n$ which takes
$g:\Delta^n \to B$ to $\hat g$, which induces an obstruction cochain $c$.
Finally, let $\Delta_{ij}$ be the $n$-simplex which does not include vertices
$v_i$ and $v_j$.  To show that $\eu(\partial f)=0$, we lift the $n$-simplices
containing $v_0$, and then extend by homotopy lifting to a lift $\tilde f$ of
$f$.  Then the restrictions $\tilde f|\Delta_{i0}$ differ from
$\hat f|\Delta_{i0}$, but the sum of the obstructions on the $(n+1)$-simplices
is the same for both.

\S VI.5 of \cite{GJ} gives a more abstract homotopy-theoretic treatment of these
invariants.  In particular, they show that it is always possible to construct a
fiber bundle with the given monodromy representation and $k$-invariant, and that
conversely, Postnikov towers are determined up to homotopy by their monodromy
representations and $k$-invariants.  Up to rational homotopy, these results
restrict to finite complexes: if $B$ has finite skeleta, then by Theorem
\ref{cor:XandY}, one can find an $X$ that does as well.  In other words, given a
base space and a monodromy representation and Euler class, one can always
construct a corresponding almost Postnikov pair.

Conversely, suppose that $X \xrightarrow{p} B$ is an almost Postnikov pair.  The
map $p$ may not be a fibration, but there is nevertheless an Euler class
associated with the rational homotopy type of the fibration which it is
approximating.  We would like to come up with a cellular representative of the
Euler class in $C^{n+1}(X;M_\rho)$ which gives an obstruction to extending lifts
through $\pi$.  So let $Z_p$ be the mapping cylinder of $p$, and let
$j_n:C_n(\tilde Z_p;\mathbb{Q}) \to C_n(\tilde X;\mathbb{Q})$ be a lifting
homomorphism with a corresponding chain homotopy $u_n:C_n(\tilde Z_p;\mathbb{Q})
\to C_{n+1}(\tilde Z_p;\mathbb{Q})$.  Then taking an $(n+1)$-cell $c$ of
$\tilde B$ to
$$c+u_n(\partial c) \in H_{n+1}(\tilde Z_p,\tilde X;\mathbb{Q}) \cong
\pi_{n+1}(\tilde Z_p,\tilde X) \otimes \mathbb{Q} \twoheadrightarrow V_p$$
gives us the cellular representative we want.  Moreover, if the almost Postnikov
pair is normal, so that $V_p$ injects into $H_n(\tilde X;\mathbb{Q})$, we can
further identify the image of $c$ with
$j_n(\partial c) \in H_n(\tilde X;\mathbb{Q})$.

\subsection*{Distortion from the Euler class}
Already in the introduction we gave an example in which the Euler class relates
distortion in the total space of an almost Postnikov pair to isoperimetry in the
base space.  Now we can translate this example into the language that we will be
using.  Suppose that instead of a bundle, our almost Postnikov pair is actually
an injective cellular map---we can always guarantee this by taking a mapping
cylinder.

For the concrete case we are interested in, suppose $(B,X)$ is a CW pair with $B
\simeq T^4$ such that the homotopy fiber of the inclusion is $S^3$ and the Euler
class is the fundamental class $[T^4]$.  Then any admissible map $f:S^3 \to X$
represents an element of $\pi_3(X)$ which is, by definition, equal to the
pairing of the representative of the Euler class constructed above with a
filling of $f$.  Conversely, any boundary in $C_n(B)$ deforms via a lifting
homomorphism to a chain in $C_n(X)$ whose filling differs by a linear amount.
In this formulation, then, the equivalence between the functions
$\FV_{\mathbb{Z}^4,\langle \eu,\cdot \rangle}^{S^3}(k)$ and $\VD_\alpha(k)$ for $0
\neq \alpha \in \pi_3(X)$ is almost tautological.  This equivalence extends to
other almost Postnikov pairs with one-dimensional fiber.
\begin{thm} \label{thm:sph}
  Let $X \xrightarrow{p} B$ be a normal almost Postnikov pair in dimension $n$
  with $V_p \cong \mathbb{Q}$ and Euler class $\eu \neq 0$.  Then any element
  $\alpha \in V_p$ has distortion function $\VD_\alpha(k) \sim_C
  \FV_{B,\lvert\langle\eu,\cdot\rangle\rvert}^{S^n}(k)$.  In particular, this is true
  if $B \cong_{\mathbb{Q}} B\Gamma$ for some group $\Gamma$.
\end{thm}
\begin{proof}
  In this case, elliptic monodromy means that every element acts by
  multiplication by $1$ or $-1$, so by taking a double-cover we can assume that
  the monodromy is trivial.  Suppose that $\VD_\alpha(k)=M$, so that there is an
  admissible $f:S^n \to X$ with volume $k$ which represents the class
  $M\alpha \in \pi_n(X)$.  By Lemma \ref{lem:tohom}, we can find a nearby
  $g:S^n \to B$ with $[[g]]=0$, and so there are constants $C$ and $C^\prime$
  such that
  $$\FV_{B,\lvert\langle\eu,\cdot\rangle\rvert}^{S^n}(Ck) \geq M-C^\prime k.$$
  Conversely, suppose that $\FV_{B,\lvert\langle\eu,\cdot\rangle\rvert}^{S^n}(k)=M$,
  so that there is a map $f:S^n \to B$ with $[[f]]=0$ such that $\langle\eu,
  \Fill(f)\rangle=M$.  Then by Lemma \ref{lem:riv}, there is a constant $p$ such
  that one can find a $g:S^n \to X$ with $g_\#([S^n])=pj_nf_\#([S^n])$, where
  $j_n$ is a lifting homomorphism.  Therefore, there are constants $C$ and
  $C^\prime$ such that
  $$\VD_\alpha(pCk) \geq M-C^\prime k.$$
  This shows the desired equivalence.
\end{proof}
In other words, the fact that even admissible maps to $X$ and to $B$ do not
match up perfectly can in any case only impose linear differences in the
asymptotics of these two functions.

In general, however, in particular for almost Postnikov pairs with nontrivial
monodromy, the situation is rather more nebulous.  Here it is still the case
that small maps $S^n \to \tilde B$ with big directed fillings correspond in a
more or less one-to-one fashion with small maps $S^n \to \tilde X$ which
represent big homotopy classes in $\pi_n(\tilde X)$.  However, because of the
twisting, different lifts to the universal cover of the same cell in $B$
correspond to different elements of $\pi_n(X)$; the homotopy class of a map in
$X$ does not depend solely on its filling homology class in $B$.  This means,
for example, that the same sequence of fillings in $B$ may demonstrate the
distortion of many different elements of $\pi_n(X)$, depending on how we lift it
to the universal cover.  Moreover, if the homotopy fiber is $(S^n)^r$ for some
$r \geq 2$, then not every map represents a multiple of some given $\alpha$.
These two complications prevent us from giving a neat characterization like
above, and we are left with a statement which is even more nakedly tautological:
\begin{thm} \label{thm:genDI}
  Let $X \xrightarrow{p} B$ be a normal almost Postnikov pair with $p$ an
  injective cellular map.  Let $j_*:C_{\leq n}(\tilde B;\mathbb{Q}) \to
  C_{\leq n}(\tilde X;\mathbb{Q})$ be a lifting homomorphism, and $\omega_j$ be
  the representative of the Euler class induced by $j_n$.  Then there are
  constants $0<c\leq1,p_n(j_*)\geq1$ such that for every $\alpha \in V_p$,
  \begin{equation} \label{eqn:DI}
    \frac{c}{p_n}|p_n\alpha|_{\vol} \leq \min\{\vol(f) \mid f:S^n \to \tilde B
    \text{ admissible and }\langle\omega_j,\Fill(f)\rangle=\alpha\} \leq
    C|\alpha|_{\vol}+C.
  \end{equation}
\end{thm}
\begin{proof}
  The right inequality is true since any map $f:S^n \to X$ deforms to an
  admissible one.  Conversely, for any $f:S^n \to B$ we can apply Lemma
  \ref{lem:riv} to concoct a map $g:S^n \to X$ representing $p\alpha$ with
  $g_\#[S^n]=p_nj_n(f_\#[S^n])$ and $\vol g \leq C_n\vol f$, where $p_n$ and
  $C_n$ depend only on $j_n$.
\end{proof}
This fact seems rather unwieldy to use, but it does allow us to analyze certain
examples.  It also has the useful consequence that homological isoperimetric
functions give us a bound on how distorted classes may be, regardless of
monodromy.
\begin{cor} \label{cor:FVbd}
  Given a normal almost Postnikov pair $X \xrightarrow{p} B$, for any $\alpha
  \in V_p$, $\VD_\alpha(k) \lesssim \FV_B^n(k)$.
\end{cor}
This stands in contrast with almost Postnikov pairs which are not normal, in
which volume distortion is always infinite.

In concrete examples, the main principle is as follows: to show a class $\alpha
\in \pi_n(X)$ to be distorted, one needs to find a sequence of boundaries in
$\tilde B$ which lift to multiples of $\alpha$.  The exact lift depends on the
exact representative of the Euler class chosen, but any pair of representatives
gives lifts which diverge from each other linearly.  Thus in principle, if
$\alpha$ is distorted, then we can use any representative of the Euler class to
demonstrate this.
\begin{exs}
  In these examples, we assume $n$ odd, so that $K(\mathbb{Q},n)
  \cong_{\mathbb{Q}} S^n$.
  \begin{enumerate}
  \item Mineyev \cite{MinIP} shows that hyperbolic groups satisfy linear
    isoperimetric inequalities for fillings of boundaries with cycles.  In
    particular, if $\Gamma$ is hyperbolic, $\FV_\Gamma^n(k)$ is linear for all
    $n \geq 1$, and hence so is
    $\FV_{\Gamma,\lvert\langle e,\cdot \rangle\rvert}^n(k)$ for any
    $e \in H^{n+1}(\Gamma;\mathbb{Q})$.  Therefore, if $X$ is the total space
    of a fibration $(S^n)^r \to X \to B\Gamma$, then $\pi_n(X)$ is undistorted.
  \item Conversely, if $\Gamma=\mathbb{Z}^d$ for $d>n$, then the isoperimetric
    inequality $k^{(n+1)/n}$ is attained by fillings of round spheres in
    $(n+1)$-dimensional coordinate hyperplanes, which correspond to the
    generators of $H^{n+1}(\Gamma)$.  Suppose that $X$ is a bundle over $T^r$
    with finite monodromy.  Given $0 \neq \omega \in H^{n+1}(\Gamma)$,
    $\langle \omega,x \rangle \neq 0$ for one of these generators $x$, so
    $\FV_{\Gamma,\lvert\langle e,\cdot \rangle\rvert}^n(k) \sim k^{(n+1)/n}$, and if
    $X$ has Euler class $\omega$, then $\pi_n(X)$ is distorted with
    $VD_\alpha(k) \sim k^{(n+1)/n}$.
  \item A somewhat more subtle situation occurs when $\Gamma=\mathbb{Z}^d$ and
    the almost Postnikov pair $X \to B\Gamma$ has monodromy of infinite order.
    As a concrete example, we take $\Gamma=\mathbb{Z}^5=\langle a,b_1,\ldots,b_4
    \rangle$ and homotopy fiber $(S^3)^3$.  Let $\hat a$ and $\hat b_i$ be the
    4-cells in $\mathbb{R}^5$ which are orthogonal to the corresponding
    generators of the fundamental group.  Suppose that the monodromy
    representation $\rho$ takes
    $$a \mapsto A=\begin{pmatrix}1&0&0\\0&3/5&-4/5\\0&4/5&3/5\end{pmatrix}
    \text{ and }b_i \mapsto I_3.$$
    Then
    $$H^*(\Gamma;M_\rho) \cong H^*(\langle a \rangle;\mathbb{Z}^3/\rho(a))
    \otimes H^*(\langle b_i \rangle;\mathbb{Z}^3),$$
    and in particular $H^4(\Gamma;M_\rho) \cong \mathbb{Q}^7$, generated by
    cochains sending each of the five 4-dimensional flats to any vector, mod the
    last two coordinates of the image of each $\hat b_i$.

    If the Euler class of the fibration takes $\hat a$ to a vector $\vec v$,
    then the situation restricted to a flat perpendicular to $a$ is the same as
    in the previous example, so $\vec v$ is distorted polynomially, with
    $\VD_{\vec v}(k) \sim k^{4/3}$.  Moreover, so is $\rho(a)^k(\vec v)$, for any
    $k$, as we can see by taking the relevant parallel flat.  Indeed, so is any
    rational linear combination of such vectors, that is, any $\rho(m)\vec v$
    for $m \in \mathbb{Q}\Gamma$.

    On the other hand, suppose the Euler class takes $\hat b_1$ to a vector
    $\vec v$, and consider the boundaries of cochains in $C^3(B;M_\rho)$.  Of
    the 3-cells bounding $\hat b_1$, the ones perpendicular to $\hat a$ can be
    lifted by such cochains to any pair of vectors differing by multiplication
    by $A$.  Since $A-I$ is of rank 2, when choosing a cellular representative
    of the Euler class we can choose to lift $\hat b_1$ to any vector as long as
    it has the correct first coordinate.  If we choose the lift to be
    $(v_1,0,0)$, then this vector is distorted by the argument in the previous
    example.  On the other hand, this does not induce distortion in any vectors
    in other directions.

    Thus vectors with nonzero second and third coordinate can only be distorted
    as a result of the pairing of the Euler class with $\hat a$, but any flat
    may cause distortion of vectors of the form $(v_1,0,0)$.
  \item Now take $\Gamma=\dmd_n$, the $n$th diamond group defined earlier, and
    its classifying space $X_n$, and suppose that $Y$ is the total space of a
    rational homotopy fibration $(S^n)^r \to Y \to X_n$.  From the earlier
    discussion we see that if the monodromy of this fibration is trivial and the
    Euler class is nontrivial, then the volume distortion function of $Y$ will
    be $\exp(k^{1/n})$.  Indeed, we can also demonstrate this for certain other
    monodromy representations.  As before, we let $a,b_1,\ldots,b_n,c_1,\ldots,
    c_n$ be the generators of $\dmd_n$.  Suppose that the monodromy
    representation $\rho:\dmd_n \to GL_r(\mathbb{Q})$ of our fibration takes $a
    \mapsto 1$ and maps the $b_i$ and $c_i$ via any elliptic representation of
    $F_2^n$, and that the Euler class $\eu \in H^{n+1}(X_n;M_\rho)$ is nonzero.

    Note that $H^{n+1}(X_n;M_\rho) \cong \mathbb{Q}^r$.  To see this, fix a lift
    $\widetilde{e_I}$ of each $(n+1)$-cell $e_I$ of $X_n$ to $\tilde X_n$ so
    that all of these lifts coincide on the edge $a^{2^n}$, as illustrated in
    Figure \ref{fig:sigma}.  In this basis, coboundaries in $C^{n+1}(X_n;M_\rho)$
    are those $\omega$ which satisfy the vector equation
    $$A(\omega):=\sum_{\substack{I \in \{b,c\}^n\\\square_i=b_i\text{ or }c_i}}
    (-1)^{\#\{i:\square_i=b_i\}} \prod_{i=1}^n (\rho(\square_i^{-1})-2I)^{-1}
    \langle\omega, \widetilde{e_I}\rangle=\vec 0.$$
    The notation indicates that the sum runs over all choices of either $b_i$ or
    $c_i$ for each $i$.  In particular, to simplify the notation in the rest of
    this example, we can choose a representative $\omega$ of $\eu \in
    H^{n+1}(X_n;M_\rho)$ which is zero on the lifts of every cell except for
    $e_C=e_{\{c,\ldots,c\}}$.

    As an upper bound on the distortion function of $\pi_n(Y)$, we already know
    that $\FV_{\dmd_n}^n(k) \sim \exp(k^{1/n})$.  To show that this bound is
    sharp, we construct a sequence of small representatives of large elements of
    $\pi_n(Y)$.  Specifically, we pick the embedding $\partial\tau_n(k+1)$ of
    $S^n \subset \tilde X_n$ described previously, which has surface area
    $O(k^n)$. Since this has filling $\tau_n(k+1)$, it lifts to a representative
    in $\pi_n(Y) \otimes \mathbb{Q} \cong M_\rho$ of
    $$\vec v(k+1) :=
    \sum_{\substack{j_1,\ldots,j_n\geq0\\j_1+\cdots+j_n \leq k}} 2^{k-\sum_{\ell=1}^nj_\ell}
    \rho(c_1)^{-j_1}\cdots\rho(c_n)^{-j_n}\langle\omega, \widetilde{e_C}\rangle.$$
    Now we multiply both sides of this equation by $\prod_{i=1}^n(\rho(c_i^{-1})
    -2I)$ and notice that the transformation on the right hand side can be
    thought of as $\rho(m_{n,k})$ for a certain element $m_{n,k} \in \mathbb{Z}^r
    \dmd_n$:
    $$\prod_{i=1}^n(\rho(c_i^{-1})-2I)\vec v(k+1)=
    \rho(m_{n,k})\langle\omega, \widetilde{e_C}\rangle.$$
    Then there is an inductive formula
    $$m_{n,k}=(c_n^{-1}-2)\sum_{j=0}^k c_n^j m_{n-1,k-j}.$$
    By cancelling certain terms at each step of this induction, we find that
    $$\prod_{i=1}^n(\rho(c_i^{-1})-2I)\vec v_n(k+1)=
    \left[-2^{k+n}+R(\rho(c_1),\cdots,\rho(c_n))\right]\langle\omega,
    \widetilde{e_C}\rangle \in \pi_n(Y) \otimes \mathbb{Q},$$
    where the remainder term $R$ is a polynomial in the $\rho(c_i)$ with
    $O(k^n)$ terms, all of which have integer coefficients in $O(2^n)$.  Indeed,
    since the argument was on the level of elements of $\mathbb{Z}^r\dmd_n$,
    one can construct a finitely-presented module $M$, and hence a finite
    $Y$, such that the equality holds even before rationalizing:
    $$\prod_{i=1}^n (c_i^{-1}-2) \cdot \alpha(k+1)=\left[-2^{k+n}+R(c_1,\cdots,
      c_n)\right] \cdot [\partial\widetilde{e_C}] \in \pi_n(Y) \cong M,$$
    where $\alpha(k+1) \in \pi_n(Y)$ is the element represented by the
    equivariant lift of $\tau_n(k+1)$ induced by $\omega$.  By Corollary
    \ref{cor:bo_pi}, for any $Y$ with the same rational homotopy type, the
    same equality holds once both sides are multiplied by an integer $p(Y)$.

    Thus we can represent the element $-2^{k+n}[\partial\widetilde{e_C}] \in
    \pi_n(Y)$ via $O(3^n)$ copies of the lift of $\tau_n(k+1)$ summed with
    $O(2^nk^n)$ copies of various inclusions of the fiber, hence with volume
    $O(k^n)$.  This means that $[\partial\widetilde{e_C}]$ is distorted with
    distortion function $\exp(\sqrt[n]{k})$.  Moreover, so is any other lift of
    this cell or linear combination of such lifts; that is, for any $m \in
    \mathbb{Z}\dmd_n$, $\rho(m)[\partial\widetilde{e_C}]$ has the same
    distortion function.  Moreover, we can use other embeddings of $S^n$,
    depending on choices of branches in the Bass-Serre tree, to obtain other,
    perhaps distinct distorted vectors.  Thus the one-to-one relationship
    between homology classes downstairs and homotopy classes upstairs that we
    see in the case of trivial monodromy, as exemplified by Theorem
    \ref{thm:sph}, does not have an equivalent in the general case.
  \end{enumerate}
\end{exs}
\subsection*{Recognizing non-distortion}
We have just demonstrated that for an almost Postnikov pair $X \xrightarrow{p}
B$, while the volume distortion of elements of $V_p$ is in some sense determined
by the monodromy representation and the Euler class, this relationship can be
complicated and perhaps even impossible to encapsulate.  However, there is a
fairly simple criterion that can be used to identify pairs for which all of
$V_p$ is volume-undistorted.
\begin{defn}[\cite{GrHodge}]
  Let $X$ be a compact piecewise Riemannian space.  A form $\omega \in
  \Omega^n(X)$ is called $\tilde d(\text{bounded})$ if its lift $\tilde\omega$
  to the universal cover $\tilde X$ is the differential of a bounded form.
\end{defn}
Being $\tilde d(\text{bounded})$ is a Lipschitz homotopy invariant and (in the
given setting) exact forms are $\tilde d(\text{bounded})$.  Indeed, this
definition is just another way of saying that the pullback of a form to the
universal cover is zero in $L_\infty$ cohomology.

In the following theorem, we put together this as well as a few other equivalent
conditions for non-distortion.  In particular, the presence of condition (5)
gives a weak converse to Lemma \ref{lem:duh}.
\begin{thm} \label{thm:LInfGen}
  Let $X \xrightarrow{p} B$, $n \geq 3$, be a normal almost Postnikov pair with
  elliptic monodromy representation $\rho:V_p \to GL(H_n(F;\mathbb{Q}))$ and
  Euler class $\eu \in H^{n+1}(B;M_\rho)$.  Write $\tilde X$ and $\tilde B$ for
  the universal covers of $X$ and $B$.  Then the following are equivalent:
  \begin{enumerate}
  \item The subspace $V_p \subseteq \pi_n(X) \otimes \mathbb{Q}$ is
    volume-undistorted.
  \item For some (any) representative $\omega$ of $\eu$, there is a constant $C$
    such that for all spherical boundaries $\sigma \in B_n(\tilde B)$,
    $\lVert\langle\omega,\Fill(\sigma)\rangle\rVert \leq C\vol\sigma$.
  \item For some (any) representative $\omega$ of $\eu$, there is a constant $C$
    such that for all boundaries $\sigma \in B_n(\tilde B)$,
    $\lVert\langle\omega,\Fill(\sigma)\rangle\rVert \leq C\vol\sigma$.
  \item The class $\eu$ is $\tilde d(\text{bounded})$, i.e.\ $\pi^*\eu=0 \in
    H^{n+1}_{(\infty)}(\tilde B;V_p)$.
  \item There is a bounded cellular cocycle $w \in C^n_{(\infty)}(\tilde X;V_p)$
    so that $\langle w, f_*[S^n]\rangle=[f]$ for any $f:S^n \to \tilde X$ with
    $[f] \in V_p$.  Equivalently, if $X$ is given a piecewise Riemannian
    structure, there is a bounded closed piecewise smooth $V_p \otimes
    \mathbb{R}$-valued form $\omega \in \Omega^n(\tilde X)$ so that
    $\int_{S^n} f^*\omega=[f]$ for any such $f:S^n \to \tilde X$.
  \end{enumerate}
\end{thm}
\begin{proof}
  The only part which we have not already essentially proved is
  (4)$\Rightarrow$(5).  In proving it, we abuse notation by using $\pi$ and $p$
  for the maps $\tilde X \xrightarrow{\pi} X$ and $\tilde X \xrightarrow{p}
  \tilde B$ which complete the commutative square.  We write $\Gamma$ for
  $\pi_1B=\pi_1X$.

  By perhaps taking a mapping cylinder, we assume that $X \subset B$.  Let
  $j_*:C_*(\tilde B;\mathbb{Q}) \to C_*(\tilde X;\mathbb{Q})$ be a lifting
  homomorphism, and let $\alpha$ be the cellular cocycle representing $\eu$
  which we previously discussed, which takes a cell $c$ to the preimage in $V_p$
  of the homology class of $j_n(\partial \tilde c) \in H_n(\tilde X;\mathbb{Q})$
  for any lift $\tilde c$ of $c$ to $\tilde B$.  By assumption, there is a
  bounded cellular $\beta \in C_n(\tilde B;V_p)$ such that $d\beta=\pi^*\alpha$.
  But from the definition of $\alpha$ we see that the restriction $w=p^*\beta$
  is a cellular cocycle on $\tilde X$ which satisfies the condition of (5).

  The technical report \cite{AB} contains a proof of a de Rham theorem which
  equates simplicial $L_p$ cohomology with the cohomology of $L_p$ forms, for
  $p$ including $\infty$.  They state this theorem for manifolds, but in fact it
  works for piecewise smooth forms on any simplicial complex of bounded
  geometry.  If $X$ has a piecewise Riemannian structure, then there is a
  compatible cellular, Lipschitz homotopy equivalence with a simplicial complex,
  so this theorem carries over to a cellular version.  Thus we can also produce
  a form with the desired properties in $\tilde X$ as a stratified space, or in
  any homotopy equivalent manifold with boundary.

  (5)$\Rightarrow$(1) is Lemma \ref{lem:duh}.

  (1)$\Rightarrow$(2) is a special case of Theorem \ref{thm:genDI}.

  (2)$\Rightarrow$(3) is Lemma \ref{lem:BBFS}; this is the only place we use the
  stipulation that $n \neq 2$.

  To get (3)$\Rightarrow$(4), apply Corollary \ref{cor:IsoDu} to a cellular
  cochain representing $\eu$.
\end{proof}
\begin{proof}[Proof of Theorem \ref{thmA}]
  We have already shown that if a finite CW complex $X$ has no distortion in its
  homotopy groups, then conditions (1) and (2) of Theorem \ref{thmA} are
  satisfied.  In particular, we can find a tower of maps
  $$X=X_m \to X_{m-1} \to \cdots \to X_0$$
  where each $X_i$ is a finite complex which is rationally equivalent up to a
  large dimension $N$ to $X_{(n_i)}$, and where each step $X_k \to X_{k-1}$ is a
  normal almost Postnikov pair.  By Theorem \ref{thm:LInfGen}, $\pi_{n_k}(X_k)$
  is undistorted if and only if the Euler class of $X_k \to X_{k-1}$ is
  $\tilde d($bounded$)$.  Since $X_k$ is rationally equivalent to $X$ up to
  dimension $k+1$, this is also the case for $\pi_{n_k}(X)$.  Thus if $X$ is a
  delicate space, then the lack of distortion in its homotopy groups is
  equivalent to condition (3).
\end{proof}
\begin{proof}[Proof of Theorem \ref{thmD}]
  Let $X \xrightarrow{p} B$ be an almost Postnikov pair.  Theorem \ref{thm:sc}
  and Corollary \ref{cor:md} show that if it is not normal, then $V_p$ has
  infinite distortion.  When it is normal, nondistortion is equivalent to the
  condition that $\pi^*\eu=0 \in H^{n+1}_{(\infty)}(\tilde B;V_p)$ by Theorem
  \ref{thm:LInfGen}.  Thus in general, nondistortion is equivalent to the two
  given conditions.
\end{proof}

Having proven our main theorem, we break for a meditation on what Theorem
\ref{thm:LInfGen} really means.  If $F \to X \to B$ is an honest fiber
bundle with $F \to \tilde X \to \tilde B$ homotopically trivial, one can think
of condition (4) as specifying that there is a section $\sigma:\tilde B \to
\tilde X$ with a bounded amount of ``twisting'' around the fiber for every
$n$-cell.  What is bounded \emph{a priori} is the number of twists, but this is
equivalent to being able to find a section with a bounded Lipschitz constant,
or, say, $n$-dilation.  This in turn induces a trivializing map $\tilde X \to F
\times \tilde B$ which is also bounded in all the same senses, that is, the
bundle $\tilde X \to \tilde B$ is ``coarsely trivial'' in a natural sense.  It
is tempting to assert that when $X$ has undistorted homotopy groups, its
universal cover must necessarily be coarsely trivial as a rational homotopy
fibration, whatever that may mean.  However, outside the important special case
of actual fiber bundles, and especially when $\rho$ is outside
$GL(r,\mathbb{Z})$, it's much harder to make such an assertion precise.

We now explore how this result applies to specific classes of groups and spaces.
\begin{ex}[Amenable groups]
  If $\Gamma=\pi_1B$ is an amenable group, then only the zero class in $H^{n+1}(
  B;M_\rho)$ is $\tilde d($bounded), i.e.\ the pullback map $H^{n+1}(B;M_\rho)
  \to H^{n+1}_{(\infty)}(\tilde B;\mathbb{Q}^r)$ is injective.  A special case of
  this fact was noted in \cite{ABW}.  In general, we can see this as follows.
  An invariant mean $\mu$ on $\Gamma$ provides a map
  $\int:C^*(\tilde B;\mathbb{Q}^r) \to C^*(B;M_\rho)$, where for any cell $e$,
  $$\left\langle{\textstyle\int}\omega,e\right\rangle=\int_{\gamma \in \Gamma}
  \rho(\gamma)^{-1}(\langle \omega, \gamma \cdot e\rangle)d\mu.$$ 
  This map commutes with the coboundary operator and thus is a well-defined
  projection on the level of homology which inverts the pullback map.  Thus in
  spaces with amenable fundamental group, any nonzero Euler class induces
  distortion.
\end{ex}
\begin{ex}[Unit tangent bundles of aspherical manifolds]
  Suppose that $\Gamma$ is a group such that $M=B\Gamma$ is an $n$-dimensional
  smooth manifold, and let $X$ be the unit tangent bundle of $M$.  If $\Gamma$
  is amenable, the Euler characteristic of $M$, and thus the Euler class of the
  bundle $X \to M$, is zero by a result of Cheeger and Gromov \cite{CheeGr}.  On
  the other hand, if $\Gamma$ is non-amenable, the fundamental class of $M$
  pulls back to zero in $H^n_{(\infty)}(X;\mathbb{Q})$, as shown in \cite{ABW}.
  Thus in any case, $\pi_n(X)$ is undistorted.
\end{ex}
\begin{ex}[Hyperbolic groups] \label{exs:AmHy}
  On the other end of the spectrum, suppose $\Gamma$ is a hyperbolic group.  As
  discussed before, $B\Gamma$ satisfies a linear isoperimetric inequality for
  fillings of $n$-boundaries with chains for $n \geq 1$, and this means, dually,
  that $H^{n+1}_{(\infty)}(\widetilde{B\Gamma};\mathbb{Q}^r)=0$ when $n+1 \geq 2$.
  Thus the $L_\infty$ cohomology of $\widetilde{B\Gamma}$ is concentrated in
  dimensions 0 and 1.  In particular, this means that if $X$ is the total space
  of a fibration $(S^n)^r \to X \to B\Gamma$, then $\pi_n(X)$ is always
  undistorted.  Can we say the same for rational Postnikov towers with more
  steps?  To answer this question, we need to analyze $L_\infty$ cohomology more
  closely.

  It turns out that in dimension 1, the $L_\infty$ cohomology is as large as
  possible, as far as we're concerned.
  \begin{lem} \label{lem:H1}
    For any finite complex $B$ with universal covering map $\pi:\tilde B \to B$
    and any elliptic monodromy representation $\rho:\pi_1B \to GL(H_n((S^n)^r;
    \mathbb{Q}))$, the pullback map $\pi^*:H^1(\tilde X;M_\rho) \to
    H^1_{(\infty)}(\tilde X;\mathbb{Q})$ is injective.
  \end{lem}
  \begin{proof}
    Suppose $\omega \in C^1(B\Gamma;M_\rho)$ is a representative of a nonzero
    cohomology class; in other words, there is a loop $\gamma:[0,1] \to B\Gamma$
    such that $\vec v=\langle\omega,\gamma\rangle \notin \img(\rho([\gamma])-
    I)$.  Then $\rho(\gamma)$ is conjugate to a rotation and $\vec v$ has a
    nonzero projection onto its invariant subspace in $\mathbb{R}^r$.  Thus
    $\sum_{i=0}^k \rho^k(\vec v)$ has unbounded length, so if $\sigma_k$ is a
    lift of $k\gamma$ to $\tilde B$, then $\langle\pi^*\omega,\sigma_k\rangle$
    grows linearly in $k$ but $\vol(\partial\sigma_k)=2$.  By Corollary
    \ref{cor:IsoDu}, this means that $0 \neq \pi^*\omega \in
    H^1_{(\infty)}(\tilde B;\mathbb{Q}^r)$.
  \end{proof}
  To extend this dichotomy to other spaces with the same fundamental group, we
  can use a Serre spectral sequence.
  \begin{lem}
    Let $F \to X \to B$ be a rational homotopy fibration of finite complexes
    with $F=\prod_i S^{2n_i+1}$.  Then the $L_\infty$ cohomology of the universal
    cover of this fibration obeys a Serre spectral sequence; that is,
    $H^*_{(\infty)}(\tilde X;\mathbb{Q}^r)$ can be assembled from the $E_\infty$
    page of a spectral sequence with $E_2^{p,q}=
    H^p_{(\infty)}(\tilde B;H^q(F;\mathbb{Q}^r))$.
  \end{lem}
  \begin{proof}
    There is nothing tricky to this argument, and we will merely sketch it.

    We can assume that $X$ is built inductively over the skeleta of $B$ as in
    the proof of Theorem \ref{cor:XandY}, with $X_k$ approximating the total
    space of a fibration over $B^{(k)}$.  This gives a filtration of $X$ and
    hence of $C^n_{(\infty)}(\tilde X;\mathbb{Q})$.  Following a standard proof
    such as that of Theorem 5.15 in \cite{HatSS}, one shows that the spectral
    sequence associated to this filtration satisfies the requirements we have
    given.
  \end{proof}
  Now suppose that we have a complex $B$ which is built up to rational homotopy
  as a fibration $F=\prod_i S^{2n_i+1} \to B \to B\Gamma$.  In the case that the
  base is $B\Gamma$ for a hyperbolic group $\Gamma$, the columns of the $E_2$
  page are zero for $p \geq 2$.  Thus if $X$ is the total space of a further
  fibration $(S^n)^r \to X \to B$, distortion occurs in $\pi_n(X)$ if and only
  if the relevant Euler class is represented in the $E_2$ page of the spectral
  sequence by a nonzero element of $H^1(\Gamma;H^n(F;\mathbb{Q}^r))$.  As will
  be proved in Theorem \ref{thm:inf}, such distortion is always weakly infinite,
  and so if $X$ has a hyperbolic fundamental group, then all subspaces of
  $\pi_n(X)$ are either undistorted or weakly infinitely distorted.

  The work of Gersten \cite{GerLost} and Mineyev \cite{MinIP} shows that
  hyperbolic groups are completely characterized by $L_\infty$ cohomology: that
  is, they are exactly those $\Gamma$ for which $H^n_{(\infty)}(\Gamma;V)=0$ for
  every $n \geq 2$ and every normed real vector space $V$; or, equivalently, for
  which $H^2_{(\infty)}(\Gamma;\ell^\infty)=0$.  It is therefore tempting to
  conjecture that the dichotomy outlined above holds \emph{if and only if}
  $\Gamma$ is hyperbolic.  In other words, if $\Gamma$ is not hyperbolic, one
  ought to always be able to hit nonzero elements of $H^n_{(\infty)}(\Gamma;
  \mathbb{Q}^r)$ by pulling back elements of $H^n(\Gamma;M_\rho)$ for some
  module $M_\rho$, and thus find distortion in higher homotopy groups which is
  not weakly infinite.  In fact, the Baumslag-Solitar group $BS(1,2)$ is a
  counterexample: it is not hyperbolic, but has a classifying space of dimension
  2, and $H^2(BS(1,2);M_\rho)=0$ for every finite-dimensional module $M_\rho$
  corresponding to an elliptic representation.  Thus the above dichotomy also
  holds for spaces $X$ with fundamental group $BS(1,2)$.
\end{ex}
\begin{ex}[Bounded cohomology classes]
  It is well known that bounded cohomology classes of groups are zero in
  $L_\infty$ cohomology: for example, this is Lemma 10.3 in \cite{GerLost2}.
  Thus bounded Euler classes always give rise to undistorted homotopy groups.
  However, the converse is not true: for example, let $e=[M^n \times S^1]$,
  where $M^n$ is any hyperbolic $n$-manifold.  This fundamental class is not
  bounded, since geodesic simplices in $\mathbb{H}^n \times \mathbb{R}$ have
  volume proportional to their height in the $\mathbb{R}$-direction, rather than
  bounded volume.  On the other hand, given a map $f:S^n \to \mathbb{H}^n \times
  \mathbb{R}$, we can deform $f$ slightly so that it is an immersion whose
  projection to the $\mathbb{R}$ factor is Morse.  Then both the area of $f$ and
  the volume of a filling are determined by integrating hyperbolic areas and
  volumes over $\mathbb{R}$, and hence one is linear in the other.  Thus for any
  bundle $S^n \to X \to M^n \times S^1$, $\pi_n(X)$ is undistorted.
\end{ex}

\section{Infinite and weakly infinite distortion}
We now turn to the characterization of infinite distortion promised in Theorem
\ref{thmB}.  As pointed out in the introduction, given a finite CW complex $X$,
we can define a minimal cellular volume functional on the homology groups
$H_n(\tilde X)$ of its universal cover.  By Lemma \ref{lem:BBFS}, when $n \geq
3$, this extends the minimal volume functional on $\pi_n(X)$, and so we can
consider the possibility that $\pi_n(X)$ is weakly infinitely distorted in
$H_n(\tilde X)$.  Later in the section we will also discuss the extent to which
this implies infinite distortion.

The theorem that follows characterizes the extent to which the resulting
subspace is weakly infinitely distorted if $X$ is a delicate space; it plays the
role in the proof of Theorem \ref{thmB} that Theorem \ref{thm:LInfGen} played
for Theorem \ref{thmA}.
\begin{thm} \label{thm:inf}
  Suppose $X \xrightarrow{p} B$ is a normal almost Postnikov pair with monodromy
  representation $\rho:\Gamma:=\pi_1(B) \to GL(V_p)$ and Euler class $\eu \in
  H^{n+1}(B;M_\rho)$.  Assume also that $\pi_k(B) \otimes \mathbb{Q}$ is finitely
  generated for $k \leq n$, so that we can find a sequence of finite complexes
  $F \to B \xrightarrow{p^\prime} A$ which is a rational homotopy fibration up
  to dimension $n$ and such that the universal cover $\tilde A$ is
  $n$-connected.  Let $\hat F$ be a rational homotopy fiber with finite
  $n$-skeleton of $p^\prime \circ p:X \to A$.  Then the following are
  equivalent:
  \begin{enumerate}
  \item There is no weakly infinite volume distortion in $V_p \subseteq \pi_n(X)
    \otimes \mathbb{Q}$ as a subset of $H_n(\tilde X;\mathbb{Q})$.
  \item For every $\gamma \in \Gamma$, $0=\gamma^*\eu \in H^{n+1}(U;
    M_\rho|_{\langle\gamma\rangle}))$, where $U$ is the total space of
    $\gamma^*p^\prime$.
  \item[($2^\prime$)] $0=\iota_1^*\eu \in H^{n+1}\left((p^{\prime})^{-1}\left(
    A^{(1)}\right);M_{\rho \circ \iota_{1*}}\right)$, where $\iota_1$ is the
    inclusion of the 1-skeleton $A^{(1)} \hookrightarrow A$.
  \item[($2^{\prime\prime}$)] In the Serre spectral sequence for $L_\infty$
    cohomology, the Euler class $\eu \in H^{n+1}_{(\infty)}(\tilde B;V_p)$
    is represented by $0\in E_2^{1,n}=H^1_{(\infty)}(\tilde A;H^n(F;V_p))$.
  \item $H_n(\tilde X;\mathbb{Q})$ splits as a direct sum of
    $\mathbb{Q}\Gamma$-modules $h_n(V_p) \oplus P$.
  \item[($3^\prime$)] There is a subspace $P \subset H_n(\hat F;\mathbb{Q})$
    which is invariant under the homological monodromy representation $\hat\rho:
    \Gamma \to GL(H_n(\hat F;\mathbb{Q}))$ of $\hat p$, such that $P \oplus
    h_n(V_p)=H_n(\hat F;\mathbb{Q})$.
  \end{enumerate}
\end{thm}
\begin{ex} \label{ex:S1}
  Let $B=S^1 \times (S^3)^3$ and $X$ be an $S^9$-bundle over $B$ with Euler
  class $\eu=[B] \in H^{10}(X;\mathbb{Q})$.  Then $\pi_9(X)$ is infinitely
  distorted, informally because arbitrarily long tubes have the same size
  boundary.  To see this, consider the product CW structure on $B$ with one cell
  for each of the spheres, and fix a lift of the $n$-skeleton to $X$.  Then the
  generator of $\pi_9(X)$ is homotopic to a lift of the attaching map of the
  10-cell; in $\tilde B$, this attaching map induces the chain $(i+1) \cdot a-
  i \cdot a$, where $a$ is the top-dimensional cell of $(S^3)^3$ and $i \in
  \mathbb{Z} \cong \pi_1(X)$.  By Lemma \ref{lem:BBFS}, we can find an
  admissible map $f:S^9 \to \tilde B$ with $f_\#([S^9])=k \cdot [(S^3)^3]-
  [(S^3)^3]$ and thus volume 2. This map has filling class $\Fill(f)=k[B]$ and
  hence lifts to $k$ times the generator of $\pi_9(\tilde X)$.

  This space $X$ may also be thought of as a bundle over $S^1$ with fiber
  $(S^3)^3 \times S^9$.  In this formulation it has nontrivial homological
  monodromy: $H_9((S^3)^3 \times S^9;\mathbb{Q})=\mathbb{Q}^2$ and the generator
  $1 \in \pi_1(X)$ acts on it via the matrix
  $\begin{pmatrix}1&0\\1&1\end{pmatrix}$.  Thus as a
  $\mathbb{Q}[\mathbb{Z}]$-module, $H_9((S^3)^3 \times S^9;\mathbb{Q}) \cong
  \mathbb{Q}[\mathbb{Z}]/([1]-1)^2$, which does not decompose as a direct sum.
  In other words, condition (3) of the theorem is also not satisfied.  This also
  offers another way of understanding why the lift of the chain
  $k \cdot [(S^3)^3]-[(S^3)^3]$ gives a nontrivial element of $\pi_9(\tilde X)$.
\end{ex}
\begin{proof}
  We will repeatedly use the fact that since the almost Postnikov pair $X
  \xrightarrow{p} B$ is normal,
  $$H_n(\hat F;\mathbb{Q})=H_n(F;\mathbb{Q}) \oplus H_n(K(V_p,n);\mathbb{Q}).$$

  We first show that (1) implies (2), by demonstrating that volume distortion
  over $S^1$ is always weakly infinite when it is present.  We do this by
  constructing a sequence of chains with small boundary whose pairing with the
  Euler class grows without bound.  One example of such a chain is $\Fill(f)$ in
  Example \ref{ex:S1}: there, adjacent lifts of the same cell are simply strung
  together with their boundaries canceling out.  In the presence of nontrivial
  monodromy, we also try to string together cells in different lifts to form a
  tube, but the construction turns out to be much more complicated.

  Let $U$ be a finite complex with $\pi_1(U)=\mathbb{Z}$ such that $\tilde U$ is
  rationally $n$-equivalent to a finite complex.  Suppose that $V
  \xrightarrow{p} U$ is a normal almost Postnikov pair with monodromy
  $\rho:\mathbb{Z} \to GL(V_p)$ and Euler class
  $0 \neq \eu \in H^{n+1}(U;M_\rho)$.

  Without loss of generality, by taking a factor, we may assume $M_\rho \cong
  \mathbb{Q}[\mathbb{Z}]/I$ with $I$ a primary ideal.  Indeed, since $\rho$ is
  elliptic, $I$ then must be a prime ideal, and thus $M_\rho$ is a simple
  module.  Let $q(x)=\sum_j q_jx^j$ be the polynomial generating $I$.

  Since $\mathbb{Q}[\mathbb{Z}]$ is a PID, the universal coefficient theorem
  gives a short exact sequence
  $$0 \to \Ext_{\mathbb{Q}[\mathbb{Z}]}(H_n(\tilde U;\mathbb{Q}),M_\rho) \to
  H^{n+1}(U;M_\rho) \to
  \Hom_{\mathbb{Q}[\mathbb{Z}]}(H_{n+1}(\tilde U;\mathbb{Q}),M_\rho) \to 0.$$
  By assumption, the Euler class pulls back to $0 \in H^{n+1}(\tilde U;
  \mathbb{Q})$, and thus it has a preimage in $\Ext_{\mathbb{Q}[\mathbb{Z}]}(
  H_n(\tilde U;\mathbb{Q}),M_\rho)$.  Note that $H_n(\tilde U;\mathbb{Q})=
  Z_n(\tilde U;\mathbb{Q})/B_n(\tilde U;\mathbb{Q})$, both of which are free
  modules; thus one can find bases $z_1,\ldots,z_m$ for $Z_n$ and $b_1,\ldots,
  b_m$ for $B_n$ such that
  $$H_n(\tilde U;\mathbb{Q})=\bigoplus_{i=1}^m \mathbb{Q}[\mathbb{Z}]z_i/
  \mathbb{Q}[\mathbb{Z}]b_i,$$
  where each of the summands is indecomposable.  Then
  $\Ext_{\mathbb{Q}[\mathbb{Z}]}(H_n(\tilde U;\mathbb{Q}),M_\rho)$ takes the form
  $$\Ext_{\mathbb{Q}[\mathbb{Z}]}(H_n(\tilde U;\mathbb{Q}),M_\rho)=
  \bigoplus_{i=1}^m \Hom(\mathbb{Q}[\mathbb{Z}]b_i,M_\rho)/
  \Hom(\mathbb{Q}[\mathbb{Z}]z_i,M_\rho) \cong \bigoplus_{i=1}^m 0\text{ or }
  M_\rho/M_i,$$
  where $M_i \subseteq M_\rho$ is the set of possible images of $b_i$ under a
  homomorphism $\mathbb{Q}[\mathbb{Z}]z_i \to M_\rho$.  Since $M_\rho$ is simple,
  each $M_i$ is either $0$ or $M_\rho$.  If $\eu \neq 0$, then it takes a
  nonzero value in one of these summands.  That is, there is some $i$ such that
  $b_i \in Iz_i \subset \mathbb{Q}[\mathbb{Z}]z_i$, in particular there is a $t$
  such that $\mathbb{Q}[\mathbb{Z}]z_i/\mathbb{Q}[\mathbb{Z}]b_i \cong
  \mathbb{Q}[\mathbb{Z}]/I^t$; and there is a nonzero $\mu \in M_\rho$ such that
  any representative of $\eu$ takes any chain $c \in C_{n+1}(\tilde U;
  \mathbb{Q})$ with $\partial c=b_i$ to $\mu$.

  Let $t$ be the least positive number such that $q([1])^tz_i$ is a boundary; we
  may set $b_i=q([1])^tz_i$.  Let $z=q([1])^{t-1}z_i$.  Recall that $M_\rho$ is
  $r$-dimensional.  Then $z, z \cdot [1],\ldots,z \cdot [r-1]$ descends to a
  basis for $\mathbb{Q}[\mathbb{Z}]z/\mathbb{Q}[\mathbb{Z}]b_i \subseteq
  \mathbb{Q}[\mathbb{Z}]z_i/\mathbb{Q}[\mathbb{Z}]b_i$.  Let $S$ be the vector
  subspace of $Z_n(\tilde U;\mathbb{Q})$ generated by
  $$\vec e_1=z,\vec e_2=z \cdot [1],\ldots,\vec e_r= z \cdot [r-1],$$
  and let $A_q:S \to S$ act on this basis via the companion matrix of $q$.  Thus
  mod $\mathbb{Q}[\mathbb{Z}]b_i$, the action of $A_q$ is the same as that of
  multiplication by $[1]$, and $A_q$ is conjugate to $\rho([-1])$.  Define a
  projection $T:S[\mathbb{Z}] \to \mathbb{Q}[\mathbb{Z}]$ sending
  $\vec e_j [j^\prime] \mapsto [j+j^\prime]$.

  Now, let $c \in C_{n+1}(\tilde U;\mathbb{Q})$ be a particular chain with
  $\partial c=b_i$, and construct a sequence of chains $(c_s)_{s \in \mathbb{N}}
  \in C_{n+1}(\tilde U;\mathbb{Q})$ by setting
  $$c_{s+1}=\sum_{j=0}^s T(A_q^j(\vec e_1)[s-j])c.$$
  For every $j$, $T(A_q^j(\vec e_1)[s-j]) \in I+[s-1]$, and thus
  $\langle\eu, c_s\rangle=s\rho([1-s])(\mu)$.  On the other hand,
  $\partial c_{s+1}$ has bounded volume.  To see this, note that for any $s \geq
  r$, and for some $\vec u_j, \vec v_j \in \mathbb{Q}^r$ not depending on $s$,
  \begin{align*}
    \partial c_{s+1} &= \sum_{j=0}^s T(A_q^j(\vec e_1)[s-j])b_i=\sum_{j=0}^s
    \sum_{\ell=0}^{\deg q} q_\ell T(A_q^j(\vec e_1)[s-j+\ell])z \\
    &= \sum_{j=0}^{s-r} T(q(A_q)A_q^j(\vec e_1)[s-j])z+\sum_{j=0}^{r-1}
    T(A_q^s(\vec u_j)[j])z+\sum_{j=0}^{r-1} T(\vec v_j[s+j+1])z \\
    & =: E(s+1)+F(s+1)+G(s+1),
  \end{align*}
  where $E(s)$, $F(s)$, and $G(s)$ are all cycles, but not necessarily
  boundaries, in $\tilde U$.

  Now we analyze each of these cycles separately.  For each $s>r$, $E(s)=0$
  since by definition $q(A_q)=0$, and $G(s)$ has constant volume.  Moreover,
  since $A_q$ is elliptic,
  $$F(s) \in \left\langle T(z \cdot [j])[j^\prime]: 0 \leq j, j^\prime<\deg q
  \right\rangle$$
  is a vector of bounded norm in a finite-dimensional vector subspace of
  $Z_n(\tilde U;\mathbb{Q})$.  Thus for any $\omega \in C^n(\tilde U;
  \mathbb{Q}^r)$ such that $d\omega$ is a representative of $\eu$,
  $\langle\omega, F(s)\rangle$ is bounded as a function of $s$, but
  $\langle\omega, F(s)+G(s)\rangle=\langle\eu,c_s\rangle=s\rho([1-s])(\mu)$.
  Also, there is a constant $K$ such that every $KG(s)$ is integral.

  So $(KG(s))_{s \in \mathbb{N}}$ is a sequence of integral $n$-cycles of bounded
  volume in $\tilde U$ which lift to elements of $H_n(\tilde V;\mathbb{Q})$
  which have norm increasing linearly in $s$.  This sequence is a bounded
  distance away from the sequence $(KF(s)+KG(s))_{s \in \mathbb{N}}$ whose terms
  are boundaries and therefore lift to $h_n\iota_*\pi_n((S^n)^r)$.  Thus
  $\iota_*\pi_n((S^n)^r)$ is weakly infinitely distorted.

  ($2^\prime$) is equivalent to (2) more or less by definition.  Similarly, it
  follows from the definition of the spectral sequence in Example
  \ref{exs:AmHy} that ($2^{\prime\prime}$) is a restatement of ($2^\prime$).

  ($2^\prime$)$\Rightarrow$(3).  Write $\widetilde{\iota_1^*X}$ and
  $\widetilde{\iota_1^*B}$ for $\hat p^{-1}(A^{(1)})$ and
  $(p^\prime)^{-1}(A^{(1)})$ respectively, and $F_m$ for the finitely generated
  free group $\pi_1(A^{(1)})$.  Finally, write $\iota_1^*V_p$ for $V_p$ seen as a
  $\mathbb{Q}F_m$ module with the action $\gamma \cdot v=\iota_{1*}\gamma \cdot
  v$ for $\gamma \in F_m$.  Then it suffices to show that when ($2^\prime$)
  holds, the short exact sequence of $\mathbb{Q}F_m$-modules
  $$0 \to \iota_1^*V_p \xrightarrow{h_n} H_n(\widetilde{\iota_1^*X};\mathbb{Q})
  \to H_n(\widetilde{\iota_1^*B};\mathbb{Q}) \to 0$$
  splits, i.e.\ that there is a module homomorphism
  $i:H_n(\widetilde{\iota_1^*B};\mathbb{Q}) \to H_n(\widetilde{\iota_1^*X};
  \mathbb{Q})$ such that $p_* \circ i=\id$.  To prove (3) from this, we tensor
  this sequence with $\mathbb{Q}\Gamma$ as a $\mathbb{Q}F_m$ module to obtain
  the sequence
  $$0 \to V_p \to H_n(\tilde X;\mathbb{Q}) \to H_n(\tilde B;\mathbb{Q}) \to 0,$$
  together with a splitting.

  Suppose $\iota_1^*\eu=0$.  Perhaps after taking a mapping cylinder so that
  $\iota_1^*X \subset \iota_1^*B$, we can choose a lifting homomorphism
  $j_*:C_{\leq n}(\widetilde{\iota_1^*B};\mathbb{Q}) \to
  C_{\leq n}(\widetilde{\iota_1^*X};\mathbb{Q})$, with a corresponding chain
  homotopy $u_*$.  In particular, $j_n$ takes cycles to cycles.  As before,
  $$\id+u_n\partial:C_{n+1}(\widetilde{\iota_1^*B};\mathbb{Q}) \to
  H_{n+1}(\widetilde{\iota_1^*B},\widetilde{\iota_1^*X};\mathbb{Q})
  \twoheadrightarrow \iota_1^*V_p$$
  defines a representative $v \in C^{n+1}(\widetilde{\iota_1^*B},\iota_1^*V_p)$
  of the Euler class.  We know that $v=de$ for some $n$-cochain $e$.  Then the
  chain map $j_n-h_n \circ e:C_n(\widetilde{\iota_1^*B};\mathbb{Q}) \to
  C_n(\widetilde{\iota_1^*X};\mathbb{Q})$ induces a splitting homomorphism
  $H_n(\widetilde{\iota_1^*B};\mathbb{Q}) \to
  H_n(\widetilde{\iota_1^*X};\mathbb{Q})$ as desired.

  A vector subspace of $H_n(\tilde X;\mathbb{Q})$ is a
  $\mathbb{Q}\Gamma$-submodule if and only if it is $\hat\rho$-invariant, and so
  ($3^\prime$) is simply a restatement of (3).

  ($3^\prime$)$\Rightarrow$(1).  Suppose ($3^\prime$) holds, and equip
  $H_n(\tilde X;\mathbb{Q}) \cong H_n(\hat F;\mathbb{Q})$ with the seminorm
  which sends $P$ to 0 and restricts to a $\rho$-invariant norm on
  $V_p$.  Note that $\hat\rho$ preserves this seminorm.  Similarly to the proof
  in \cite{AWP} that higher-order Dehn functions of groups are finite, we will
  show that when ($3^\prime$) holds, the homology classes of integral cycles of
  volume $k$ in $\tilde X$ have seminorm bounded by some $C(k)$.  This is enough
  to show (1).

  The proof is by induction on $k$.  Clearly there is a $C(1)$, since there's a
  finite number of $\Gamma$-equivalence classes of cycles of volume 1.  Now
  suppose we have determined $C(i)$ for $1 \leq i<k$.  Given an $n$-cycle $c$ in
  $\tilde X$ of volume $k$, assume that it contains a cell from a fundamental
  domain $D$.  Since the action of $\Gamma$ preserves the seminorm, we can do
  this without loss of generality.  Then either $c$ consists entirely of cells
  within distance $k$ of $D$ (in the graph defined by adjacency of cells, in the
  sense of having common $(n-1)$-cells in their boundaries) or it consists of
  two disjoint cycles.  There are a finite number of cycles of the first kind,
  since there are a finite number of such cells, and so the seminorm of their
  homology classes is bounded by some $B(k)$.  Thus we can set
  $C(k)=\max\{B(k)\} \cup \{C(i)+C(k-i):0<i<k\}$.
\end{proof}
\begin{proof}[Proof of Theorem \ref{thmB}]
  We have already shown that if a finite CW complex $X$ has no infinite
  distortion in its homotopy groups, then conditions (1) and (2) of Theorem
  \ref{thmB} are satisfied.  In particular, we can find a tower of maps
  $$X=X_m \to X_{m-1} \to \cdots \to X_0$$
  where each $X_i$ is a finite complex which is rationally equivalent up to a
  large dimension $N$ to $X_{(n_i)}$, and where each step $X_k \to X_{k-1}$ is a
  normal almost Postnikov pair.  By Theorem \ref{thm:inf}, $\pi_{n_k}(X_k)$ does
  not have weakly infinite distortion if and only if
  $H_{n_k}(\tilde X_k;\mathbb{Q})$ splits as a $\mathbb{Q}\pi_1X$-module into the
  image of the Hurewicz map and its complement.  Since $X_k$ is rationally
  equivalent to $X$ up to dimension $k+1$, this is also the case for
  $\pi_{n_k}(X)$.  Thus if $X$ is a delicate space, then the lack of weakly
  infinite distortion in its homotopy groups is equivalent to condition (3) of
  Theorem \ref{thmB}.
\end{proof}

Note, however, that weakly infinite distortion inside the homology group does
not imply that $\pi_n$ is actually infinitely distorted.  As an example, it is
sufficient to construct a space $X$ whose homological monodromy behaves as in
Example \ref{exs:oddD}(\ref{qqu}).  That is, we want $\pi_1X=\mathbb{Z}$ and
$M:=H_n(\tilde X;\mathbb{Q})$ to be a module which is isomorphic to
$\mathbb{Q}^4$ as a vector space, with the fundamental group acting by the
matrix $\hat A=\begin{pmatrix}A&0\\I&A\end{pmatrix}$, where $A=
\begin{pmatrix}3/5&-4/5\\4/5&3/5\end{pmatrix}$.  Moreover, we want the homotopy
submodule $h_n\pi_n(\tilde X) \otimes \mathbb{Q}$ to be generated by the last
two basis vectors.  Then the argument of Example \ref{exs:oddD}(\ref{qqu}) will
show that $\pi_n(X)$ is weakly infinitely but not infinitely distorted.
Actually, it is enough to find a space $X$ such that $H_n(\tilde X;\mathbb{Q})$
has a copy of $M$ containing the homotopy submodule as a direct summand.

We build such a space $X$ as follows.  Let $B$ be the ``mapping torus'' (in the
sense of Section 3) of $(S^3)^4$ corresponding to the matrix $\tilde A=
\begin{pmatrix}A^{-1}&0\\0&I\end{pmatrix}$ acting on the four spheres.
Then the monodromy of $H_9(\tilde B;\mathbb{Q}) \cong \bigwedge^3 H_3(\tilde B;
\mathbb{Q})$ with respect to the action of $\pi_1$ is given by
$\tilde A^{\wedge 3}$; a computation shows that in the basis given by the
Poincar\'e duals of the four spheres, $\tilde A^{\wedge 3}=
\begin{pmatrix}A&0\\0&I\end{pmatrix}$.  So let $X$ be the total space of a
rational $(S^9)^2$-bundle over $B$ which has the monodromy representation
$\rho:\mathbb{Z} \to GL(2,\mathbb{Q})$ which takes the generator to $A$, and
Euler class
$$0 \neq \eu \in H^{10}(B;M_\rho) \cong \Ext_{\mathbb{Q}[\mathbb{Z}]}(M_\rho \oplus
\mathbb{Q} \oplus \mathbb{Q},M_\rho) \cong M_\rho.$$
By the argument in Example \ref{ex:S1}, $H_n(\tilde X;\mathbb{Q}) \cong
M \oplus \mathbb{Q} \oplus \mathbb{Q}$.

On the other hand, suppose that we take an almost Postnikov pair $X \to B$ with
$\pi_1X=\pi_1B=\Gamma$, with monodromy reprsentation $\rho$ and Euler class
$\eu$, such that for some $\gamma \in \Gamma$, $\gamma^*\eu$ is nontrivial when
restricted to a direct summand $\mathbb{Q} \cong L \subseteq M_{\gamma^*\rho}$ (a
module over $\mathbb{Q}[\mathbb{Z}]$) on which the action of $\gamma$ is
trivial.  Then it is in fact possible to find true infinite distortion via the
method of Example \ref{ex:S1}.

In the case when $\rho$ is trivial, this is equivalent to condition (2) of
Theorem \ref{thm:inf} and hence is a necessary condition for finding infinite
distortion.  We conjecture that this is the case in general.
\begin{conj}
  The condition stated above is equivalent to the presence of infinite
  distortion in $\pi_n(X)$.
\end{conj}
The following theorem shows that this is true in the case $\pi_1(X)=\mathbb{Z}$.
In other words, a counterexample to this conjecture must have infinite
distortion which is in some sense ``non-local'', that is, it would have to rely
on the interaction of weakly infinite distortions induced by the homological
monodromy of multiple elements of $\pi_1(X)$.  However, because the algebra and
geometry of a general fundamental group can make themselves relevant, this
realization still leaves us quite far from a proof of the conjecture.
\begin{thm} \label{thm:infcyc}
  Suppose $V \xrightarrow{p} U$ is an almost Postnikov pair such
  that $U$ is the total space of a rational homotopy fibration $F \to U \to
  S^1$.  Let $\rho:\mathbb{Z} \to GL(V_p)$ be its elliptic monodromy
  and $\eu \in H^{n+1}(U;M_\rho)$ be the Euler class.  Then the following are
  equivalent:
  \begin{enumerate}
  \item There is infinite volume distortion in $V_p$ as a subset of
    $H_n(\tilde V;\mathbb{Q})$.
  \item There is a finite-sheeted cover $\ph:\hat U \to U$ such that
    $\ph^*M_\rho$ has a trivial submodule $L \cong \mathbb{Q}$ for which the
    projection of $\ph^*\eu \in H^{n+1}(\hat U;\ph^*M_\rho)$ onto $L$ is nonzero.
  \end{enumerate}
\end{thm}
\begin{proof}
  To show that (2)$\Rightarrow$(1), note that $H_{n+1}(\hat U;\ph^*M_\rho)$ has a
  submodule of the form in Example \ref{ex:S1}.  Therefore, we can construct a
  sequence of $(n+1)$-chains as in that example.

  Now suppose that (2) is not true, and choose a simple direct summand $M$ of
  $M_\rho$ with projection map $p_M:M_\rho \to M$ and dimension $d=\dim M$.
  Since any elliptic element of $GL(m,\mathbb{Z})$ is of finite order, the
  transformation $p_M \circ \rho(1)$ does not conjugate into $GL(m,\mathbb{Z})$.
  Let $A$ be the $d \times d$ matrix of $p_M \circ \rho(1)$ in some basis.

  To show that $M$ is not infinitely distorted, we will reproduce in greater
  generality the argument in the example given above.  Let $\hat\rho:\mathbb{Z}
  \to GL(H_n(\hat F;\mathbb{Q}))$ be the homological monodromy of the total
  fibration $\hat F \to V \to S^1$, and let $\hat M$ be the smallest direct
  summand of the module $H_n(\hat F;\mathbb{Q})$ which contains $M$.  Then we
  can find a basis for $\hat M$ such that the action of $\rho(1)$ on $\hat M$ is
  given in this basis by a matrix which consists of generalized Jordan blocks of
  the form
  $$\hat A_i=\begin{pmatrix}
    A&0&\cdots&0\\I&A&\ddots&\vdots\\\vdots&\ddots&\ddots&0\\0&\cdots&I&A
  \end{pmatrix},$$
  each acting on an $r_id$-dimensional submodule $\hat M_i$.  In this
  formulation, we can write $M=\sum_i t_iM_i$, where the $t_i$ are fixed
  constants and the $M_i$ are the modules isomorphic to $M$ generated by the
  last $d$ coordinates of each $\hat M_i$.  It is enough for our purposes to
  show that each $M_i$ is only finitely distorted inside $\hat M_i$; the next
  two paragraphs make this idea more precise.

  Perhaps after taking a mapping cylinder so that $V \subset U$, we can choose a
  lifting homomorphism $j_*:C_{\leq n}(\tilde U;\mathbb{Q}) \to
  C_{\leq n}(\tilde V;\mathbb{Q})$.  In particular, $j_n$ takes cycles to cycles
  and induces a surjective homomorphism $J:Z_n(\tilde U;\mathbb{Q}) \to \hat M
  \subset H_n(\tilde V;\mathbb{Q})$.  Note also that $J$ takes boundaries into
  $M$.

  Consider a fundamental domain $B \subset \tilde U$ large enough that chains
  contained in $B$ generate $C_n(\tilde U;\mathbb{Q})$ as a
  $\mathbb{Q}[\mathbb{Z}]$-module.  Fix a norm $\lVert\cdot\rVert$ on $M$ which
  is invariant under the action of $\pi_1U$.  We want to show that $M \subseteq
  \pi_n(V)$ is not infinitely distorted, i.e.\ that there is a function
  $L:\mathbb{N} \to \mathbb{N}$ such that for every $k$ and every integral
  boundary $b \in B_n(\tilde U)$ with $\vol b \leq k$, $\lVert J(b) \rVert \leq
  L(k)$.  In fact, it is enough to show that for each $i$, vectors of the form
  $$\vec u=\hat A^{t_0}p_iJ(z_0)+\hat A^{t_1}p_iJ(z_1)+\cdots+\hat A^{t_k}p_iJ(z_k)
  \in M_i \subseteq \hat M_i,$$
  where $p_i:H_n(\tilde V;\mathbb{Q}) \to \hat M_i$ is a projection map and each
  $z_j \in C_n(B)$ is an integral cycle of size at most $k$, have norm bounded
  as a function of $k$.

  This last condition implies that the vectors $\vec u_j=p_iJ(z_j)$ are of size
  linear in $k$ in some lattice $\hat\Lambda \subset \hat M_i$, and so this is
  now purely a linear algebra problem similar, but not identical, to Lemma
  \ref{lem:oddD}.  We solve it by induction on $k$ and the number $r$ of Jordan
  blocks in $\hat A_i$.

  When $r=1$, then $\hat A_i=A$ preserves the norm, and so $\lVert\vec u\rVert
  \leq Ck^2$ for some constant $C$.

  When $k=0$, then $\vec u=\hat A^{t_0}\vec u_0$ and $\vec u_0 \in M_i$ is one of
  a finite number of vectors, and once again $\hat A_i$ preserves the norm of
  such a vector, so $\lVert\vec u\rVert \leq C$ for some constant $C$.  This
  completes the base case.

  Now fix $r \geq 2$ and $k \geq 2$, and suppose we have a bound $L_r(k^\prime)$
  for every $k^\prime<k$.  We will show that there is a bound $L_r(k)$.  Note
  that we can assume that the $t_j$ are increasing and $t_0=0$, since
  multiplying by $\hat A_i^t$ doesn't change the length of a vector in $M_i$.
  Moreover, since the $M_i$-coordinates of the $\vec u_j$ only change
  $\lVert \vec u \rVert$ by a linear amount, we can assume that they are zero.

  Now, let $p_{r-1}:\hat M_i \to \hat M_i$ be the projection onto the first
  $(r-1)d$ coordinates.  By assumption, $p_{r-1}\vec u=\vec 0$, and therefore
  either (1) $\sum_{i=0}^\ell \hat A_i^{t_j} \vec u_j \neq \vec 0$ for every
  $\ell<k$, or (2) we know that $\lVert u\rVert \leq L_r(k_1)+L_r(k_2)$ for some
  $k_1+k_2=k$.
  \begin{lem}
    In case (1), for every $1 \leq \ell \leq k$, $t_k \leq kf_r(k)$,
    where $f_r(k) \sim \log k$ is a function depending on $A$ and $r$.
  \end{lem}
  \begin{proof}
    Given $\ell$, let $1 \leq s \leq r-1$ be the index of the first
    $d$-dimensional block such that the coordinates of
    $$\sum_{j=0}^{\ell-1} p_{r-1}\hat A_i^{t_j}\vec u_j=-\sum_{j=\ell}^k
    p_{r-1}\hat A_i^{t_j}\vec u_j$$
    in that block are nonzero, and let $\vec v \in \mathbb{Q}^d$ be the
    coordinates of this vector in the $s$th block.  Then we can write
    $$\vec v=\sum_{j=0}^{\ell-1} \sum_{a=0}^{s-1} P_{a,s}(t_j)A^{t_j-a})\vec u_{j,a}
    \cap \sum_{j=\ell}^k \sum_{a=0}^{s-1} P_{a,s}(t_j)A^{t_j-a}\vec u_{j,a},$$
    where $\vec u_{j,a}$ is the projection of $\vec u_j$ to the $a$th block and
    $P_{a,s}$ is a polynomial of degree $s-a$.  In particular,
    $\lVert\vec u_{j,a}\rVert \leq k$ and $\vec u_{j,a} \in \Lambda$, the lattice
    generated by the projections of $\hat\Lambda$ to each block.  Applying Lemma
    \ref{lem:NT} and maximizing over $s$, we see that for large $k$,
    $$t_k \leq k\big[f\big(|P_{1,r-1}(t_k)| \cdot (r-1)k\big)+r-2\big]$$
    where $f$ is logarithmic.  This forces $t_k \lesssim k\log k$.
  \end{proof}
  In case (1), this means that there is a finite number of choices for the
  vector $\vec u$, and so they are bounded by some $L_r^\prime(k)$.  Thus we can
  set
  $$L_r(k)=\max\{L_r^\prime(k),L_r(k-1)+L_r(1),L_r(k-2)+L_r(2),\ldots\}.$$

  Having shown this for every factor $M$, we see that (1) does not hold.
\end{proof}

\appendix
\section{A number-theoretic lemma}
The following lemma is crucial for proving that certain kinds of twisting
preclude distortion from monodromy from being infinite.  Because it is somewhat
technical and requires tools which have little to do with the rest of the paper,
we have sequestered it here.
\begin{lem} \label{lem:NT}
  Let $A$ be an irreducible $n \times n$ matrix over $\mathbb{Q}$ which is of
  infinite order, diagonalizable over $\mathbb{C}$, and all of whose eigenvalues
  are on the unit circle, and let $\Lambda \subset \mathbb{Q}^n$ be a lattice.
  Then there is a function $f(V)=a+b\log V$, where $a$ and $b$ depend on $A$ and
  $\Lambda$, such that if for some $k$ and $\ell$ and vectors $\vec u_i \in
  \Lambda$ with $\sum_{i=1}^\ell \lVert\vec u_i\rVert \leq V$,
  $$\vec v_1+\vec v_2:=\sum_{i=0}^k A^i\vec u_i+\sum_{i=k+m}^\ell A^i\vec u_i=
  \vec 0,$$
  then either $m \leq f(V)$ or $\vec v_1=\vec v_2=\vec 0$.
\end{lem}
We start with the following observation.  Consider the matrix $A=
\begin{pmatrix}3/5&-4/5\\4/5&3/5\end{pmatrix}$, which represents multiplication
in $\mathbb{C}$ by $\frac{3+4i}{5}$.  Then if $P$ is an integer polynomial, we
can think of $P(A)\vec v$ as the complex number $P\left(\frac{3+4i}{5}\right)
\cdot (v_1+v_2i)$.  Then
$$A\begin{pmatrix}2\\-1\end{pmatrix}=\begin{pmatrix}2\\1\end{pmatrix};$$
this is directly related to the fact that this eigenvalue can be expressed as
$$\frac{3+4i}{5}=\frac{2+i}{2-i},$$
which is a ratio of relatively prime Gaussian integers.  Because the Gaussian
integers are a UFD, if $P$ has only high-degree terms but
$P\left(\frac{2+i}{2-i}\right)$ is a Gaussian integer, then coefficients of $P$
must contain large powers of $2-i$, forcing them to be large in absolute value.

The proof which follows generalizes this observation.
\begin{proof}
  We can find a single vector $\vec w \in \mathbb{Q}^n$ such that the lattice
  generated by $\vec w, A\vec w,\ldots,A^{n-1}\vec w$ contains $\Lambda$.  In
  other words, every vector $\lambda \in \Lambda$ can be expressed as
  $P_\lambda(A)\vec w$ for some integer polynomial $P_\lambda$ of degree $n-1$.
  Therefore we can reexpress the given sum as $(P_1(A)+P_2(A))\vec w$, where
  $P_1$ and $P_2$ are integer polynomials whose nonzero terms $p_i$ are in
  degrees $0$ through $k+n-1$ and degrees $k+m$ through $\ell$, respectively.
  Moreover, since $A$ is irreducible and $P_1(A)+P_2(A)$ does not have full
  rank, $P_1(A)+P_2(A)=0$, and therefore, as polynomials, $P_1(x)+P_2(x)=Q(x)
  \chi(x)$ where $\chi$ is the characteristic polynomial of $A$.  We will show
  that if $m$ is large enough, then $Q$ must break into two polynomials $Q_1$
  and $Q_2$ such that $Q_1\chi=P_1$ and $Q_2\chi=P_2$.  This then implies that
  $P_1(A)=P_2(A)=0$.

  Note that since $A$ is of infinite order, the roots of $\chi$ are not
  algebraic integers; let $r \in \mathbb{C}$ be one such root.  Let $K$ be the
  splitting field of $\chi$ and $\mathcal{O}_K$ its ring of integers.  We can
  express $r$ as the ratio of algebraic integers; suppose first that we can make
  these relatively prime, $r=p/q$.  Then we have
  $$P_1(p/q) \cdot q^{k+n-1}=-P_2(p/q) \cdot q^{k+n-1}.$$
  Here, the term on the left is evidently an algebraic integer.  Because $p$ and
  $q$ are relatively prime, the right side demonstrates that this integer must
  be divisible by $p^{k+m}$.  Therefore, either $P_1(p/q)=0$ or
  $$|P_1(p/q)|\cdot|q|^{k+n-1} \geq |p|^{k+f(V)}.$$
  Since $|p|=|q|$, this would mean that $V \geq |p|^{m-n+1}$.  So whenever $m>n-
  1+\log_p V$, $P_1$ must be divisible by $\chi$.  In other words, we can set
  $$f(V)=n-1+\log_p V.$$

  Since $\mathcal{O}_K$ may not have unique factorization, we may not be able to
  set $r=p/q$ with the $p$ and $q$ relatively prime.  However, the ideal class
  group of $\mathcal{O}_K$ is finite, meaning that every ideal has a power which
  is principal.  Therefore, if we express $r$ as the ratio of algebraic integers
  $p/q$, there is some $t$ such that the ideals $(p^t)$ and $(q^t)$ are products
  of principal primary ideals.  This means that we can express $r^t$ as the
  quotient of relatively prime algebraic integers $p^\prime/q^\prime$.  Then we
  can repeat the argument above to get
  $$f(V)=t(n-1+\log_{p^\prime} V),$$
  completing the proof.
\end{proof}

\bibliographystyle{amsalpha}
\bibliography{distortion}

\newcommand{\etalchar}[1]{$^{#1}$}
\providecommand{\bysame}{\leavevmode\hbox to3em{\hrulefill}\thinspace}
\providecommand{\MR}{\relax\ifhmode\unskip\space\fi MR }
\providecommand{\MRhref}[2]{%
  \href{http://www.ams.org/mathscinet-getitem?mr=#1}{#2}
}
\providecommand{\href}[2]{#2}
\begin{thebibliography}{Gut08b}

\bibitem[ABl]{AB}
Oliver Attie and Jonathan Block, \emph{Poincar\'e duality for {$L^p$}
  cohomology and characteristic classes}, Tech. Report 98-48, DIMACS, 1998,
  Available at
  \url{http://dimacs.rutgers.edu/TechnicalReports/TechReports/1998/98-48.ps.gz}.

\bibitem[ABlW]{ABW}
Oliver Attie, Jonathan Block, and Shmuel Weinberger, \emph{Characteristic
  classes and distortion of diffeomorphisms}, Journal of the American
  Mathematical Society \textbf{5} (1992), no.~4, 919--921.

\bibitem[AWP]{AWP}
J.M. Alonso, X.~Wang, and S.J. Pride, \emph{Higher-dimensional isoperimetric
  (or {D}ehn) functions of groups}, Journal of Group Theory \textbf{2} (1999),
  81--112.

\bibitem[BB]{BeBr}
Mladen Bestvina and Noel Brady, \emph{Morse theory and finiteness properties of
  groups}, Inventiones mathematicae \textbf{129} (1997), no.~3, 445--470.

\bibitem[BBFS]{BBFS}
Noel Brady, Martin~R. Bridson, Max Forester, and Krishnan Shankar,
  \emph{Snowflake groups, {P}erron-{F}robenius eigenvalues and isoperimetric
  spectra}, Geometry and Topology \textbf{13} (2009), no.~1, 141--187.

\bibitem[BrG]{BG_F}
Kenneth~S. Brown and Ross Geoghegan, \emph{An infinite-dimensional torsion-free
  {$\mathbf{FP}_\infty$} group}, Inventiones mathematicae \textbf{77} (1984),
  no.~2, 367--381.

\bibitem[BRS]{BRS}
Sandro Buoncristiano, Colin~Patrick Rourke, and Brian~Joseph Sanderson, \emph{A
  geometric approach to homology theory}, vol.~18, Cambridge University Press,
  1976.

\bibitem[BlW]{BWJAMS}
Jonathan Block and Shmuel Weinberger, \emph{Aperiodic tilings, positive scalar
  curvature, and amenability of spaces}, Journal of the American Mathematical
  Society \textbf{5} (1992), no.~4, 907--918.

\bibitem[CHM]{VHM}
{\'E}ric Colin~de Verdi{\`e}re, Alfredo Hubard, and Arnaud de~Mesmay,
  \emph{Discrete systolic inequalities and decompositions of triangulated
  surfaces}, Proceedings of the thirtieth annual symposium on Computational
  geometry, ACM, 2014, p.~335.

\bibitem[ChGr]{CheeGr}
Jeff Cheeger and Mikhail Gromov, \emph{{$L_2$}-cohomology and group
  cohomology}, Topology \textbf{25} (1986), no.~2, 189--215.

\bibitem[CLRS]{CLRS}
Thomas~H Cormen, Charles~E Leiserson, Ronald~L Rivest, Clifford Stein, et~al.,
  \emph{Introduction to algorithms}, MIT Press, 2001.

\bibitem[DGS]{Gluck}
Dennis DeTurck, Herman Gluck, and Peter Storm, \emph{Lipschitz minimality of
  {H}opf fibrations and {H}opf vector fields}, Algebraic \& Geometric Topology
  \textbf{13} (2013), no.~3, 1368--1412.

\bibitem[EPC{\etalchar{+}}]{ECHLP}
David Epstein, Michael Paterson, James Cannon, Derek Holt, Silvio Levy, and
  William~P. Thurston, \emph{Word processing in groups}, Jones and Bartlett,
  1992.

\bibitem[FHT]{FHT}
Yves F{\'e}lix, Steve Halperin, and Jean-Claude Thomas, \emph{Rational homotopy
  theory}, Graduate Texts in Mathematics, vol. 205, Springer, 2012.

\bibitem[Ger92]{GerLost2}
S~Gersten, \emph{Bounded cohomology and combings of groups}, preprint (1992),
  available at \url{http://www.math.utah.edu/~sg/Papers/bdd.pdf}.

\bibitem[Ger96]{GerLost}
\bysame, \emph{A cohomological characterization of hyperbolic groups}, preprint
  (1996), available at \url{http://www.math.utah.edu/~sg/Papers/ch.pdf}.

\bibitem[GJ]{GJ}
Paul~G Goerss and John~F Jardine, \emph{Simplicial homotopy theory},
  Birkh\"auser, 2009.

\bibitem[GM]{GrMo}
Phillip~A. Griffiths and John~W. Morgan, \emph{Rational homotopy theory and
  differential forms}, Birkh{\"a}user, 1981.

\bibitem[Gro78]{GrDil}
Mikhail Gromov, \emph{Homotopical effects of dilatation}, Journal of
  Differential Geometry \textbf{13} (1978), no.~3, 303--310.

\bibitem[Gro91]{GrHodge}
\bysame, \emph{K\"ahler hyperbolicity and {$L_2$-H}odge theory}, Journal of
  Differential Geometry \textbf{33} (1991), no.~1, 263--292.

\bibitem[Gro96]{GrAsym}
\bysame, \emph{Geometric group theory, vol. 2: Asymptotic invariants of
  infinite groups}, Bull. Amer. Math. Soc \textbf{33} (1996), 0273--0979.

\bibitem[Gro98]{GrMS}
\bysame, \emph{Metric structures for {R}iemannian and non-{R}iemannian spaces},
  vol. 152, Birkh{\"a}user, 1998.

\bibitem[Gro99]{GrQHT}
\bysame, \emph{Quantitative homotopy theory}, Invited Talks on the Occasion of
  the 250th Anniversary of {P}rinceton {U}niversity (H.~Rossi, ed.), Prospects
  in Mathematics, 1999, pp.~45--49.

\bibitem[Groft]{Groft}
Chad Groft, \emph{Generalized {D}ehn functions {I}}, arXiv preprint
  arXiv:0901.2303 (2009).

\bibitem[Gut08a]{Guth3}
Larry Guth, \emph{Directional isoperimetric inequalities and rational homotopy
  invariants}, arXiv preprint arXiv:0802.3549 (2008).

\bibitem[Gut08b]{Guth2}
\bysame, \emph{Isoperimetric inequalities and rational homotopy invariants},
  arXiv preprint arXiv:0802.3550 (2008).

\bibitem[Gut13]{Guth}
\bysame, \emph{Contraction of areas vs. topology of mappings}, Geometric and
  Functional Analysis \textbf{23} (2013), no.~6, 1804--1902.

\bibitem[Hat]{Hatc}
Allen Hatcher, \emph{Algebraic topology}, Cambridge University Press, 2001.

\bibitem[Hat5]{HatSS}
\bysame, \emph{Algebraic topology}, ch.~5: Spectral {S}equences, 2015,
  available at \url{https://www.math.cornell.edu/~hatcher/AT/ATch5.pdf}.

\bibitem[Katz]{Katz}
Mikhail Katz, \emph{Systolic geometry and topology}, Mathematical surveys and
  monographs, vol. 137, American Mathematical Society, 2007.

\bibitem[KK]{KK}
Beno{\^\i}t Kloeckner and Greg Kuperberg, \emph{The {L}ittle {P}rince and
  {W}eil's isoperimetric problem}, arXiv preprint arXiv:1303.3115 (2013).

\bibitem[M]{thesis}
Fedor Manin, \emph{Asymptotic invariants of homotopy groups}, Ph.D. thesis,
  University of Chicago, June 2015.

\bibitem[Mes]{Mesk}
Stephen Meskin, \emph{Nonresidually finite one-relator groups}, Trans. Amer.
  Math. Soc. \textbf{164} (1972), 105--114.

\bibitem[Min]{MinIP}
Igor Mineyev, \emph{Higher dimensional isoperimetric functions in hyperbolic
  groups}, Mathematische Zeitschrift \textbf{233} (2000), no.~2, 327--345.

\bibitem[N{\v{S}}]{NoSpa}
Piotr~W Nowak and J{\'a}n {\v{S}}pakula, \emph{Controlled coarse homology and
  isoperimetric inequalities}, Journal of Topology \textbf{3} (2010), no.~2,
  443--462.

\bibitem[Spa]{Spa}
Edwin~H. Spanier, \emph{Algebraic topology}, Springer-Verlag, 1981.

\bibitem[Wall]{Wall}
C.T.C. Wall, \emph{Finiteness conditions for {CW}-complexes}, The Annals of
  Mathematics \textbf{81} (1965), no.~1, 56--69.

\bibitem[Wen]{Wen}
Haomin Wen, \emph{Lipschitz minimality of the multiplication maps of unit
  complex, quaternion and octonion numbers}, Algebraic \& Geometric Topology
  \textbf{14} (2014), no.~1, 407--420.

\bibitem[White]{White}
Brian White, \emph{Mappings that minimize area in their homotopy classes},
  Journal of Differential Geometry \textbf{20} (1984), no.~2, 433--446.

\bibitem[Young]{Young}
Robert Young, \emph{Homological and homotopical higher-order filling
  functions}, Groups, Geometry, and Dynamics \textbf{5} (2011), no.~3,
  683--690.

\end{thebibliography}
\end{document}